\documentclass[11pt]{amsart}
\usepackage{amsfonts,amsmath,latexsym,amssymb,verbatim,amsbsy,amsthm}
\usepackage{graphicx}
\usepackage{caption}
\usepackage{float}
\usepackage{wrapfig}
\usepackage[shortlabels]{enumitem}
\usepackage[english]{babel}
\usepackage{mathtools}
\usepackage{array}

\usepackage[top=1in, bottom=1in, left=1in, right=1in]{geometry}
\usepackage[dvipsnames]{xcolor}
\usepackage[colorlinks=true, pdfstartview=FitV, linkcolor=RoyalBlue,citecolor=ForestGreen, urlcolor=blue]{hyperref}
\usepackage{xfrac}

\newcommand{\norm}[1]{\left\lVert#1\right\rVert}

\theoremstyle{plain}
\newtheorem{THEOREM}{Theorem}[section]

\newtheorem{theorem}[THEOREM]{Theorem}

\newtheorem{lemma}[THEOREM]{Lemma}

\theoremstyle{definition}

\theoremstyle{remark}

\newtheorem{remark}[THEOREM]{Remark}

\DeclareMathOperator{\supp}{supp} %
\def \a {\alpha}

\def \g {\gamma}
\def \d {\delta}
\def \k {\kappa}
\def \e {\varepsilon}

\def \l {\lambda}
\def \n {\nabla}
\def \s {\sigma}

\def \D {\Delta}

\def \bb {{\bf b}}

\def \bn {{\bf n}}

\def \bu {{\bf u}}

\def \bx {{\bf x}}

\def \bz {{\bf z}}

\def \bx {{\bf x}}

\def \cL {\mathcal{L}}

\def \cN {\mathcal{N}}

\def \cQ {\mathcal{Q}}
\def \cR {\mathcal{R}}
\def \cS {\mathcal{S}}
\def \cT {\mathcal{T}}

\def \aa {\mathsf{a}}

\def \uu {\mathsf{u}}
\def \vv {\mathsf{v}}
\def \bb {\mathsf{b}}
\def \ff {\mathsf{f}}
\def \hh {\mathsf{h}}
\def \HH {\mathsf{H}}
\def \gg {\mathsf{g}}
\def \pp {\mathsf{p}}

\usepackage{amsmath}

\def \cn {\mathcal{C}_n}
\def \sn {\mathcal{S}_n}

\def \cm {\mathcal{C}_m}
\def \sm {\mathcal{S}_m}

\newcommand{\N}{\ensuremath{\mathbb{N}}}   
\newcommand{\R}{\ensuremath{\mathbb{R}}}   
\newcommand{\T}{\ensuremath{\mathbb{T}}}   

\newcommand{\pa}{\partial}

\def \sign {\mathrm{sgn}}

\def \p {\partial}

\def \dt  {\, \mbox{d}t}

\def \dz  {\, \mbox{d}z}
\def \dr  {\, \mbox{d}r}
\def \ds  {\, \mbox{d}s}

\newcommand{\ep}{\varepsilon}

\usepackage{amsfonts}

\begin{document}

\title[2D EULER: TIME PERIODIC SOLUTIONS NEAR TAYLOR-COUETTE]{Time periodic solutions for the  2D Euler equation\\
  near Taylor-Couette flow}

\author{\'Angel Castro and Daniel Lear} 

\address{Instituto de Ciencias Matematicas, Madrid.}
\email{angel\_castro@icmat.es}

\address{Universidad de Cantabria, Santander.}
\email{daniel.lear@unican.es}

\date{\today}




\maketitle

\begin{abstract}
In this paper we consider the incompressible 2D Euler equation in an annular domain with non-penetration boundary condition. In this setting, we prove the existence of a family of non-trivially smooth time-periodic solutions at an arbitrarily small distance from the stationary Taylor-Couette flow in $H^s$, with $s < \sfrac{3}{2}$, at the vorticity level. 
\end{abstract}

\addtocontents{toc}{\protect\setcounter{tocdepth}{1}}

\tableofcontents

\section{Introduction and main result}
\subsection{Motivation and background}
The 2D Euler equation is an old but still very active field of research. Unlike its 3D version, the global existence of smooth solutions is a classic result dating back to the 1930's, see \cite{H,W}. The qualitative behavior of these global smooth solutions is a very important aspect of the study of Euler flows. However, their long time dynamics is difficult to understand in general due to the lack of a global relaxation mechanism.

An important class of flows, in terms of long term behavior, are those that are independent of time.
Then, a more realistic task is to study the dynamics of solutions near steady-states. 
Coherent structures, such as shear flows, radial flows and vortices, are especially important in the study of the 2D Euler equation, since physical experiments, numerical simulations and formal asymptotic  show that they tend to form dynamically and become the dominant part of the solution at very long times.

The stability of these steady-states is a classical subject and a fundamental problem in the context of \textit{hydrodynamic stability} and has been considered by many authors. We refer to \cite{BGMsurvey,Gallay,IJ_icm} for a general overview of this topic and a more detailed list of references. 

\subsection{The Taylor-Couette flow} In this article we are interested in the 2D Euler dynamics near Taylor-Couette flow in an annular domain.
This radial flow is given by
\begin{equation}\label{def:u_TC} 
u_{\textsf{TC}}^{\theta}(r)=Ar +\frac{B}{r}, \qquad r_1\leq r\leq r_2.
\end{equation}
Here $A,B\in\R$ and $r_1,r_2\in(0,\infty)$ are the inner and outer radius of the annular domain.

Such type of two-dimensional flow can be obtained experimentally by rotating two concentric cylinders, where the rotation is slow enough not to cause a 3D Taylor-Couette instability.

Despite being a simple type of steady solution, perturbation of Taylor-Couette has proven to be a challenging topic and has been extensively studied experimentally, theoretically and numerically for a long time. Many questions remain unanswered, making it an active field of research in fluid mechanics, see \cite{CI,GLS,HLS,T}.

Now, for the sake of completeness, we recall that velocity, vorticity and stream function associated to the stationary Taylor-Couette 
flow are given in Cartesian coordinates by
\[
\bu_{\textsf{TC}}(\bx)=\left(A|\bx|+\frac{B}{|\bx|}\right)\frac{\bx^{\perp}}{|\bx|}, \qquad \omega_{\textsf{TC}}(\bx)=2A, \qquad \psi_{\textsf{TC}}(\bx)=\frac{A}{2}(r_1^2-|\bx|^2)+B \log\left(\frac{r_1}{|\bx|}\right).
\]

In order to present our result, we start considering the vorticity formulation for the Euler equation in the two-dimensional annular domain $\overline{B_{r_2}(0)}\setminus B_{r_1}(0)$ with no-penetration boundary conditions, i.e. $\bu\cdot\bn=0$ on $\p B_{r_2}(0) \cup \p B_{r_1}(0)$. Then, we have
\begin{equation}\label{eq:2dvorticity}
\p_t \omega + \bu \cdot \nabla \omega =0, \qquad \bu = \nabla^{\perp}\psi,
\end{equation}
with the stream function $\psi$ satisfying $\Delta\psi=\omega$ and the boundary conditions
\begin{equation}\label{eq:BC}
\psi|_{\p B_{r_1}(0)}=0, \qquad \psi|_{\p B_{r_2}(0)}=\gamma.
\end{equation}
Clearly, $\psi$ is defined uniquely by $\bu$ up to a constant. Therefore, without loss of generality we set $\psi|_{\p B_{r_1}(0)}=0$. The constant $\gamma$ may in general depend on time but by Kelvin's theorem on conservation of circulation $\bu$ is also independent of time along the flow. In fact, see \eqref{eq:gammacirculation} we have
\[
\gamma=\frac{-1}{2\pi}\int_{[r_1,r_2]\times\T} u_0^{\theta}(r,\theta)\dr\mbox{d}\theta.
\]

Since we are  interested in the 2D Euler dynamics near Taylor-Couette, we will set the circulation to be exactly the same as the Taylor-Couette flow, i.e.
\begin{equation}\label{def:gamma}
\gamma:=-\int_{r_1}^{r_2} u_{\textsf{TC}}^{\theta}(r) \dr=-\left(\frac{A}{2}(r_2^2-r_1^2) + B \log\left(\frac{r_2}{r_1}\right)\right).
\end{equation}

\subsection{Main result}
The 2D Euler equation can be seen as a Hamiltonian system, and thus it is quite natural from a dynamical system point of view to explore whether time periodic solutions near specific equilibrium states may exist.

For the sake of clarity we shall now give an elementary statement of our main result. Let $0<r_1<r_2<\infty$ and $A,B\in\R$ with $B\neq 0$, then the following result holds:
\begin{theorem}\label{thmbasic}
For any $0 \leq s < \sfrac{3}{2}$ and $\epsilon>0$, the  2D Euler system \eqref{eq:2dvorticity}-\eqref{def:gamma} admits a non-trivial smooth time periodic solution $\omega$ with frequency $\lambda$ such that:
\begin{itemize}
    \item The distance to Taylor-Couette flow satisfies that $$\|\omega-2A\|_{H^{s}(\overline{B_{r_2}(0)}\setminus B_{r_1}(0))}\leq \epsilon.$$

    \item The frequency $\lambda$ of the temporal periodicity is of order $O(1)$ with respect to $\epsilon$.
\end{itemize}
\end{theorem}
\begin{remark}
Non-trivial means solutions that depend on the angular variable in a non-trivial way.
\end{remark}
This result will follow from Theorems \ref{thmGOAL}, \ref{thm:distance} and \ref{thm:fullregularity}, in which the properties of these solutions can be further elaborated.

\subsection{Steady states and their (asymptotic) stability}
The problem of finding steady states solutions for 2D Euler has been addressed by Nadirashvili in \cite{N2013}, where he studies the geometry and the stability of solutions, following the works of Arnold 
\cite{Arnold1,Arnold2,Arnold3}.
In recent years, a large number of results taking into account the geometry of the steady states and their flexibility/rigidity properties have emerged, see \cite{CG,CDG,ZEW,CS,HN_arma,HN_jems,HN_cpam,IK,LS,N} and references therein.

There are a multitude of steady solutions of the 2D Euler equation. For example, every shear flow in a strip-type domain or every radial flow in a circular-type domain. Being a steady state imposes the condition that vorticity gradients are locally parallel to gradients of
the stream function $(\n^\perp \psi \cdot \n\omega=0)$.
So, as a  consequence, a large class of steady solutions can be generated by specifying a Lipschitz function $F : \R\rightarrow \R$ and solving the elliptic problem $\D \psi=F(\psi).$
There is a privileged class of such steady state solutions given by solving the above elliptic problem, for some specific $F$, called \textit{Arnold stable}. These flows are nonlinearly Lyapunov stable in the $L^2$ topology of vorticity. It is also well known that there exist particular steady state solutions which possess even stronger stability properties like the Couette flow on the periodic channel which is nonlinearly Lyapunov stable in $L^\infty$.
For the most up-to-date overview, see \cite{DE} and references therein.\medskip

Asymptotic stability can also occur, although the theorems and scenarios are much more subtle and challenging.
This convergence back to equilibrium, despite time reversibility and the lack of dissipative mechanisms, is now known as \textit{inviscid damping} and shares analogies with \textit{Landau damping} in plasma physics, see the celebrated paper of Mohout-Villani \cite{MV} and references therein.

This work on \textit{inviscid damping}, initiated with the  breakthrough of Bedrossian-Masmoudi \cite{BM}, established that sufficiently small perturbations of the Couette flow in the Gevrey spaces $\mathcal{G}^s$, with $s >\sfrac{1}{2}$, converge  to a possibly new shear flow near Couette in $\T\times\R$. The exact same result does not hold for $s < \sfrac{1}{2}$, as it was shown in \cite{DM} by Deng-Masmoudi. See also the very recent results \cite{CWZZ,Zhao} for the inhomogeneous case.

In recent works Ionescu-Jia \cite{IJ_cmp,IJ_acta} and Masmoudi-Zhao \cite{MZ} proved that nonlinear asymptotic stability holds true also for perturbations around the Couette flow and more general monotone shear flows in the finite periodic channel $\T\times [0,1]$.

For the steady state of interest in this manuscript, i.e., the Taylor-Couette flow it is important to mention the very recent results \cite{An1} and \cite{An2}, where the authors investigate the  nonlinear stability of the  Taylor-Couette  for the incompressible Navier-Stokes equations.
See \cite{BZV,CZ,GS,Zillinger0}  for other related results on the mathematical study of the Taylor-Couette flow or more general circular flows.\\

Motivated by the results mentioned above, the linearized equations around more general shear flows were investigated intensely in the last few years, see for example \cite{BMlinear,DZ1,DZ2,GNRS,Jia1,Jia2,Wei-Zhang-Zhao,Wei-Zhang-Zhao_1,Wei-Zhang-Zhao_2,Zillinger1,Zillinger2} and their references for a complete, but not exhaustive, list of references.

An interesting related work is that of Lin-Zeng \cite{LZ},  where they  proved that nonlinear inviscid damping is not true for perturbations of the Couette flow in $H^s$ with $s < \sfrac{3}{2}$. More specifically, they discovered Kelvin's cat's eyes steady states
nearby to the Couette shear flow. In fact, the dynamics at this low regularity is even richer and there are also travelling waves with velocity of order $O(1)$ as showed by the authors in \cite{CL}.

Recently, there has been a growing interest in the study of existence or not of invariant structures, and their stability for 2D Euler near other shear flows and for related equations, see \cite{ZEW,FMM,LLZ,LYZ,LWZZ,N,SZ,WZZ}.

\subsection{Tools and ideas behind the proof}
An important tool for the proof will be the use of bifurcation theory, which has been very useful in showing the existence of solutions to various equations that arise in the field of fluid mechanics.

The fact that the vorticity of the solutions we are interested in is not constant within its support induces a more complex spectral problem  due to the larger dimension of the space on which the linear part of the equation acts. This topic turns out to be less explored in the literature and only a few relevant results are known, see \cite{CCG_memoirs,GHM,GHS}.

To be more specific, we point out that the vorticity profile we are interested in is constant outside a pair of very thin regions $I_\e^{R_1}, I_\e^{R_2}$ located at $R_1,R_2\in(r_1,r_2)$ where the dynamics occurs, and the thickness $\e$ of this regions is related with the distance of the steady Taylor-Couette flow and the time periodic solution that we obtain in the procedure.

\subsubsection{Comparison and differences with Couette flow} Compared to shear-type Couette flow, the circular-type Taylor-Couette flow involves additional coefficients $A$ and $B$ that take into account rotational effects. Unlike the monotone and linear Couette flow, the appropriate choice of the free parameters $A\in\R$ and $B\in\R\setminus\{0\}$ makes the Taylor-Couette flow a much more interesting and challenging radial flow, which can have regions of concavity and convexity.

We observe that for the case $B=0$, i.e., constant angular velocity, the perturbations rotate while maintaining their shape at the linear level. However, in the case when $B\neq 0$, the perturbation is ``sheared'' in a way reminiscent of plane Couette flow.

An important difference with respect to the Couette flow is that here we have to be much more careful when choosing the parameters $R_1, R_2\in(r_1,r_2).$ The choice of points $R_1, R_2\in(r_1,r_2)$ around where $\supp(\nabla\omega)$ is concentrated depends strongly on the geometry of the Taylor-Couette flow profile. As we will see in Section \ref{s:dim1kernel}, in order to get a one dimensional kernel of the linear operator we need to impose that
\[
R_1,R_2\in(r_1,r_2)\quad \text{such that}\quad u_{\textsf{TC}}^{\theta}(R_1) \neq u_{\textsf{TC}}^{\theta}(R_2).
\]

Another important difficulty to be faced is related to the form of the Green's function associated to the Laplacian in an annular domain. The situation in bounded domains turns out to be more delicate due to the presence of boundaries. This has an impact on the study of the regularity of the functional that will describe the periodic solutions over time. 

For the Couette flow we used the explicit representation of the Green's function of the Laplacian in $\T\times\R$, which facilitates the study of the regularity of the functional. However, for the radial Taylor-Couette flow, as we are working in an annular domain, we do not have a manageable explicit expression for it. This type of difficulty has appeared previously in the literature. For example, in \cite{HXX}, for the generalized surface quasi-geostrophic equation in the disc. In that paper, the authors circumvent this issue by a suitable splitting into a singular explicit part  and a smooth implicit one induced by the boundary of the domain. That is,
\[
K^\alpha(x,y)=\frac{c_\alpha}{|x-y|^\alpha}+ K_1^\alpha(x,y), \qquad x,y\in \overline{B_1(0)},
\]
where $K_1^\alpha$ is a smooth function in $\overline{B_1(0)}\times \overline{B_1(0)}$. 

In the present paper we use a different way to study the regularity of the $F$-functional introduced in \eqref{def:Functional}, which can be easily adapted to other settings. Roughly speaking, we conduct a detailed study through the
elliptic problem associated to the stream function.

\subsection{Time periodic solutions in the literature} 
The topic of the existence of time periodic solutions is a classical issue in fluid dynamics, with a long history and many contributions.

In particular, following the approach of Burbea \cite{B}, there has been many works concerning the existence of single or multiple patches  and their degenerate case of point vortices not only for 2D Euler equation but also for other two-dimensional active scalar equations such as the generalized surface quasi-geostrophic equation through  contour dynamics equations paired with bifurcation
theory, desingularization techniques and variational tools.
We refer to the interested reader, for more details, to the non-exhaustive list \cite{Cao1,Cao2,Cao3,Cao4,DPMW,HHHM,HHMV,G,GH,GD,GPSY1,GPSY2,HH,HMW,HR,HM} and the references therein.

Very recently, some important progress has been done on the existence of time quasi-periodic solutions using KAM techniques and Nash-Moser scheme, see \cite{BHM,GIP,HR_quasi,HHM} and references therein.


\subsection{Notation and organization of the paper}
Here, we will define the annular domain of interest using Cartesian and polar coordinates respectively. That is,
\begin{align}
\mathsf{C}_{r_1,r_2}&:=\{\bx=(x_1,x_2)\in\R^2: r_1 \leq |\bx|\leq r_2\}\equiv \overline{B_{r_2}(0)}\setminus B_{r_1}(0),\label{dom:cartesian}\\
\mathsf{P}_{r_1,r_2}&:=\{(r,\theta)\in \R^{+}\times\T: r_1\leq r \leq r_2\}\equiv [r_1,r_2]\times\T.\label{dom:polar}
\end{align}

In the rest, we are going to work mainly with polar coordinates. Therefore, setting the parameters $\e>0$ small enough and $R_1,R_2\in(r_1,r_2)$, we define an auxiliary domain $D_\e$ within $\mathsf{P}_{r_1,r_2}$ where all the dynamics occurs. This domain $D_\e$ will be given by
\begin{equation}\label{dom:dynamics}
D_\e:=I_\e^{R_1,R_2}\times\T,
\end{equation}
with
\[
I_\e^{R_1,R_2}=I_\e^{R_1}\cup I_\e^{R_2}, \qquad \text{and} \qquad I_\e^{R_i}=(R_i-\e,R_i+\e).
\]

\subsubsection{Organization}
The rest of the paper is organized as follows. In Section \ref{s:formulation}, we will write the equation for the vorticity contour lines via the Biot-Savart law. In Section \ref{s:bifurcation}, we introduce the spaces we will work with to apply the Crandall-Rabinowitz  theorem and study the regularity of the nonlinear functional. To obtain a more manageable linear operator, we decompose and rescale it appropriately in the Sections \ref{s:linear} and \ref{s:rescaling}. In Sections \ref{s:dim1kernel}, \ref{s:codimension}, \ref{s:adjoint} and \ref{s:transversality} we conduct the spectral study of the linearized operator around zero and under suitable assumptions we obtain a Fredholm operator of zero index. Finally, in Section \ref{s:mainthm} we quantify the distance between our solution and the Taylor-Couette flow and address the complete regularity of the constructed time periodic solution.
In the Appendix \ref{s:appendix} we recall and give the proof of some auxiliary lemmas used in the paper.

\section{Formulation of the problem}\label{s:formulation}
Since we are considering an annular domain, that is, a domain of circular type, it will be natural to work in the more appropriate setting of polar coordinates.

In polar coordinates $(r,\theta)\in \R^{+}\times\T$, the vorticity formulation for the Euler equation in the two-dimensional bounded smooth annular domain $\overline{B_{r_2}(0)}\setminus B_{r_1}(0)$ with no-penetration boundary conditions, i.e. $\bu\cdot\bn=0$ on $\p B_{r_2}(0)\cup \p B_{r_1}(0)$  reads as
\begin{equation}\label{eq:omegaradial}
\p_t \omega + u^r \p_r \omega + \frac{u^\theta}{r}\p_\theta \omega=0, \qquad (r,\theta)\in \mathsf{P}_{r_1,r_2},
\end{equation}
where the velocity vector field $\bu=(u^r,u^\theta)$ is recovered from the vorticity $\omega$ by means of the stream function $\psi$, via the relations
\begin{equation}\label{eq:BSlaw}
    (u^r,u^\theta)=\left(\frac{1}{r}\p_{\theta}\psi,-\p_r\psi\right), \qquad -\left(\p_{rr}+\frac{1}{r}\p_r +\frac{1}{r^2}\p_{\theta\theta} \right)\psi=\omega.
\end{equation}

The main task of this section will be to obtain a functional equation whose solutions give rise to time-periodic rotating solutions of
\begin{equation}\label{eq:2dEvorticity}
\p_t \omega + \frac{1}{r}\left(\p_\theta \psi\p_r \omega - \p_r \psi \p_\theta\omega\right)=0,   \qquad (r,\theta)\in  \mathsf{P}_{r_1,r_2},
\end{equation}
with stream function $\psi$ solving 
\begin{equation}\label{eq:BVPpsi}
\left\{
\begin{aligned}
\Delta_B \psi(r,\theta)  &= \omega(r,\theta),\qquad (r,\theta)\in \mathsf{P}_{r_1,r_2},\\
\psi|_{r=r_1}&=0,\\
\psi|_{r=r_2}&=\gamma.
\end{aligned}
\right.
\end{equation}
Here and in the rest of the papaer, $\Delta_B\equiv -\left(\p_{r}^2+\frac{1}{r}\p_r +\frac{1}{r^2}\p_\theta^2\right)$ is the Laplace-Beltrami operator and $\gamma$ is the circulation given by \eqref{def:gamma}.

\subsection{The 2D Euler as an equation for the level sets of the vorticity.}

In order to find solutions of \eqref{eq:2dEvorticity} by looking to the level sets of $\omega$. We assume that level sets of vorticity $\omega$ can be parameterized by $\Phi_t(\rho,\theta)=(\rho+f(\rho,\theta,t))(\cos\theta,\sin\theta)$ inside the support of $\nabla\omega$, in such a way that
\begin{equation}\label{levels}
\omega(\Phi_t(\rho,\theta),t)=\varpi(\rho),\qquad (\rho,\theta)\in D_\e\subsetneq \mathsf{P}_{r_1,r_2},
\end{equation}
for some smooth profile function $\varpi$.

We remark that, since the 2D Euler equation is a transport type equation, we can assume that profile function $\varpi$ does not depend on $t$ without any loss of generality.

Here, we assume that $f(\cdot,t)\in C^{2,\alpha}(D_\e)$ for all $t\geq 0$ with norm small enough, $\|f(t)\|_{C^{2,\alpha}(D_\e)}\ll 1$.  We immediately have that $\Phi_t(D_\e)\subsetneq \mathsf{C}_{r_1,r_2}$. This fact allows us to write (abusing the notation a little) the rest of the annular domain in Cartesian coordinates as follows
\[
\mathsf{C}_{r_1,r_2}\setminus \Phi_t(D_\e)=\Omega_{\textsc{Inner}}(t)\cup\Omega_{\textsc{Middle}}(t)\cup \Omega_{\textsc{Outer}}(t),
\]
with
\begin{align*}
    \Omega_{\textsc{Inner}}(t)&:=\{(\rho,\theta)\in \R^{+}\times \T : r_1\leq \rho \leq R_1-\e+f(R_1-\e,\theta,t)\},\\
    \Omega_{\textsc{Middle}}(t)&:=\{(\rho,\theta)\in \R^{+}\times \T :  R_1+\e+f(R_1+\e,\theta,t)\}\leq \rho \leq R_2-\e + f(R_2-\e,\theta,t)\},\\
    \Omega_{\textsc{Outer}}(t)&:=\{(\rho,\theta)\in \R^{+}\times \T :  R_2+\e+f(R_2+\e,\theta,t)\}\leq \rho \leq r_2\}.
\end{align*}
\begin{remark}
Here, we are considering $\e$ small enough such that
\begin{align*}
r_1&<R_1-\e+f(R_1-\e,\theta,t),\\
R_1+\e+f(R_1+\e,\theta,t) &< R_2-\e + f(R_2-\e,\theta,t),\\
R_2+\e+f(R_2+\e,\theta,t)&< r_2,
\end{align*}
for all $t\geq 0$ and $\theta\in\T.$
\end{remark}

Thus, in the remaining domain $\mathsf{C}_{r_1,r_2}\setminus \Phi_t(D_\e)$ the scalar vorticity is  just piecewise constant. Since we will define $\varpi$ such that
\[
\varpi(\rho)= 2A +
\begin{cases}
0, \qquad (\rho,\theta)\in \Omega_{\textsc{Inner}}(t),\\
\e, \qquad (\rho,\theta)\in \Omega_{\textsc{Middle}}(t),\\
0, \qquad (\rho,\theta)\in \Omega_{\textsc{Outer}}(t),
\end{cases}
\]
we have that
\begin{equation}\label{omegacte}
\omega(\rho(\cos\theta,\sin\theta),t)= 2A +
\begin{cases}
0, \qquad (\rho,\theta)\in \Omega_{\textsc{Inner}}(t),\\
\e, \qquad (\rho,\theta)\in \Omega_{\textsc{Middle}}(t),\\
0, \qquad (\rho,\theta)\in \Omega_{\textsc{Outer}}(t).
\end{cases}
\end{equation}

\begin{remark}
The precise definition of $\varpi$ has been postponed to Section \ref{s:profile}. For this point, the only information we need to know is that 
\begin{align*}
\text{supp}(\varpi)&=(R_1-\e,R_2+\e),\\
\text{supp}(\varpi')&=(R_1-\e,R_1+\e)\cup (R_2-\e,R_2+\e)\equiv I_\e^{R_1,R_2}.
\end{align*}
\end{remark}

Using contour dynamics together with the equations  \eqref{levels} and \eqref{omegacte} we obtain an equivalent formulation for the scalar vorticity equation \eqref{eq:2dEvorticity}-\eqref{eq:BVPpsi} in terms of a time-evolution equation for $f(\cdot,t)$ that depends only on the profile of $\varpi$, see \eqref{eq:flevel}.\\

Differentiating \eqref{levels} with respect to $\theta$ and with respect to $\rho$ we have that
\begin{align*}
\nabla \omega(\Phi_t(\rho,\theta),t)\p_\theta \Phi_t(\rho,\theta)&=0,\\
\nabla \omega(\Phi_t(\rho,\theta),t)\p_\rho \Phi_t(\rho,\theta)&=\varpi'(\rho).
\end{align*}
By straightforward computation, we get
\begin{equation}\label{eq:nablaomega}
\nabla \omega(\Phi_t(\rho,\theta),t)=\frac{\varpi'(\rho)}{\p_\theta\Phi_t^{\perp}(\rho,\theta)\cdot \p_\rho \Phi_t(\rho,\theta)}\p_\theta \Phi_t^{\perp}(\rho,\theta).
\end{equation}
Taking a time derivative in \eqref{levels}  and using \eqref{eq:2dEvorticity} and \eqref{eq:nablaomega} yields
\begin{align*}
0=\frac{d}{dt} \omega(\Phi_t(\rho,\theta),t)&=\p_t\omega(\Phi_t(\rho,\theta),t) + \nabla \omega(\Phi_t(\rho,\theta),t)\p_t\Phi_t(\rho,\theta)\\
&=(-\nabla^{\perp}\psi(\Phi_t(\rho,\theta),t)+\p_t\Phi_t(\rho,\theta))\cdot\nabla\omega(\Phi_t(\rho,\theta),t)\\
&=(-\nabla^{\perp}\psi(\Phi_t(\rho,\theta),t)+\p_t\Phi_t(\rho,\theta))\cdot \frac{\varpi'(\rho)}{\p_\theta\Phi_t^{\perp}(\rho,\theta)\cdot \p_\rho \Phi_t(\rho,\theta)}\p_\theta \Phi_t^{\perp}(\rho,\theta).
\end{align*}

Since by definition of the profile function $\varpi$ we have that $\text{supp}(\varpi')=I_\e^{R_1,R_2}$,  the above expression reduces to 
\[
\p_t\Phi_t(\rho,\theta)\cdot\p_\theta \Phi_t^{\perp}(\rho,\theta)=\nabla^{\perp}\psi(\Phi_t(\rho,\theta),t)\cdot\p_\theta \Phi_t^{\perp}(\rho,\theta).
\]
Recalling that $\Phi_t(\rho,\theta)=(\rho+f(\rho,\theta,t))(\cos\theta,\sin\theta)$ and defining the auxiliary function $\bar{\psi}$ adapted to the level sets
\[
\bar{\psi}(\rho,\theta,t):=\psi(\Phi_t(\rho,\theta),t),\qquad (\rho,\theta)\in D_\e,
\]
we immediately get
\begin{align*}
\p_t\Phi_t(\rho,\theta)\cdot\p_\theta \Phi_t^{\perp}(\rho,\theta)&=-(\rho+f(\rho,\theta,t))\pa_t f(\rho,\theta,t),\\
\nabla^{\perp}\psi(\Phi_t(\rho,\theta),t)\cdot\p_\theta \Phi_t^{\perp}(\rho,\theta)&=\p_\theta\bar{\psi}(\rho,\theta,t),
\end{align*}
and consequently the problem reduces to solve the following equation
\begin{equation}\label{eq:flevel}
(\rho+f(\rho,\theta,t)) \p_t f(\rho,\theta,t)+\p_\theta\bar{\psi}(\rho,\theta,t)=0, \qquad (\rho,\theta)\in D_\e, \qquad t\geq 0,
\end{equation}
with $\bar{\psi}=\psi\circ \Phi_t$ and $\psi$ solving \eqref{eq:BVPpsi}.

\subsubsection{Recovering the solution of 2D Euler from the level sets}
The above expression \eqref{eq:flevel} gives us the equation satisfying the level sets $\Phi_t(\rho,\theta)=(\rho +f(\rho,\theta,t))(\cos\theta,\sin\theta)$ of the scalar vorticity of the 2D Euler equation \eqref{eq:2dEvorticity} that we are interested in.\\

Conversely, if $f(\cdot,t)\in C^{\infty}(D_\e)$  with $\|f(t)\|_{C^{2,\alpha}(D_\e)}\ll 1$ for all $t\geq 0$ satisfies the previous equation \eqref{eq:flevel} on $D_\e$ for some $\varpi\in C^{\infty}([r_1,r_2])$, we can prove that the function $\omega(\cdot,t)$ defined implicitly by
\eqref{levels} on $D_\e$, i.e. $\omega=\varpi\circ\Phi_t^{-1}$
and extended by suitable constants to the complementary of
\[
\left\lbrace \bx\in \mathsf{C}_{r_1,r_2} : \bx=(\rho+f(\rho,\theta,t))(\cos\theta,\sin\theta) \text{ with } (\rho,\theta)\in D_\e \right\rbrace,
\]
is a smooth solution of the original problem \eqref{eq:2dEvorticity} over $\mathsf{P}_{r_1,r_2}$.

\subsection{The equation for the time periodic  solution}

In this manuscript we are interested in the existence of time periodic rotating solutions of \eqref{eq:flevel}. That is, we will look for solutions of the form
\begin{equation}\label{ansatzrho}
f(\rho,\theta,t):=\ff(\rho,\theta+\lambda t), \qquad (\rho,\theta)\in D_\e, \quad t\geq 0,
\end{equation}
for some value $\lambda\in\R\setminus\{0\}$ and some profile $\ff\in C^{2,\alpha}(D_\e).$

Putting the previous ansatz \eqref{ansatzrho} into \eqref{eq:flevel}, our original time-dependent problem reduces to solve the time-independent equation
\begin{equation}
\l(\rho+\ff(\rho,\theta))  \p_\theta\ff(\rho,\theta)+\p_\theta\bar{\psi}[\ff](\rho,\theta)=0, \qquad (\rho,\theta)\in D_\e,
\end{equation}
with $\bar{\psi}[\ff](\rho,\theta)=\psi(\rho+\ff(\rho,\theta),\theta)$ and $\psi$ solving \eqref{eq:BVPpsi}.

\begin{remark}
It is important to emphasize that $\bar{\psi}=\psi\circ\Phi_t$ depends functionally on the function $\ff$. Then, we will  write in the rest of the manuscript $\bar{\psi}[\ff]$ to make that dependence explicit.
\end{remark}

\subsubsection{The functional equation}

Let us define our functional candidate taking into account \eqref{eq:flevel}. That is,
\[
\tilde{F}[\l,\ff](\rho,\theta):=\l(\rho+\ff(\rho,\theta))  \p_\theta\ff(\rho,\theta)+\p_\theta\bar{\psi}[\ff](\rho,\theta), \qquad (\rho,\theta)\in D_\e,
\]
which can be written also as
\[
\tilde{F}[\l,\ff](\rho,\theta)=\pa_\theta[\lambda (\rho + \ff(\rho,\theta))^2/2 + \bar{\psi}[\ff](\rho,\theta)].
\]
We also introduce the auxiliary functional
\[
A[\lambda,\ff](\rho):=\frac{1}{2\pi}\int_{\mathbb{T}}\lambda (\rho + \ff(\rho,s))^2/2 + \bar{\psi}[\ff](\rho,s) \mbox{d}s, \qquad \rho\in I_\e^{R_1,R_2},
\]
and define our final functional as
\begin{equation}\label{def:Functional}
F[\lambda,\ff](\rho,\theta):=\lambda (\rho + \ff(\rho,\theta))^2/2 + \bar{\psi}[\ff](\rho,\theta)-A[\lambda,\ff](\rho), \qquad (\rho,\theta)\in D_\e.
\end{equation}
Therefore, just by construction of $F[\cdot,\cdot]$, we have that $(\l,\ff)\in \R\times C^{2,\alpha}(D_\e)$ satisfies 
\[
F[\l,\ff]=0 \qquad \Longrightarrow \qquad \tilde{F}[\l,\ff]=0.
\]

Then, the problem of finding time periodic rotating solutions reduces to find the roots of the
nonlinear and nonlocal functional $F[\cdot,\cdot]$, that is, to study the equation
\begin{equation}\label{functionaleq}
F[\l,\ff](\rho,\theta)=0, \qquad (\rho,\theta)\in D_\e.
\end{equation}

\begin{remark}
Note that a  trivial curve of solutions is known
\[
F[\lambda,0](\rho,\theta)=0, \quad (\rho,\theta)\in D_\e, \quad \forall \l\in\R.
\]
\end{remark}

The main goal of the paper is to show the existence of non-trivial solutions of the equation \eqref{functionaleq}. We remind the reader that by non-trivial we mean solutions that depend on the angular variable in a non-trivial way. In particular, we prove the existence of time periodic rotating solutions of 2D Euler equation as close as we want to the steady Taylor-Couette bifurcating from the trivial one.

\remark{Notice that the solution given by 
\[
\omega((\rho+\ff(\rho,\theta+\l t))(\cos\theta,\sin\theta),t)=\varpi(\rho), \qquad (\rho,\theta)\in D_\e,
\]
and by expression $\eqref{omegacte}$ on the complementary set 
give rise, by an appropriate change of variables, to solutions of 2D Euler equation \eqref{eq:2dvorticity}-\eqref{def:gamma} of the form
\[
\bu(r,\theta,t)=(u^r(r,\theta,t),u^\theta(r,\theta,t))=(\bar{u}^r(r,\theta+\lambda t),\bar{u}^\theta(r,\theta+\lambda t)).
\]
}

\subsection{The profile function $\varpi_{\e,\k}$}\label{s:profile} To continue, we define the profile function $\varpi\in C^{\infty}([r_1,r_2])$ that will be used in Section \ref{s:dim1kernel} to solve the functional equation \eqref{functionaleq}. 

To begin with, let us first consider a profile with compact support in a much less regular space, that is
\[
\varpi\in  W^{1,\infty}([r_1,r_2])\cap H^{\sfrac{3}{2}^-}([r_1,r_2]).
\]
Let us consider the auxiliary function
\[
\varphi(z)=\frac{1-z}{2}, \quad |z|\leq 1.
\]
Fixed $R_1,R_2\in(r_1,r_2)$. Let $\e>0$, we define $\varpi_\e(r)$ in the following way
\begin{align*}
\varpi_\e(r)=\left\{\begin{array}{cr}
0 										& r_1<r<R_1-\e, \\
\e\varphi\left(\frac{R_1-r}{\e}\right) 	& R_1-\e \leq r \leq R_1+\e,\\
\e 										& R_1+\e<r<R_2-\e,\\
\e\varphi\left(\frac{r-R_2}{\e}\right) 	& R_2-\e \leq r \leq R_2 +\e,\\
0 										& R_2+\e<r<r_2.
\end{array}\right.
\end{align*}

Since we are interested in smooth solutions, we have to consider a smooth profile function of $\varpi_\e$.
In order to do that, we will use a $\k$-regularization of $\varphi$. The function $\varphi_\k:[-1,1]\mapsto[0,1]$ will be defined by 
\begin{equation}\label{varphikappa}
\varphi_\k(z)=1-\frac{\int_{-1}^{z}\left(\int_{-1+\k}^{+1-\k}\Theta\left(\frac{\bar{z}-\bar{\bar{z}}}{\k}\right)\mbox{d}\bar{\bar{z}}\right)\mbox{d}\bar{z}}{\int_{-1}^{+1}\left(\int_{-1+\k}^{+1-\k}\Theta\left(\frac{\bar{z}-\bar{\bar{z}}}{\k}\right)\mbox{d}\bar{\bar{z}}\right)\mbox{d}\bar{z}}, 
\end{equation}
where $\Theta$  is a mollifier so that is smooth and positive with $\supp(\Theta)\subset(-1,1)$ and $\int_{-1}^{+1}\Theta(\bar{z})\mbox{d}\bar{z}=1.$

We summarize the properties of this $\k$-regularization function in the following lemma without proof. For the interested reader, we will refer to  \cite[Lemma 2.2]{CL} for more details. 
\begin{lemma}\label{propiedadesvarphi}The function $\varphi_{\kappa}$ given by \eqref{varphikappa} satisfies
\begin{enumerate}
\item $\varphi_{\kappa}(-1)=1$, $\varphi_{\kappa}(1)=0$, $(\pa_z^{n}\varphi_{\kappa})(\pm 1)=0$, for $n=1,2,\ldots$
\item $\varphi_{\kappa}'(z)<0$, for $|z|<1$.
\item $\|\varphi'_{\kappa}+\frac{1}{2}\|_{L^1([-1,1])}\leq C\kappa$.
\end{enumerate}
\end{lemma}

The final profile  that will be used in this manuscript is given by
\begin{align}\label{varpiepkappa}
\varpi_{\e,\k}(r)=\left\{\begin{array}{cr}
0 										& r_1<r<R_1-\e, \\
\e\varphi_\k\left(\frac{R_1-r}{\e}\right) 	& R_1-\e \leq r \leq R_1+\e,\\
\e 										& R_1+\e<r<R_2-\e,\\
\e\varphi_\k\left(\frac{r-R_2}{\e}\right) 	& R_2-\e \leq r \leq R_2 +\e,\\
0 										& R_2+\e<r<r_2.
\end{array}\right.
\end{align}
Let us point out that $\supp(\varpi'_{\ep,\kappa})=I_\e^{R_1,R_2}$. That is, the domain
\begin{align*}
I_\ep^{R_1,R_2}=I_\e^{R_1}\cup I_\e^{R_2}=(R_1-\ep,R_1+\ep)\cup (R_2-\ep, R_2+\ep)\subset [r_1,r_2].
\end{align*}

In principle the values of $R_1, R_2$ could be any real numbers just satisfying $r_1<R_1<R_2<r_2$. But as we will see in Section \ref{s:dim1kernel} to prove that $\text{dim}(\cN(\cL[\l_\star]))=1$ we need to choose $R_1,R_2$ such that 
\[
u_{\textsf{TC}}^{\theta}(R_1) \neq u_{\textsf{TC}}^{\theta}(R_2).
\]
Moreover, the properties of the time periodic solution we construct  depend strongly on these values. For example, the frequency of the rotation we obtain will be $u_{\textsf{TC}}^{\theta}(R_1)$ or $u_{\textsf{TC}}^{\theta}(R_2).$\medskip

The rest of the manuscript consists of finding a non-trivial solution of \eqref{functionaleq} with $\varpi \equiv 2A+\varpi_{\e,\k}$ in \eqref{def:Functional} for parameters $\e$ and $\k$ small enough. It will be done using the bifurcation theory through Crandall--Rabinowitz theorem \cite{CR}. For the completeness of the paper we recall this basic theorem and it will
referred to as sometimes by C-R theorem.

\section{Bifurcation theory and Crandall-Rabinowitz}\label{s:bifurcation}

The resolution of a nonlinear equation of the type \eqref{functionaleq} together with the existence of a trivial
line of solutions $F[\l,0]=0$ can be studied through bifurcation theory. In particular, here we will apply
the well-known Crandall-Rabinowitz theorem whose proof can be found in \cite{CR}.

\begin{theorem}\label{th:CR} Let $X, Y$ be two Banach spaces, $V$ a neighborhood of $0$ in $X$ and let
$
F : \R \times V \to Y
$
with the following  properties:
\begin{enumerate}
\item $F [\lambda, 0] = 0$ for any $\lambda\in \R$.
\item The partial derivatives $D_\l F$, $D_{\ff}F$ and $D^2_{\l,\ff}F$ exist and are continuous.
\item There exists $\lambda_\star$ such that if $\mathcal{L}_\star= D_{\ff} F[\l_\star,0]$ then  $\cN(\mathcal{L}_\star)$ and $Y/\cR(\mathcal{L}_\star)$ are one-dimensional.
\item {\it Transversality assumption}: $D^2_{\l,\ff}F[\l_\star, 0]\hh_\star \not\in \cR(\mathcal{L}_\star)$, where
$$
\cN(\mathcal{L}_\star) = \text{span}\{\hh_\star\}.
$$
\end{enumerate}
If $Z$ is any complement of $\cN(\mathcal{L}_\star)$ in $X$, then there is a neighborhood $U$ of $(\l_\star,0)$ in $\R \times X$, an interval $(-\sigma_0,\sigma_0)$, and continuous functions $\varphi: (-\sigma_0,\sigma_0) \to \R$, $\psi: (-\sigma_0,\sigma_0) \to Z$ such that $\varphi(0) = 0$, $\psi(0) = 0$ and
\begin{align*}
F^{-1}(0)\cap U=&\Big\{\big(\lambda_\star+\varphi(\sigma), \sigma h_\star+\sigma \psi(\sigma)\big)\,;\,\vert \sigma\vert<\sigma_0\Big\} \cup \Big\{(\lambda,0)\,;\, (\lambda,0)\in U\Big\}.
\end{align*}

\end{theorem}

The bulk of the work consists of proving all assumptions of Theorem \ref{th:CR}. This will be done in detail in the
the following sections.

\subsection{Functional setting and regularity}
In order to apply the Crandall-Rabinowitz  theorem we need first to fix the functional spaces. We should look for Banach spaces $X$ and $Y$ such that $F : \R \times V\subset X \rightarrow Y$ is well-defined and satisfies the required assumptions.

Our first step is to recall that we have to solve the functional equation $\eqref{functionaleq}$ over $\overset{\circ}{\supp}(\varpi'_{\e,\k})\times \T$, where the profile function $\varpi_{\e,\k}$ is given by \eqref{varpiepkappa}. By definition, the support of $\varpi'_{\e,\k}$ is just 
$$I_\ep^{R_1,R_2}=(R_1-\ep,R_1+\ep)\cup (R_2-\ep, R_2+\ep).$$

Then, recalling that $D_\e=I_\ep^{R_1,R_2} \times \T$, the spaces $X$ and $Y$ are given by
\begin{equation}\label{def:spaceX}
X(D_\e):=\left\lbrace \gg\in C^{2,\alpha}(D_\e) :  \text{$\gg$ is even in $\theta$ with} \int_{\T}\gg(r,\theta)\mbox{d}\theta=0 \right\rbrace,
\end{equation}
and 
\begin{equation}\label{def:spaceY}
Y(D_\e):=X(D_\e).
\end{equation}
Here $C^{2,\alpha}(D_{\ep})$ is the H\"older space of $2\pi$-periocic functions in the $\theta$-variable with norm
\begin{align*}
\|\gg\|_{C^{2,\alpha}(D_{\ep})}:= \| \gg\|_{C^2(D_{\ep})}+|\nabla^2 \gg|_{C^{0,\alpha}(D_\e)},
\end{align*}
where $\|\cdot\|_{C^{2}}$ and $|\cdot|_{C^{0,\alpha}}$ are the standard norm and semi-norm respectively.\\

Given these definitions, the main objective of this section is to show the following lemmas:

\begin{lemma} For all $0<\e<\e_0(r_1,R_1,R_2,r_2),$ there exist $\d(\ep_0)$ small enough such that
\begin{align*}
F\,:\, \R\times \mathbb{B}_{\d}(X(D_{\ep})) & \to Y(D_{\ep}),\\
\quad (\l,\ff) & \to F[\l,\ff]
\end{align*}
where
\begin{align*}
\mathbb{B}_{\d}(X(D_{\ep})) :=\{ \gg\in X(D_{\ep})\,:\, \|\gg\|_{C^{2,\alpha}(D_{\ep})}< \d\}.
\end{align*}
\end{lemma}
\begin{proof}
Let $\ff\in X(D_\e)$, we have that $F[\l,\ff]\in Y(D_\e)$ by direct application of the next three facts:
\begin{itemize}
    \item[i)] $\ff\in C^{2,\a}(D_\e)\quad \Longrightarrow \quad \bar{\psi}[\ff]\in C^{2,\a}(D_\e).$ 
    
    \item[ii)] $\ff \text{ is even in } \theta \quad \Longrightarrow \quad \bar{\psi}[\ff] \text{ is even in }\theta.$

    \item[iii)] $F[\l,\ff]$ is mean zero in the angular variable.
\end{itemize}

\noindent
\textit{Proof of i).} Recall that $\bar{\psi}[\ff]=\psi[\ff]\circ \Phi[\ff]$ with $\Phi[\ff]=(\Phi_1[\ff],\Phi_2[\ff])=(r+\ff(r,\theta),\theta)\in C^{2,\a}(D_\e).$ Since $r\in I_\e^{R_1,R_2}\subset [r_1,r_2]$ with $0<r_1<r_2<\infty$ and the norm $\norm{\ff}_{C^{2,\a}(D_\e)}$ is small, we can define their inverse $\Phi^{-1}[\ff]=(\Phi_1^{-1}[\ff],\Phi_2^{-1}[\ff])\in C^{2,\a}(\Phi[\ff](D_\e))$, with $\Phi[\ff](D_\e)\subsetneq \mathsf{P}_{r_1,r_2}$ satisfying 
\[
r_1< \Phi_1[\ff](R_1-\e), \quad \Phi_1[\ff](R_1+\e)< \Phi_1[\ff](R_2-\e), \quad \Phi_1[\ff](R_2+\e)<r_2.
\]
As $\varpi\in  C^{\infty}([r_1,r_2])$ and $\Phi_1^{-1}[\ff]\in C^{2,\a}(\Phi[\ff](D_\e))$, we have that $\omega[\ff]$ given by \eqref{aux:omega[f]} belongs to  $C^{2,\a}(\mathsf{P}_{r_1,r_2})$. Moreover, applying elliptic regularity theory into the system \eqref{aux:DBpsi} we obtain that 
$\psi[\ff]\in C^{4,\a}(\mathsf{P}_{r_1,r_2}).$
Therefore, we can conclude that $\bar{\psi}[\ff]=\psi[\ff]\circ \Phi[\ff]\in C^{2,\a}(D_\e).$ \medskip

\noindent
\textit{Proof of ii).} Since $\ff$ is even in the angular variable, the result follows immediately by the definition of $\bar{\psi}[\ff](r,\theta)=\psi[\ff](r+\ff(r,\theta),\theta)$ and the evenness of $\psi[\ff]$ in the angular variable. To obtain the evenness of $\psi[\ff]$ we just remember that it is solution of the elliptic problem $\Delta_B \psi[\ff]=\omega[\ff]$ with an even ``source'' term. The latter is deduced by the facts that $\omega[\ff](r+\ff(r,\theta),\theta)=\varpi(r)$ for all $\theta\in\T$, and $\ff$ is even in the angular variable by hypothesis. \medskip

\noindent
\textit{Proof of iii).} This is direct consequence of the definition of the auxiliary functional $A[\l,\ff].$
\end{proof}

Moreover, the functional satisfies the regularity conditions given by the next lemma.

\begin{lemma}
For all $0<\e<\e_0(r_1,R_1,R_2,r_2),$ there exist $\d(\ep_0)$ small enough such that
\begin{enumerate}
	 \item The functional $F:\R \times \mathbb{B}_{\d}(X(D_{\ep}))\rightarrow Y(D_{\ep})$ is of class $C^1.$
	 \item The partial derivative $D^2_{\l,\ff} F:\R \times \mathbb{B}_{\d}(X(D_{\ep}))\rightarrow Y(D_{\ep})$ is continuous.
\end{enumerate}
\end{lemma}

\begin{proof}
The derivatives with respect $\l$ are trivial. Consequently, the result reduces to proving that there exists the Fr\'{e}chet differential of $F$:
\begin{align*}
D_\ff F[\l,\ff]\,:\, X(D_{\ep}) & \to Y(D_{\ep}),\\
\quad \hh & \to D_\ff F[\l,\ff]\hh
\end{align*}
and that it is continuous for all $\ff\in \mathbb{B}_{\d}(X(D_{\ep}))$ and $\lambda\in\R$.

This will be done by firstly showing the existence of the linear G\^ateaux derivative and secondly its continuity in the strong topology.

Actually, the G\^ateaux differential of $F[\l,\cdot]$ at $\ff\in \mathbb{B}_{\d}(X(D_{\ep}))$ in the direction $\hh\in X(D_{\ep})$
is defined as
\[
D_\ff F[\l,\ff]\hh:=\frac{\mbox{d}}{{\mbox{d}\tau}}F[\l,\ff+\tau \hh]|_{\tau=0}=\lim_{\tau\to0}\frac{F[\l,\ff+\tau\hh]-F[\l,\ff]}{\tau}.
\]
A formal straightforward computation show that this derivative is given by
\begin{multline}\label{def:dF}
D_\ff F[\l,\ff]\hh(r,\theta)=\l(r+\ff(r,\theta))\hh(r,\theta) + D_\ff\bar{\psi}[\ff]\hh(r,\theta)\\
-\frac{1}{2\pi}\int_\T \l(r+\ff(r,s))\hh(r,s)+ D_\ff\bar{\psi}[\ff]\hh(r,s)\ds.
\end{multline}
To prove this rigorously, we need to get
\[
\lim_{\tau\to 0} \left\Vert \frac{F[\l,\ff+\tau\hh]-F[\l,\ff]}{\tau} -D_\ff F[\l,\ff]\hh\right\Vert_{C^{2,\a}(D_{\ep})}=0.
\]
By virtue of \eqref{def:dF}, and  straightforward computations, it is enough to prove that there exists $D_\ff\bar{\psi}[\ff]\hh$ such that
\begin{equation}\label{aux:dbarpsi}
\lim_{\tau\to 0} \left\Vert \frac{\bar{\psi}[\ff+\tau\hh]-\bar{\psi}[\ff]}{\tau} -D_\ff\bar{\psi}[\ff]\hh \right\Vert_{C^{2,\a}(D_{\ep})}=0.
\end{equation}

In order to prove that, we start recalling  that for a generic function $f\in C^{2,\a}(D_\e)$ we have that $\bar{\psi}[f](\rho,\theta)=\psi[f](\Phi[f])(r,\theta)=\psi[f](\rho + f(\rho,\theta),\theta)$ with $\psi[f]$ solving the problem
\begin{equation}\label{aux:DBpsi}
\left\{
\begin{aligned}
\Delta_B \psi[f](r,\theta)&=\omega[f](r,\theta),\qquad (r,\theta)\in \mathsf{P}_{r_1,r_2},\\
\psi[f](r_1,\theta)&=0, \\
\psi[f](r_2,\theta)&=\gamma,
\end{aligned}
\right.
\end{equation}
and 
\begin{equation}\label{aux:omega[f]}
\omega[f](r,\theta)=2A +\begin{cases}
0 \qquad & r_1\leq r < \Phi_1[f](R_1-\e,\theta),\\
\varpi_{\e,\k}(\Phi_1^{-1}[f](r,\theta))\qquad &\Phi_1[f](R_1-\e,\theta) \leq r \leq \Phi_1[f](R_1+\e,\theta),\\
\e & \Phi_1[f](R_1+\e,\theta) < r < \Phi_1[f](R_2-\e,\theta),\\
\varpi_{\e,\k}(\Phi_1^{-1}[f](r,\theta))\qquad &\Phi_1[f](R_2-\e,\theta) \leq r \leq \Phi_1[f](R_2+\e,\theta),\\
0 & \Phi_1[f](R_2+\e,\theta) < r \leq r_2,
\end{cases}
\end{equation}
where
\[
\Phi[f](r,\theta)\equiv (\Phi_1[f](r,\theta),\Phi_2[f](r,\theta)):=(r+f(r,\theta),\theta), \qquad (r,\theta)\in D_\e.
\]
\begin{remark}
In this section, we remove the subscripts $\e$ and $\k$ from $\varpi_{\e,\k}$ to alleviate
the notation.
\end{remark}

Proceeding similarly as before, we  start computing formally $D_\ff \bar{\psi}[\ff]\hh$, which can be compute as
\begin{align*}
D_\ff \bar{\psi}[\ff]\hh&=\frac{d}{d\tau}\bar{\psi}[\ff + \tau \hh]\bigg|_{\tau=0}\\
&=\frac{d}{d\tau} \left(\psi[\ff + \tau \hh]\circ \Phi[\ff + \tau \hh]\right)\bigg|_{\tau=0}\\
&=\left(\frac{d}{d\tau}\psi[\ff + \tau \hh]\right)\circ \Phi[\ff + \tau \hh] + \left(\nabla \psi[\ff + \tau \hh]\circ \Phi[\ff + \tau \hh]\right)\cdot\frac{d}{d\tau}\Phi[\ff + \tau \hh]\bigg|_{\tau=0}\\
&=D_\ff \psi[\ff]\hh \circ \Phi[\ff] + \left(\p_r \psi[\ff]\circ\Phi[\ff] \right)\hh.
\end{align*}
That is,
\begin{equation}\label{def:dbarpsi}
D_\ff \bar{\psi}[\ff]\hh(r,\theta):=D_\ff \psi[\ff]\hh(r+\ff(r,\theta),\theta) + \p_r \psi[\ff](r+\ff(r,\theta),\theta) \hh(r,\theta).
\end{equation}
Since the second term of the above expression \eqref{def:dbarpsi} is just a derivative in the classical sense and $\p_r \psi[\ff](r+\ff(r,\theta),\theta) \hh(r,\theta)\in C^{2,\a}(D_\e)$ (remember that $\psi\in C^{4,\a}(\mathsf{P}_{r_1,r_2})$), we have that \eqref{aux:dbarpsi} reduces to prove that there exists $D_\ff \psi[\ff]\hh$ such that
\begin{equation}\label{aux:dpsi}
\lim_{\tau\to 0} \left\Vert \frac{\psi[\ff+\tau\hh]-\psi[\ff]}{\tau} -D_\ff \psi[\ff]\hh \right\Vert_{C^{2,\a}(D_{\ep})}=0.
\end{equation}

Taking into account \eqref{aux:DBpsi} we have that $D_\ff \psi[\ff]\hh$ should solve the problem
\begin{equation}\label{aux:DeltaDF}
\left\{
\begin{aligned}
\Delta_B \left(D_\ff \psi[\ff]\hh\right)(r,\theta)&=D_\ff \omega[\ff]\hh(r,\theta),\qquad (r,\theta)\in \mathsf{P}_{r_1,r_2},\\
D_\ff \psi[\ff]\hh(r_1,\theta)&=0, \\
D_\ff \psi[\ff]\hh(r_2,\theta)&=0.
\end{aligned}
\right.
\end{equation}
Now, we are going to obtain an explicit expression for the ``source'' term of \eqref{aux:DeltaDF}. That is, we are going to compute the differential of $\omega[f]$ with respect $f$ at the point $\ff$ in the direction $\hh$.

Recalling \eqref{aux:omega[f]}, that is $\omega[f](\Phi[f](r,\theta))=\varpi(r)$ for $(r,\theta)\in D_\e$, we obtain 
\[
D_\ff \omega[\ff]\hh(r+\ff(r,\theta),\theta)+\p_r\omega[\ff] (r+\ff(r,\theta),\theta)\hh(r,\theta)=0, 
\]
and
\[
\p_r\omega[\ff](r+\ff(r,\theta),\theta)(1+\p_r\ff(r,\theta))=\varpi'(r) \quad \Longrightarrow \quad \p_r\omega[\ff](r,\theta)=\frac{\varpi'(\Phi_1^{-1}[\ff](r,\theta))}{1+\p_r\ff(\Phi^{-1}[\ff](r,\theta))}.
\]
Hence, we immediately get the candidate
\begin{equation}\label{aux:Dfw[f]}
D_\ff \omega[\ff]\hh (r,\theta)=\begin{cases}
0 \qquad & r_1\leq r < \Phi_1[\ff](R_1-\e,\theta),\\
-\frac{\varpi'(\Phi_1^{-1}[\ff](r,\theta)) \hh(\Phi^{-1}[\ff](r,\theta))}{1+\p_r\ff(\Phi^{-1}[\ff](r,\theta))} \qquad & \Phi_1[\ff](R_1-\e,\theta) \leq r \leq \Phi_1[\ff](R_1+\e,\theta),\\
0 & \Phi_1[\ff](R_1+\e,\theta) < r < \Phi_1[\ff](R_2-\e,\theta),\\
-\frac{\varpi'(\Phi_1^{-1}[\ff](r,\theta)) \hh(\Phi^{-1}[\ff](r,\theta))}{1+\p_r\ff(\Phi^{-1}[\ff](r,\theta))} \qquad & \Phi_1[\ff](R_2-\e,\theta) \leq r \leq \Phi_1[\ff](R_2+\e,\theta),\\
0 & \Phi_1[\ff](R_2+\e,\theta) < r\leq r_2.
\end{cases}
\end{equation}

At this point, combining \eqref{aux:DBpsi} and \eqref{aux:DeltaDF} we obtain that auxiliary function 
\[
\Psi_\tau[\ff]\hh:=D_\ff \psi[\ff]\hh -\frac{\psi[\ff+\tau\hh]-\psi[\ff]}{\tau},
\]
solves (by definition) the elliptic-type problem
\begin{equation}\label{aux:dF_elliptic}
\left\{
\begin{aligned}
\Delta_B \left(\Psi_\tau[\ff]\hh\right)(r,\theta)&=\left(D_\ff \omega[\ff]\hh-\frac{\omega[\ff+\tau\hh]-\omega[\ff]}{\tau}\right)(r,\theta),\qquad (r,\theta)\in \mathsf{P}_{r_1,r_2},\\
\left(\Psi_\tau[\ff]\hh\right)(r_1,\theta)&=0, \\
\left(\Psi_\tau[\ff]\hh\right)(r_2,\theta)&=0.
\end{aligned}
\right.
\end{equation}
Hence, by classical elliptic regularity theory we can conclude from \eqref{aux:dF_elliptic} that
\[
\norm{D_\ff \psi[\ff]\hh -\frac{\psi[\ff+\tau\hh]-\psi[\ff]}{\tau}}_{C^{2,\alpha}(\mathsf{P}_{r_1,r_2})}\lesssim \norm{D_\ff \omega[\ff]\hh-\frac{\omega[\ff+\tau\hh]-\omega[\ff]}{\tau}}_{C^{0,\alpha}(\mathsf{P}_{r_1,r_2})}.
\]
with $D_\ff \omega[\ff]\hh$ given by \eqref{aux:Dfw[f]}. Now, we will check that actually 
\[
\lim_{\tau\to 0} \norm{D_\ff \omega[\ff]\hh-\frac{\omega[\ff+\tau\hh]-\omega[\ff]}{\tau}}_{C^{0,\alpha}(\mathsf{P}_{r_1,r_2})}=0,
\]
what it implies \eqref{aux:dpsi}. In order to do that, we start analyzing in detail the difference $\omega[\ff+\tau\hh]-\omega[\ff]$.
Without lost of generality, the direction $\hh\in X(D_\e)$ can be taken with norm $\norm{\hh}_{C^{2,\a}(D_\e)}=1$. Then, taking $\tau$ small enough, we have that 
\begin{equation}\label{aux:difference}
\omega[\ff + \tau \hh]-\omega[\ff]=\begin{cases}
0                   &r_1\leq r < \mathsf{m}_\tau(R_1-\e,\theta),\\
\omega[\ff + \tau \hh]-\omega[\ff]               & \mathsf{m}_\tau(R_1-\e,\theta) \leq r < \mathsf{M}_\tau(R_1-\e,\theta),\\
\varpi\circ \Phi_1^{-1}[\ff+\tau\hh]-\varpi\circ\Phi_1^{-1}[\ff]\quad &\mathsf{M}_\tau(R_1-\e,\theta)\leq r < \mathsf{m}_\tau(R_1+\e,\theta),\\
\omega[\ff + \tau \hh]-\omega[\ff] &\mathsf{m}_\tau(R_1+\e,\theta)\leq r < \mathsf{M}_\tau(R_1+\e,\theta),\\
0 & \mathsf{M}_\tau(R_1+\e,\theta)\leq r < \mathsf{m}_\tau(R_2-\e,\theta),\\
\omega[\ff + \tau \hh]-\omega[\ff] & \mathsf{m}_\tau(R_2-\e,\theta)\leq r < \mathsf{M}_\tau(R_2-\e,\theta),\\
\varpi\circ \Phi_1^{-1}[\ff+\tau\hh]-\varpi\circ\Phi_1^{-1}[\ff]\quad &\mathsf{M}_\tau(R_2-\e,\theta)\leq r < \mathsf{m}_\tau(R_2+\e,\theta),\\
\omega[\ff + \tau \hh]-\omega[\ff] & \mathsf{m}_\tau(R_2+\e,\theta)\leq r < \mathsf{M}_\tau(R_2+\e,\theta),\\
0 & \mathsf{M}_\tau(R_2+\e,\theta)\leq r \leq r_2,
\end{cases}
\end{equation}
where
\begin{align*}
\mathsf{m}_\tau(r,\theta):&= \min\{\Phi_1[\ff+\tau \hh](r,\theta),\Phi_1[\ff](r,\theta)\},\\
\mathsf{M}_\tau(r,\theta):&= \max\{\Phi_1[\ff+\tau \hh](r,\theta),\Phi_1[\ff](r,\theta)\}.   
\end{align*}
\begin{remark}
Over each region $\{(r,\theta)\in \mathsf{P}_{r_1,r_2}: \mathsf{m}_\tau(R_i+(-1)^j\e,\theta) \leq r < \mathsf{M}_\tau(R_i+(-1)^j\e,\theta)\}$ with $i,j=1,2$ we have that the difference $\omega[\ff + \tau \hh]-\omega[\ff]$ will be 
 \[
 \varpi(\Phi_1^{-1}[\ff+\tau\hh]) \qquad\text{or} \qquad \varpi(\Phi_1^{-1}[\ff]).
 \]
\end{remark}
In addition, we define the auxiliary sets 
\begin{equation}\label{def:setV}
V_\tau(R_i+(-1)^j \e):=\{ (r,\theta)\in \mathsf{P}_{r_1,r_2}: \mathsf{m}_{\tau}(R_i+(-1)^j \e,\theta)< r < \mathsf{M}_{\tau}(R_i+(-1)^j \e,\theta)\}, 
\end{equation}
and
\begin{equation}\label{def:setU}
U_\tau^k:=\{(r,\theta)\in \mathsf{P}_{r_1,r_2}: \mathsf{M}_\tau (R_k-\e,\theta)\leq r \leq \mathsf{m}_\tau(R_k+\e,\theta)\}.
\end{equation}
Then, as an immediate consequence of \eqref{aux:Dfw[f]}-\eqref{aux:difference} and \eqref{def:setV}-\eqref{def:setU} we have that
\[
D_\ff \omega[\ff]\hh-\frac{\omega[\ff+\tau\hh]-\omega[\ff]}{\tau}=0, \qquad (r,\theta)\in \mathsf{P}_{r_1,r_2}\setminus\left(\bigcup_{i,j=1,2}V_\tau(R_i+(-1)^j\e) \cup \bigcup_{k=1,2}U_\tau^k\right).
\]
Hence, we can directly write 
\begin{multline*}
\norm{D_\ff \omega[\ff]\hh-\frac{\omega[\ff+\tau\hh]-\omega[\ff]}{\tau}}_{C^{0,\alpha}(\mathsf{P}_{r_1,r_2})}= \sum_{i,j=1,2} \norm{D_\ff \omega[\ff]\hh-\frac{\omega[\ff+\tau\hh]-\omega[\ff]}{\tau}}_{C^{0,\a}(V_\tau(R_i+(-1)^j\e))}\\
+\sum_{k=1,2}\norm{D_\ff \omega[\ff]\hh-\frac{\omega[\ff+\tau\hh]-\omega[\ff]}{\tau}}_{C^{0,\a}(U_\tau^k)}.
\end{multline*}

Now, we are going to study each of the terms of the previous expression. On the one hand, applying \eqref{aux:Dfw[f]} and \eqref{aux:difference} we obtain
\begin{multline}\label{aux:trozogrande}
\norm{D_\ff \omega[\ff]\hh-\frac{\omega[\ff+\tau\hh]-\omega[\ff]}{\tau}}_{C^{0,\a}(U_\tau^k)}\\
=\norm{\frac{\left(\varpi'\circ\Phi_1^{-1}[\ff]\right)  \left(\hh\circ\Phi^{-1}[\ff]\right)}{1+\left(\p_r\ff\circ\Phi^{-1}[\ff]\right)}-\frac{\varpi\circ \Phi_1^{-1}[\ff]-\varpi\circ \Phi_1^{-1}[\ff+\tau\hh]}{\tau}}_{C^{0,\a}(U_\tau^k)}.
\end{multline}
Here we will give all the details on how to manage the above type of term, which will appear repeatedly in this section.  First of all, from Taylor's expansion, we can ensure that there exists $\tau^\ast\in(0,\tau)$ such that
\begin{multline*}
\varpi\circ \Phi_1^{-1}[\ff+\tau\hh]=\left(\varpi\circ \Phi_1^{-1}[\ff+\tau \hh]\right)\big|_{\tau=0}+\frac{d}{d\tau}\left(\varpi\circ \Phi_1^{-1}[\ff+\tau\hh]\right)\bigg|_{\tau=0} \tau \\
+ \frac{d^2}{d^2\tau}\left(\varpi\circ \Phi_1^{-1}[\ff+\tau\hh]\right)\bigg|_{\tau=\tau^\ast}\frac{\tau^2}{2},
\end{multline*}
with
\begin{equation}\label{aux:dtau1comp}
\frac{d}{d\tau}\left(\varpi\circ \Phi_1^{-1}[\ff+\tau\hh]\right)\bigg|_{\tau=0}=\varpi'\circ  \Phi_1^{-1}[\ff+\tau\hh] \cdot\frac{d}{d\tau}\Phi_1^{-1}[\ff+\tau\hh]\bigg|_{\tau=0},
\end{equation}
and
\begin{multline*}
\frac{d^2}{d^2\tau}\left(\varpi\circ \Phi_1^{-1}[\ff+\tau\hh]\right)\bigg|_{\tau=\tau^\ast}=\varpi'' \circ  \Phi_1^{-1}[\ff+\tau\hh]\cdot \left(\frac{d}{d\tau}\Phi_1^{-1}[\ff+\tau\hh]\right)^2\bigg|_{\tau=\tau^\ast} \\
+ \varpi'\circ  \Phi_1^{-1}[\ff+\tau\hh] \cdot\frac{d^2}{d^2\tau}\Phi_1^{-1}[\ff+\tau\hh]  \bigg|_{\tau=\tau^\ast}.
\end{multline*}
Since $\Phi_1[\ff+\tau\hh](\Phi^{-1}[\ff+\tau\hh](r,\theta))=r$ and $\Phi_2[\ff+\tau\hh](\Phi^{-1}[\ff+\tau\hh](r,\theta))=\theta$, we have
\begin{align*}
\Phi_1^{-1}[\ff+\tau\hh](r,\theta)+\ff(\Phi^{-1}[\ff+\tau\hh](r,\theta))+\tau \hh(\Phi^{-1}[\ff+\tau\hh](r,\theta))=r,\\
\Phi_2^{-1}[\ff+\tau\hh](r,\theta)=\theta,
\end{align*}
and calculating their derivatives with respect to $\tau$-variable, we obtain
\begin{align}
\frac{d}{d\tau}\Phi_1^{-1}[\ff+\tau \hh]&=-\frac{\hh\circ\Phi^{-1}[\ff+\tau\hh]}{1+(\p_r\ff \circ \Phi^{-1}[\ff+\tau\hh])+\tau(\p_r\hh \circ\Phi^{-1}[\ff+\tau\hh])}, \label{aux:1dtauinversa}\\
\frac{d}{d\tau}\Phi_2^{-1}[\ff+\tau \hh]&=0,\label{aux:1dtauinvtheta}
\end{align}
and proceeding similarly
\begin{multline}\label{aux:2dtauinversa}
\left[1+(\p_r\ff \circ \Phi^{-1}[\ff+\tau\hh])+\tau(\p_r\hh \circ\Phi^{-1}[\ff+\tau\hh])\right]\frac{d^2}{d^2\tau}\Phi_1^{-1}[\ff+\tau \hh]\\
+\left[\left(\p_r^2\ff \circ \Phi^{-1}[\ff+\tau\hh]\right)+\tau\left(\p_r^2\hh \circ \Phi^{-1}[\ff+\tau\hh]\right)\right]\left(\frac{d}{d\tau}\Phi_1^{-1}[\ff+\tau \hh]\right)^2\\
+2 \p_r\hh\circ\Phi^{-1}[\ff+\tau\hh]\, \frac{d}{d\tau}\Phi_1^{-1}[\ff+\tau \hh]=0.
\end{multline}
Then, combining the above, we have that \eqref{aux:trozogrande} is just
\[
\norm{D_\ff \omega[\ff]\hh-\frac{\omega[\ff+\tau\hh]-\omega[\ff]}{\tau}}_{C^{0,\a}(U_\tau^k)}=\frac{\tau}{2}\norm{\frac{d^2}{d^2\tau}\left(\varpi\circ \Phi_1^{-1}[\ff+\tau\hh]\right)\bigg|_{\tau=\tau^\ast}}_{C^{0,\a}(U_\tau^k)},
\]
which can be upper bounded, taking $\tau$ small enough and using \eqref{aux:1dtauinversa}-\eqref{aux:2dtauinversa}, as follows
\begin{equation*}
\norm{D_\ff \omega[\ff]\hh-\frac{\omega[\ff+\tau\hh]-\omega[\ff]}{\tau}}_{C^{0,\a}(U_\tau^k)}\leq C(\|\ff\|_{C^{2,\a}(D_\e)},\|\hh\|_{C^{2,\a}(D_\e)},\|\varpi\|_{C^{2,\a}([r_1,r_2])})\tau.
\end{equation*}
In particular, we have proved that
\begin{equation}\label{final:trozogrande}
\lim_{\tau\to 0}\norm{D_\ff \omega[\ff]\hh-\frac{\omega[\ff+\tau\hh]-\omega[\ff]}{\tau}}_{C^{0,\a}(U_\tau^k)}=0.
\end{equation}

On the other hand, since $|V_\tau(R_i+(-1)^j\e)|=O(\tau)$ and we have computed explicitly \eqref{aux:Dfw[f]} and \eqref{aux:difference} we obtain that
\begin{multline}\label{aux:trozomedio}
\norm{D_\ff \omega[\ff]\hh-\frac{\omega[\ff+\tau\hh]-\omega[\ff]}{\tau}}_{C^{0,\a}(V_\tau(R_i+(-1)^j\e))}=  \norm{\frac{\varpi\circ\Phi_1^{-1}[\ff+\tau\hh]}{\tau}}_{C^{0,\a}(V_\tau^{+}(R_i+(-1)^j\e))}\\
+\norm{\frac{\varpi\circ\Phi_1^{-1}[\ff]}{\tau}-\frac{(\varpi'\circ\Phi_1^{-1}[\ff]) (\hh\circ \Phi^{-1}[\ff])}{1+(\p_r\ff\circ\Phi^{-1}[\ff])}}_{C^{0,\a}(V_\tau^{-}(R_i+(-1)^j\e))},
\end{multline}
where we have used the disjoint union given by 
\[
V_\tau(R_i+(-1)^j\e))=V_\tau^{+}(R_i+(-1)^j\e))\cup V_\tau^{-}(R_i+(-1)^j\e)),
\]
with
\begin{align*}
V_\tau^{+}(z):=\{ (r,\theta)\in V_\tau(z): \Phi_1[\ff+\tau\hh](z,\theta)< r < \Phi_1[\ff](z,\theta)\},\\
V_\tau^{-}(z):=\{ (r,\theta)\in V_\tau(z): \Phi_1[\ff](z,\theta)< r < \Phi_1[\ff+\tau\hh](z,\theta)\},
\end{align*}
for $z\in\left\{ R_i+(-1)^j \e : i,j=1,2\right\}$.

Since $\varpi\circ\Phi_1^{-1}[\ff]\equiv 0$ over $V_\tau^{+}(R_i+(-1)^j \e)$ by definition, we directly have that
\[
\norm{\frac{\varpi\circ\Phi_1^{-1}[\ff+\tau\hh]}{\tau}}_{C^{0,\a}(V_\tau^{+}(R_i+(-1)^j\e))}=\norm{\frac{\varpi\circ\Phi_1^{-1}[\ff+\tau\hh]-\varpi\circ\Phi_1^{-1}[\ff]}{\tau}}_{C^{0,\a}(V_\tau^{+}(R_i+(-1)^j\e))}.
\]
Moreover, applying Taylor's expansion as we did before, we can ensure that there exists $\tau^\ast\in(0,\tau)$ such that
\[
\norm{\frac{\varpi\circ\Phi_1^{-1}[\ff+\tau\hh]}{\tau}}_{C^{0,\a}(V_\tau^{+}(R_i+(-1)^j\e))}=\norm{\frac{d}{d\tau}\left(\varpi\circ \Phi_1^{-1}[\ff+\tau\hh]\right)\bigg|_{\tau=\tau^\ast}}_{C^{0,\a}(V_\tau^{+}(R_i+(-1)^j\e))}.
\]
In fact, using \eqref{aux:dtau1comp} and the fact that $C^{0,\a}(V_\tau^{+}(R_i+(-1)^j\e))$  is a Banach algebra we get
\begin{multline*}
\norm{\frac{\varpi\circ\Phi_1^{-1}[\ff+\tau\hh]}{\tau}}_{C^{0,\a}(V_\tau^{+}(R_i+(-1)^j\e))}\leq \norm{\varpi'\circ\Phi_1^{-1}[\ff+\tau^\ast\hh]}_{C^{0,\a}(V_\tau^{+}(R_i+(-1)^j\e))}\\
\norm{\frac{d}{d\tau}\Phi_1^{-1}[\ff+\tau\hh]\bigg|_{\tau=\tau^\ast}}_{C^{0,\a}(V_\tau^{+}(R_i+(-1)^j\e))},
\end{multline*}
with (see \eqref{aux:1dtauinversa})
\[
\norm{\frac{d}{d\tau}\Phi_1^{-1}[\ff+\tau\hh]\bigg|_{\tau=\tau^\ast}}_{C^{0,\a}(V_\tau^{+}(R_i+(-1)^j\e))}\lesssim \frac{\norm{\hh}_{C^{0,\a}(D_\e)}}{1-\left(\norm{\ff}_{C^{1,\a}(D_\e)}+\tau^\ast \norm{\hh}_{C^{1,\a}(D_\e)}\right)},
\]
and
\[
\lim_{\tau\to 0}\norm{\varpi'\circ\Phi_1^{-1}[\ff+\tau^\ast\hh]}_{C^{0,\a}(V_\tau^{+}(R_i+(-1)^j\e))}=0,
\]
where, in the limit above, we have used the fact that $\varpi$ is smooth with $\varpi\circ\Phi_1^{-1}[\ff]\equiv 0$ over $V_\tau^{+}(R_i+(-1)^j \e)$ and $\tau^\ast\in(0,\tau).$

Hence, we have concluded that
\begin{equation}\label{aux:trozomedio1}
\lim_{\tau\to 0}\norm{\frac{\varpi\circ\Phi_1^{-1}[\ff+\tau\hh]}{\tau}}_{C^{0,\a}(V_\tau^{+}(R_i+(-1)^j\e))}=0,
\end{equation}
and proceeding in a similar way, using the same type of arguments, we can also conclude that
\begin{equation}\label{aux:trozomedio2}
\lim_{\tau\to 0}\norm{\frac{\varpi\circ\Phi_1^{-1}[\ff]}{\tau}-\frac{(\varpi'\circ\Phi_1^{-1}[\ff]) (\hh\circ \Phi^{-1}[\ff])}{1+(\p_r\ff\circ\Phi^{-1}[\ff])}}_{C^{0,\a}(V_\tau^{-}(R_i+(-1)^j\e))}=0.
\end{equation}
Applying \eqref{aux:trozomedio1}-\eqref{aux:trozomedio2} into \eqref{aux:trozomedio} we finally get our desired goal. That is,
\begin{equation}\label{final:trozomedio}
\lim_{\tau\to 0} \norm{D_\ff \omega[\ff]\hh-\frac{\omega[\ff+\tau\hh]-\omega[\ff]}{\tau}}_{C^{0,\a}(V_\tau(R_i+(-1)^j\e))}=0.
\end{equation}

To sum up, combining \eqref{final:trozomedio} and \eqref{final:trozogrande} we have proved that
\[
\lim_{\tau\to 0}\norm{D_\ff \omega[\ff]\hh-\frac{\omega[\ff+\tau\hh]-\omega[\ff]}{\tau}}_{C^{0,\a}(\mathsf{P}_{r_1,r_2})}=0,
\]
and consequently, going backwards in our argument, we have proven our initial objective \eqref{aux:dbarpsi}, i.e.,
\[
\lim_{\tau\to 0} \left\Vert \frac{\bar{\psi}[\ff+\tau\hh]-\bar{\psi}[\ff]}{\tau} -D_\ff\bar{\psi}[\ff]\hh \right\Vert_{C^{2,\a}(D_{\ep})}=0,
\]
or equivalently,
\[
\lim_{\tau\to 0} \left\Vert \frac{F[\l,\ff+\tau\hh]-F[\l,\ff]}{\tau} -D_\ff F[\l,\ff]\hh\right\Vert_{C^{2,\a}(D_{\ep})}=0.
\]

This shows the existence of G\^ateaux derivative and now we intend to prove the continuity of the map  $\ff\to D_\ff F[\l,\ff]$
from $B_\delta(X(D_{\ep}))$ to the space $\mathsf{L}(X(D_{\ep}),Y(D_{\ep}))$ of  bounded linear operators from $X(D_{\ep})$ to $Y(D_{\ep})$.  This is a consequence of the following estimate:
\begin{equation}\label{continuity}
\|D_\ff F[\l,\ff']\hh-D_\ff  F[\l,\ff'']\hh\|_{C^{2,\alpha}(D_{\ep})}\lesssim \|\ff'-\ff''\|_{C^{2,\a}(D_{\ep})},
\end{equation}
for any pair $\ff',\ff'' \in \mathbb{B}_{\d}(X(D_{\ep}))$ and $\hh\in X(D_{\ep}).$

By virtue of \eqref{def:dF} we get
\begin{multline*}
D_\ff F[\l,\ff']\hh-D_\ff  F[\l,\ff'']\hh=\l(\ff'-\ff'')\hh-\left(D_\ff\bar{\psi}[\ff']\hh-D_\ff\bar{\psi}[\ff'']\hh\right)\\
-\frac{1}{2\pi}\int_\T \l(\ff'-\ff'')\hh-\left(D_\ff\bar{\psi}[\ff']\hh-D_\ff\bar{\psi}[\ff'']\hh\right)\ds,
\end{multline*}
and as a direct consequence we obtain that to prove \eqref{continuity}  it is enough to verify that
\begin{equation}\label{aux:continuity}
\|D_\ff\bar{\psi}[\ff']\hh-D_\ff\bar{\psi}[\ff'']\hh\|_{C^{2,\alpha}(D_{\ep})}\lesssim \|\ff'-\ff''\|_{C^{2,\a}(D_{\ep})}.
\end{equation}
Furthermore, taking into account the explicit expression of $D_\ff\bar{\psi}[\ff]\hh$ given by \eqref{def:dbarpsi} and by adding and subtracting appropriate terms we can write
\begin{multline*}
D_\ff\bar{\psi}[\ff']\hh-D_\ff\bar{\psi}[\ff'']\hh=D_\ff \psi [\ff']\hh\circ \Phi[\ff']-D_\ff \psi[\ff'']\hh  \circ\Phi[\ff'']\pm D_\ff \psi[\ff'']\hh\circ \Phi[\ff']\\
+\p_r\psi[\ff']\circ\Phi[\ff']-\p_r\psi[\ff'']\circ\Phi[\ff''] \pm \p_r\psi[\ff'']\circ\Phi[\ff'].
\end{multline*}
Next, we have to study each of the following terms
\begin{equation}\label{aux:trozosstability_1}
\norm{D_\ff \psi [\ff'']\hh\circ\Phi[\ff']-D_\ff \psi [\ff'']\hh\circ\Phi[\ff'']}_{C^{2,\a}(D_\e)},\quad \norm{\p_r \psi [\ff'']\circ\Phi[\ff']-\p_r \psi [\ff'']\circ\Phi[\ff'']}_{C^{2,\a}(D_\e)},
\end{equation}
and
\begin{equation}\label{aux:trozosstability_2}
\norm{\left(\p_r\psi[\ff']-\p_r\psi[\ff'']\right)\circ\Phi[\ff']}_{C^{2,\a}(D_\e)}, \quad \|\left(D_\ff \psi[\ff']\hh -D_\ff \psi [\ff'']\hh\right)\circ\Phi[\ff']\|_{C^{2,\alpha}(D_{\ep})}.
\end{equation}

Let us start with \eqref{aux:trozosstability_1}, we recall that $D_\ff \psi [\ff'']\hh, \p_r\psi[\ff'']\in C^{3,\a}(\mathsf{P}_{r_1,r_2})$. Then, taking $\gg$ equal to $D_\ff \psi [\ff'']\hh$ or $\p_r\psi[\ff'']$ we get
\begin{multline*}
\gg\circ \Phi[\ff'](r,\theta)-\gg\circ\Phi[\ff''](r,\theta)=\gg(r+\ff'(r,\theta),\theta)-\gg(r+\ff''(r,\theta),\theta)\\
=\int_0^1\p_s \left[\gg(r+s\ff'(r,\theta)+(1-s)\ff''(r,\theta),\theta)\right]\ds\\
=\left(\ff'(r,\theta)-\ff''(r,\theta)\right)\int_0^1\p_r \gg(r+s\ff'(r,\theta)+(1-s)\ff''(r,\theta),\theta)\ds.
\end{multline*}
Since $\gg\in C^{3,\a}(\mathsf{P}_{r_1,r_2})$ and $\ff',\ff''\in C^{2,\a}(D_\e)$ we directly have that 
\[
\p_r \gg(r+s\ff'(r,\theta)+(1-s)\ff''(r,\theta),\theta)\in C^{2,\a}(D_\e), \qquad \forall s\in(0,1),
\]
and consequently, using the fact that $C^{2,\a}(D_\e)$ is Banach algebra, we can conclude that
\[
\|\gg\circ\Phi[\ff']-\gg\circ\Phi[\ff'']\|_{C^{2,\alpha}(D_{\ep})}\lesssim \|\ff'-\ff''\|_{C^{2,\a}(D_{\ep})}.
\]

The first factor of \eqref{aux:trozosstability_2} can be handled as follows
\begin{multline*}
\norm{\left(\p_r\psi[\ff']-\p_r\psi[\ff'']\right)\circ\Phi[\ff']}_{C^{2,\a}(D_\e)} \leq \norm{\p_r\psi[\ff']-\p_r\psi[\ff'']}_{C^{2,\a}(\mathsf{P}_{r_1,r_2})}\\
\leq \norm{\psi[\ff']-\psi[\ff'']}_{C^{3,\a}(\mathsf{P}_{r_1,r_2})}
\lesssim \norm{\omega[\ff']-\omega[\ff'']}_{C^{1,\a}(\mathsf{P}_{r_1,r_2})},
\end{multline*}
where, in the last step, we have used the fact that $\psi[\ff']-\psi[\ff'']$ solves the elliptic-type problem 
\begin{equation*}
\left\{
\begin{aligned}
\Delta_B \left(\psi[\ff']-\psi[\ff'']\right)(r,\theta)&=\left(\omega[\ff']-\omega[\ff'']\right)(r,\theta),\qquad (r,\theta)\in \mathsf{P}_{r_1,r_2},\\
\left(\psi[\ff']-\psi[\ff'']\right)(r_1,\theta)&=0, \\
\left(\psi[\ff']-\psi[\ff'']\right)(r_2,\theta)&=0.
\end{aligned}
\right.
\end{equation*}
Proceeding in the same way as we did before we get $\norm{\omega[\ff']-\omega[\ff'']}_{C^{1,\a}(\mathsf{P}_{r_1,r_2})}\leq \|\ff'-\ff''\|_{C^{2,\a}(D_{\ep})}.$

Then, combining all the above we have proved that
\[
\norm{D_\ff\bar{\psi}[\ff']\hh-D_\ff\bar{\psi}[\ff'']\hh}_{C^{2,\a}(D_\e)}\lesssim \|\ff'-\ff''\|_{C^{2,\a}(D_{\ep})} + \|D_\ff \psi[\ff']\hh-D_\ff \psi [\ff'']\hh\|_{C^{2,\alpha}(\mathsf{P}_{r_1,r_2})},
\]
and we conclude that \eqref{aux:continuity} reduces to check that
\begin{equation}\label{aux:continuity1}
\|D_\ff \psi[\ff']\hh-D_\ff \psi [\ff'']\hh\|_{C^{2,\alpha}(\mathsf{P}_{r_1,r_2})}\lesssim \|\ff'-\ff''\|_{C^{2,\a}(D_{\ep})}.
\end{equation}
Since $D_\ff \psi[\ff']\hh$ and $D_\ff \psi[\ff'']\hh$ solve an elliptic-type problem like \eqref{aux:DeltaDF}, we have that their difference satisfies
\begin{equation*}
\left\{
\begin{aligned}
\Delta_B \left(D_\ff \psi[\ff']\hh-D_\ff \psi[\ff'']\hh\right)(r,\theta)&=\left(D_\ff \omega[\ff']\hh-D_\ff \omega[\ff'']\hh\right)(r,\theta),\qquad (r,\theta)\in \mathsf{P}_{r_1,r_2},\\
\left(D_\ff \psi[\ff']\hh-D_\ff \psi[\ff'']\hh\right)(r_1,\theta)&=0, \\
\left(D_\ff \psi[\ff']\hh-D_\ff \psi[\ff'']\hh\right)(r_2,\theta)&=0.
\end{aligned}
\right.
\end{equation*}
Applying classical elliptic regularity theory we have 
\begin{equation}\label{aux:continuity2}
\|D_\ff \psi[\ff']\hh-D_\ff \psi [\ff'']\hh\|_{C^{2,\alpha}(\mathsf{P}_{r_1,r_2})}\lesssim \norm{D_\ff \omega[\ff']\hh-D_\ff \omega[\ff'']\hh}_{C^{0,\a}(\mathsf{P}_{r_1,r_2})}.
\end{equation}
Therefore, combining \eqref{aux:continuity1} and \eqref{aux:continuity2}, our original problem is reduced to prove
\begin{equation}\label{aux:continuityfinal}
\norm{D_\ff \omega[\ff']\hh-D_\ff \omega[\ff'']\hh}_{C^{0,\a}(\mathsf{P}_{r_1,r_2})} \lesssim \|\ff'-\ff''\|_{C^{2,\a}(D_{\ep})}.
\end{equation}

Since we have an explicit expression for $D_\ff \omega[\ff]\hh$ given by \eqref{aux:Dfw[f]}. We start analyzing in detail their difference. That is,
\begin{multline*}
D_\ff \omega[\ff']\hh-D_\ff \omega[\ff'']\hh=\\
\begin{cases}
0 \qquad & r_1 \leq r <\mathrm{m}(R_1-\e,\theta),\\
D_\ff \omega[\ff']\hh-D_\ff \omega[\ff'']\hh \quad & \mathrm{m}(R_1-\e,\theta)\leq r < \mathrm{M}(R_1-\e,\theta),\\
\frac{\left(\varpi'\circ\Phi_1^{-1}[\ff'']\right)  \left(\hh\circ\Phi^{-1}[\ff'']\right)}{1+\left(\p_r\ff''\circ\Phi^{-1}[\ff'']\right)}-\frac{\left(\varpi'\circ\Phi_1^{-1}[\ff']\right)  \left(\hh\circ\Phi^{-1}[\ff']\right)}{1+\left(\p_r\ff'\circ\Phi^{-1}[\ff']\right)} \quad &\mathrm{M}(R_1-\e,\theta) \leq r < \mathrm{m}(R_1+\e,\theta),\\
D_\ff \omega[\ff']\hh-D_\ff \omega[\ff'']\hh \quad & \mathrm{m}(R_1+\e,\theta)\leq r < \mathrm{M}(R_1+\e,\theta),\\
0 & \mathrm{M}(R_1+\e,\theta)\leq r < \mathrm{m}(R_2-\e,\theta),\\
D_\ff \omega[\ff']\hh-D_\ff \omega[\ff'']\hh \quad & \mathrm{m}(R_2-\e,\theta)\leq r < \mathrm{M}(R_2-\e,\theta),\\
\frac{\left(\varpi'\circ\Phi_1^{-1}[\ff'']\right)  \left(\hh\circ\Phi^{-1}[\ff'']\right)}{1+\left(\p_r\ff''\circ\Phi^{-1}[\ff'']\right)}-\frac{\left(\varpi'\circ\Phi_1^{-1}[\ff']\right)  \left(\hh\circ\Phi^{-1}[\ff']\right)}{1+\left(\p_r\ff'\circ\Phi^{-1}[\ff']\right)} \qquad &\mathrm{M}(R_2-\e,\theta) \leq r < \mathrm{m}(R_2+\e,\theta),\\
D_\ff \omega[\ff']\hh-D_\ff \omega[\ff'']\hh \quad & \mathrm{m}(R_2+\e,\theta)\leq r < \mathrm{M}(R_2+\e,\theta),\\
0 & \mathrm{M}(R_2+\e,\theta)\leq r \leq  r_2
\end{cases}   
\end{multline*}
where
\begin{align*}
    \mathrm{m}(r,\theta)&:=\min\{\Phi_1[\ff'](r,\theta),\Phi_1[\ff''](r,\theta)\},\\
    \mathrm{M}(r,\theta)&:=\max\{\Phi_1[\ff'](r,\theta),\Phi_1[\ff''](r,\theta)\}.
\end{align*}
\begin{remark}\label{auxremark:Df'-Df''}
Over each region $\{(r,\theta)\in \mathsf{P}_{r_1,r_2}: \mathrm{m}(R_i+(-1)^j\e,\theta) \leq r < \mathrm{M}(R_i+(-1)^j\e,\theta)\}$ with $i,j=1,2$ we have that the difference $D_\ff \omega[\ff']\hh-D_\ff \omega[\ff'']\hh$ will be 
 \[
-\frac{\left(\varpi'\circ\Phi_1^{-1}[\ff']\right)  \left(\hh\circ\Phi^{-1}[\ff']\right)}{1+\left(\p_r\ff'\circ\Phi^{-1}[\ff']\right)} \qquad\text{or} \qquad -\frac{\left(\varpi'\circ\Phi_1^{-1}[\ff'']\right)  \left(\hh\circ\Phi^{-1}[\ff'']\right)}{1+\left(\p_r\ff''\circ\Phi^{-1}[\ff'']\right)}.
 \]
\end{remark}
We can proceed as we did before and define the auxiliary sets
\begin{equation*}
V(R_i+(-1)^j \e):=\{ (r,\theta)\in \mathsf{P}_{r_1,r_2}: \mathrm{m}(R_i+(-1)^j \e,\theta)< r < \mathrm{M}(R_i+(-1)^j \e,\theta)\}, 
\end{equation*}
and
\begin{equation*}
U^k:=\{(r,\theta)\in \mathsf{P}_{r_1,r_2}: \mathrm{M} (R_k-\e,\theta)\leq r \leq \mathrm{m}(R_k+\e,\theta)\}.
\end{equation*}
Then, as an immediate consequence we get 
\begin{multline}\label{aux:finalcont}
\|D_\ff\bar{\psi}[\ff']\hh-D_\ff\bar{\psi}[\ff'']\hh\|_{C^{0,\alpha}(\mathsf{P}_{r_1,r_2})}=\sum_{i,j=1,2}\norm{D_\ff\bar{\psi}[\ff']\hh-D_\ff\bar{\psi}[\ff'']\hh}_{C^{0,\a}(V(R_i+(-1)^j \e))}\\
+\sum_{k=1,2}\norm{D_\ff\bar{\psi}[\ff']\hh-D_\ff\bar{\psi}[\ff'']\hh}_{C^{0,\a}(U^k)}.
\end{multline}

Now, we are going to study each of the terms of the previous expression. On the one hand, we have
\[
\norm{D_\ff\bar{\psi}[\ff']\hh-D_\ff\bar{\psi}[\ff'']\hh}_{C^{0,\a}(U^k)}=\norm{\frac{\left(\varpi'\circ\Phi_1^{-1}[\ff'']\right)  \left(\hh\circ\Phi^{-1}[\ff'']\right)}{1+\left(\p_r\ff''\circ\Phi^{-1}[\ff'']\right)}-\frac{\left(\varpi'\circ\Phi_1^{-1}[\ff']\right)  \left(\hh\circ\Phi^{-1}[\ff']\right)}{1+\left(\p_r\ff'\circ\Phi^{-1}[\ff']\right)}}_{C^{0,\a}(U^k)},
\]
which can be decomposed and upper bounded (adding and subtracting appropriate terms) as follows 
\begin{multline*}
\norm{D_\ff\bar{\psi}[\ff']\hh-D_\ff\bar{\psi}[\ff'']\hh}_{C^{0,\a}(U^k)}\\
\leq  \norm{\left(\varpi'\circ\Phi_1^{-1}[\ff'']\right)  \left(\hh\circ\Phi^{-1}[\ff'']\right)\left[\frac{1}{1+\left(\p_r\ff''\circ\Phi^{-1}[\ff'']\right)}-\frac{1}{1+\left(\p_r\ff'\circ\Phi^{-1}[\ff'']\right)}\right]}_{C^{0,\a}(U^k)}\\
+\norm{\left(\varpi'\circ\Phi_1^{-1}[\ff'']\right)  \left(\hh\circ\Phi^{-1}[\ff'']\right)\left[\frac{1}{1+\left(\p_r\ff'\circ\Phi^{-1}[\ff'']\right)}-\frac{1}{1+\left(\p_r\ff'\circ\Phi^{-1}[\ff']\right)}\right]}_{C^{0,\a}(U^k)}\\
+\norm{\frac{\left(\varpi'\circ\Phi_1^{-1}[\ff'']\right)  \left(\hh\circ\Phi^{-1}[\ff'']\right)-\left(\varpi'\circ\Phi_1^{-1}[\ff']\right)  \left(\hh\circ\Phi^{-1}[\ff']\right)}{1+\left(\p_r\ff'\circ\Phi^{-1}[\ff']\right)}}_{C^{0,\a}(U^k)}.
\end{multline*}
Since H\"{o}lder space $C^{0,\a}(U^k)$ is a Banach algebra we get
\begin{multline*}
\norm{D_\ff\bar{\psi}[\ff']\hh-D_\ff\bar{\psi}[\ff'']\hh}_{C^{0,\a}(U^k)}\leq 
\frac{1}{1-\norm{\ff'}_{C^{1,\a}(D_\e)}}\norm{(\varpi'\hh)\circ\Phi^{-1}[\ff'']- (\varpi'\hh)\circ\Phi^{-1}[\ff']}_{C^{0,\a}(U^k)}\\
+\frac{\norm{\varpi}_{C^{1,\a}([r_1,r_2])}\norm{\hh}_{C^{0,\a}(D_\e)}}{1-\max\{\norm{\ff'}_{C^{1,\a}(D_\e)}, \norm{\ff''}_{C^{1,\a}(D_\e)}\}}\norm{\left(\p_r\ff'-\p_r\ff''\right)\circ\Phi^{-1}[\ff'']}_{C^{0,\a}(U^k)}\\
+\frac{\norm{\varpi}_{C^{1,\a}([r_1,r_2])}\norm{\hh}_{C^{0,\a}(D_\e)}}{1- \norm{\ff''}_{C^{1,\a}(D_\e)}}\norm{\p_r\ff'\circ\Phi^{-1}[\ff']-\p_r\ff'\circ\Phi^{-1}[\ff'']}_{C^{0,\a}(U^k)}.
\end{multline*}
In addition, we have directly that 
\begin{align*}
\norm{(\varpi'\hh)\circ\Phi^{-1}[\ff'']- (\varpi'\hh)\circ\Phi^{-1}[\ff']}_{C^{0,\a}(U^k)}&\leq  \norm{\varpi'\hh}_{C^{1,\a}(D_\e)}\norm{\Phi^{-1}[\ff'']-\Phi^{-1}[\ff']}_{C^{0,\a}(U^k)}, \\
\norm{\p_r\ff'\circ\Phi^{-1}[\ff']-\p_r\ff'\circ\Phi^{-1}[\ff'']}_{C^{0,\a}(U^k)}&\leq \norm{\ff'}_{C^{2,\a}(D_\e)}\norm{\Phi^{-1}[\ff'']-\Phi^{-1}[\ff']}_{C^{0,\a}(U^k)},\\
\norm{\left(\p_r\ff'-\p_r\ff''\right)\circ\Phi^{-1}[\ff'']}_{C^{0,\a}(U^k)}&\leq \norm{\ff'-\ff''}_{C^{1,\a}(D_\e)}.
\end{align*}
To conclude, we only need to see that
\begin{equation}\label{aux:Phif-Phif}
\norm{\Phi^{-1}[\ff'']-\Phi^{-1}[\ff']}_{C^{0,\a}(U^k)}\leq \norm{\ff'-\ff''}_{C^{0,\a}(D_\e)}.
\end{equation}
In order to prove this we apply once again Taylor's expansion, which ensures that there exists $\tau^\ast\in(0,1)$ such that
\begin{align*}
\Phi^{-1}[\ff'']-\Phi^{-1}[\ff']&=\Phi^{-1}[\ff''+\tau(\ff'-\ff'')]\big|_{\tau=0}-\Phi^{-1}[\ff''+\tau(\ff'-\ff'')]\big|_{\tau=1}\\
&=\frac{d}{d\tau}\Phi^{-1}[\ff''+\tau(\ff'-\ff'')]\big|_{\tau=\tau^\ast} \,(\ff''-\ff').
\end{align*}
Moreover, using \eqref{aux:1dtauinversa}-\eqref{aux:1dtauinvtheta} we get
\[
\norm{\Phi^{-1}[\ff'']-\Phi^{-1}[\ff']}_{C^{0,\a}(U^k)}\leq \frac{\norm{\ff'-\ff''}_{C^{0,\a}(D_\e)}}{1-\max\{\norm{\ff''}_{C^{1,\a}(D_\e)},\norm{\ff'-\ff''}_{C^{1,\a}(D_\e)}\}}\norm{\ff'-\ff''}_{C^{0,\a}(D_\e)}.
\]
Therefore, combining all these calculations we have proved that
\begin{equation}\label{aux:diffDF_U}
\norm{D_\ff\bar{\psi}[\ff']\hh-D_\ff\bar{\psi}[\ff'']\hh}_{C^{0,\a}(U^k)} \lesssim \norm{\ff'-\ff''}_{C^{1,\a}(D_\e)}.
\end{equation}

On the other hand, taking into account Remark \ref{auxremark:Df'-Df''} and proceeding as before, we have that
\begin{multline*}
\norm{D_\ff\bar{\psi}[\ff']\hh-D_\ff\bar{\psi}[\ff'']\hh}_{C^{0,\a}(V(R_i+(-1)^j \e))}=\norm{\frac{\left(\varpi'\circ\Phi_1^{-1}[\ff']\right)  \left(\hh\circ\Phi^{-1}[\ff']\right)}{1+\left(\p_r\ff'\circ\Phi^{-1}[\ff']\right)}}_{C^{0,\a}(V^+(R_i+(-1)^j \e))}\\
+\norm{\frac{\left(\varpi'\circ\Phi_1^{-1}[\ff'']\right)  \left(\hh\circ\Phi^{-1}[\ff'']\right)}{1+\left(\p_r\ff''\circ\Phi^{-1}[\ff'']\right)}}_{C^{0,\a}(V^-(R_i+(-1)^j \e))},
\end{multline*}
where we have used the disjoint union given by
\[
V(R_i+(-1)^j \e))=V^+(R_i+(-1)^j \e))\cup V^-(R_i+(-1)^j \e)),
\]
with
\begin{align*}
V^{+}(z):=\{ (r,\theta)\in V(z): \mathrm{m}(z,\theta)= \Phi_1[\ff'](z,\theta)\},\\
V^{-}(z):=\{ (r,\theta)\in V(z): \mathrm{M}(z,\theta)= \Phi_1[\ff'](z,\theta)\},
\end{align*}
for $z\in\left\{ R_i+(-1)^j \e : i,j=1,2\right\}$.

Since $\varpi'\circ\Phi_1^{-1}[\ff'']\equiv 0$ over $V^{+}(R_i+(-1)^j \e)$ by definition, we directly have that
\begin{multline*}
\norm{\frac{\left(\varpi'\circ\Phi_1^{-1}[\ff']\right)  \left(\hh\circ\Phi^{-1}[\ff']\right)}{1+\left(\p_r\ff'\circ\Phi^{-1}[\ff']\right)}}_{C^{0,\a}(V^+(R_i+(-1)^j \e))}\\
=\norm{\frac{\left(\varpi'\circ\Phi_1^{-1}[\ff']-\varpi'\circ\Phi_1^{-1}[\ff'']\right)  \left(\hh\circ\Phi^{-1}[\ff']\right)}{1+\left(\p_r\ff'\circ\Phi^{-1}[\ff']\right)}}_{C^{0,\a}(V^+(R_i+(-1)^j \e))}.
\end{multline*}
Now, using the fact that H\"{o}lder space $C^{0,\a}(V^+(R_i+(-1)^j \e))$ is a Banach algebra together with the analogous of \eqref{aux:Phif-Phif} we get
\begin{multline*}
\norm{\frac{\left(\varpi'\circ\Phi_1^{-1}[\ff']\right)  \left(\hh\circ\Phi^{-1}[\ff']\right)}{1+\left(\p_r\ff'\circ\Phi^{-1}[\ff']\right)}}_{C^{0,\a}(V^+(R_i+(-1)^j \e))}\\
\leq \norm{\varpi'}_{C^{1,\a}(D_\e)} \norm{\Phi^{-1}[\ff']-\Phi^{-1}[\ff'']}_{C^{0,\a}(V^+(R_i+(-1)^j \e))}\frac{\norm{\hh}_{C^{0,\a}(D_\e)}}{1-\norm{\ff'}_{C^{1,\a}(D_\e)}}\\
\leq 
\norm{\varpi'}_{C^{1,\a}(D_\e)} \norm{\ff'-\ff''}_{C^{0,\a}(D_\e)}\frac{\norm{\hh}_{C^{0,\a}(D_\e)}}{1-\norm{\ff'}_{C^{1,\a}(D_\e)}}.
\end{multline*}
Proceeding similarly with straightforward modifications we can conclude the same result over the domain $V^-(R_i+(-1)^j \e)$.  As a immediate consequence, combining both, we have proved that
\begin{equation}\label{aux:diffDF_V}
\norm{D_\ff\bar{\psi}[\ff']\hh-D_\ff\bar{\psi}[\ff'']\hh}_{C^{0,\a}(V(R_i+(-1)^j \e))}\lesssim \norm{\ff'-\ff''}_{C^{0,\a}(D_\e)}.
\end{equation}

In summary, by putting together \eqref{aux:finalcont}, \eqref{aux:diffDF_U} and \eqref{aux:diffDF_V} we have obtained \eqref{aux:continuityfinal}. So, backing up our argument, we have proved our initial objective \eqref{continuity}, i.e.,
\[
\|D_\ff F[\l,\ff']\hh-D_\ff  F[\l,\ff'']\hh\|_{C^{2,\alpha}(D_{\ep})}\lesssim \|\ff'-\ff''\|_{C^{2,\a}(D_{\ep})},
\]
for any pair $\ff',\ff'' \in \mathbb{B}_{\d}(X(D_{\ep}))$ and $\hh\in X(D_{\ep}).$

Consequently, we have obtained that the G\^ateaux derivatives are continuous
with respect to the strong topology and hence they are, in fact, Fr\'echet derivatives. Therefore, we can conclude that the Fr\'{e}chet derivative exists and coincides with the G\^ateaux derivative. See \cite{DM_book, Gateaux-Frechet} for more details.
\end{proof}

\subsection{An explicit expression for the differential at the origin}\label{s:explictDF0}
Recall that the linearization of \eqref{def:Functional} around $\ff=0,$ thanks to the expression \eqref{def:dF}, is given by
\[
D_\ff F[\l,0]\hh(r,\theta)=\l r \hh(r,\theta) + D_\ff\bar{\psi}[0]\hh(r,\theta)\\
-\frac{1}{2\pi}\int_\T \l r \hh(r,s) + D_\ff\bar{\psi}[0]\hh(r,s)\ds.
\]
Evaluating \eqref{def:dbarpsi} at $\ff=0$ we have
\[
D_\ff \bar{\psi}[0]\hh(r,\theta)=D_\ff \psi[0]\hh(r,\theta) + \p_r \psi[0](r,\theta) \hh(r,\theta).
\]
Recall that $\psi[0]$ is the solution of \eqref{aux:DBpsi}-\eqref{aux:omega[f]} for $\ff=0$. In this particular case, the system is reduced to a system that depends only on the radial variable. That is, $\psi[0](r,\theta)\equiv \psi[0](r)$ solving
\begin{equation*}
\left\{
\begin{aligned}
-\left(\psi[0]''(r)+\frac{1}{r}\psi[0]'(r)\right)&=\varpi(r),\qquad r\in[r_1,r_2],\\
\psi[0](r_1)&=0, \\
\psi[0](r_2)&=\gamma,
\end{aligned}
\right.
\end{equation*}
Likewise, we recall that $D_\ff \psi[0]\hh$ is the solution of \eqref{aux:DeltaDF}-\eqref{aux:Dfw[f]} for $\ff=0$.  That is,
\begin{equation*}
\left\{
\begin{aligned}
\Delta_B \left(D_\ff \psi[0]\hh\right)(r,\theta)&=D_\ff \omega[0]\hh(r,\theta),\qquad (r,\theta)\in \mathsf{P}_{r_1,r_2},\\
D_\ff \psi[0]\hh(r_1,\theta)&=0, \\
D_\ff \psi[0]\hh(r_2,\theta)&=0,
\end{aligned}
\right.
\end{equation*}
with
\begin{equation*}
D_\ff \omega[0]\hh (r,\theta)=\begin{cases}
0 \qquad & r_1\leq r < R_1-\e,\\
-\varpi'(r) \hh(r,\theta) \qquad & R_1-\e \leq r \leq R_1+\e,\\
0 & R_1+\e < r < R_2-\e,\\
-\varpi'(r) \hh(r,\theta) \qquad & R_2-\e \leq r \leq R_2+\e,\\
0 & R_2+\e < r\leq r_2.
\end{cases}
\end{equation*}

Since $\hh\in X(D_\e)$, in particular $\hh$ is mean zero in the angular variable. Then, we can deduce that $D_\ff \omega[0]\hh$ is also mean zero in the angular variable and consequently also 
$D_\ff \psi[0]\hh$. For this reason, we can conclude that $\int_\T \l r \hh(r,s) + D_\ff\bar{\psi}[0]\hh(r,s)\ds=0$ and finally the linearization around $\ff=0$ is simply 
\[
D_\ff F[\l,0]\hh(r,\theta)=\l r \hh(r,\theta) + D_\ff\bar{\psi}[0]\hh(r,\theta).
\]

Up to this point we have proved Hypothesis 1 and 2 of C-R theorem \ref{th:CR}. The bulk and main task of the manuscript will be to prove Hypothesis 3. In order to do that, we need an explicit and manageable expression of the linear operator.

\section{Analysis of the linear operator}\label{s:linear}

Let $\hh\in X(D_\e)$ with expansion
\[
\hh(r,\theta)=\sum_{n\geq 1}h_n(r)\cos(n\theta), \qquad (r,\theta)\in D_\e.
\]
Above, see Section \ref{s:explictDF0}, we have proved that
\[
D_\ff F[\lambda,0]\hh (r,\theta)=\lambda r \hh(r,\theta) + D_\ff\bar{\psi}[0]\hh(r,\theta),
\]
with
\[
D_\ff\bar{\psi}[0]\hh(r,\theta)=\phi'(r)\hh(r,\theta)+ \Phi(r,\theta),
\]
where $\phi\equiv\phi(r)$ solves
\begin{equation}\label{sys:Phi_0}
\left\{
\begin{aligned}
-\left(\phi''(r)+\frac{1}{r}\phi'(r)\right)&=\varpi(r),\qquad r\in[r_1,r_2],\\
\phi(r_1)&=0, \\
\phi(r_2)&=\gamma,
\end{aligned}
\right.
\end{equation}
and $\Phi \equiv \Phi(r,\theta)$ solves
\begin{equation}\label{sys:Phi_1}
\left\{
\begin{aligned}
\Delta_B \Phi(r,\theta)&=\varpi'(r)h(r,\theta),\qquad (\rho,\theta)\in \mathsf{P}_{r_1,r_2}\\
\Phi(r_1,\theta)&=0, \\
\Phi(r_2,\theta)&=0.
\end{aligned}
\right.
\end{equation}
Here, with a little abuse of notation, since $\supp(\varpi')=I_\e^{R_1,R_2}$ we will understand that 
\[
\varpi'(r)h(r,\theta)=
\begin{cases}
    \varpi'(r)h(r,\theta), \qquad &(r,\theta)\in D_\e,\\
    0, &(r,\theta)\in\mathsf{P}_{r_1,r_2}\setminus D_\e.
\end{cases}
\]
Notice that $\hh$ is not defined in $\mathsf{P}_{r_1,r_2}\setminus D_\e.$

\subsection{Solving the boundary value problems \eqref{sys:Phi_0} and \eqref{sys:Phi_1}}
For general $\varpi$, the solution of \eqref{sys:Phi_0} is given by
\begin{equation}\label{def:Phi_0}
\phi(r)= C_\varpi \log\left(\frac{r}{r_1}\right)-\int_{r_1}^{r}\frac{1}{s}\left(\int_{r_1}^{s} t \varpi(t)\dt\right)\ds,   
\end{equation}
with
\begin{equation}\label{def:ctePhi_0}
C_\varpi=\log^{-1}\left(\frac{r_2}{r_1}\right)\left[\gamma+\int_{r_1}^{r_2}\frac{1}{s}\left(\int_{r_1}^{s} t \varpi(t)\dt\right)\ds\right].
\end{equation}
Similarly, we also need an explicit expression for the system \eqref{sys:Phi_1}. In order to do that, taking the Fourier transform  in the angular variable and using the expansion
\begin{equation}\label{def:sumPhi}
\Phi(r,\theta)=\sum_{n\geq 1} \Phi_n(r)\cos(n\theta),
\end{equation}
we have (see Appendix \ref{app:green}) that
\begin{equation}\label{def:Phi_n}
\Phi_n(r)=\frac{1}{n}\int_{r_1}^{r} s \varpi'(s)\sn(r/s)h_n(s)\ds -\frac{1}{n}\frac{\sn(r/r_1)}{\sn(r_2/r_1)} \int_{r_1}^{r_2}s\varpi'(s)\sn(r_2/s)h_n(s)\ds,   
\end{equation}
where we have introduced the auxiliary hyperbolic functions
\begin{equation}\label{def:auxhyper}
\sn(r):=\sinh(n\log r), \qquad \cn(r):=\cosh(n \log r). 
\end{equation}

Combining all, we immediately obtain that
\[
D_\ff F[\lambda,0]\hh (r,\theta)=(\phi'(r)+\lambda r) \hh(r,\theta) + \Phi(r,\theta), \qquad (r,\theta)\in D_\e,
\]
with $\Phi$ given by \eqref{def:sumPhi}-\eqref{def:Phi_n} and $\phi$ given by \eqref{def:Phi_0}-\eqref{def:ctePhi_0}. Then, defining $\cL[\l]\hh:=D_\ff F[\lambda,0]\hh$ give us that
\begin{equation}\label{def:LsumLn}
\cL[\l]\hh=\cL[\l]\left(\sum_{n\geq1}h_n(r)\cos(n\theta)\right)=\sum_{n\geq 1}\cos(n\theta)\cL_n[\l]\Pi_n \hh
\end{equation}
where $\Pi_n$ is just the projector onto $\cos(n\theta)$ and $\cL_n[\l]$ is the operator given by
\begin{multline}\label{def:Lnoperator}
\cL_n[\l]\gg(r):=\left(\phi'(r)+\l r\right)\gg(r)\\
+\frac{1}{n}\int_{r_1}^{r} s\varpi'(s)\sn(r/s)\gg(s)\ds-\frac{\sn(r/r_1)}{\sn(r_2/r_1)}\frac{1}{n}\int_{r_1}^{r_2} s\varpi'(s)\sn(r_2/s)\gg(s)\ds, \qquad r\in I_\e^{R_1,R_2}.
\end{multline}

\subsection{Introducing the Taylor-Couette profile}
Up to this moment we have been working during all this section for a general circular flow. However, we are just interested in the dynamics around the particular case of the Taylor-Couette flow. 

That is, we need to focus only on
\begin{equation}\label{vapi:background+perturbation}
\varpi(r) = 2A+\varpi_{\e,\k}(r),
\end{equation}
with $\varpi_{\e,\k}$ given by \eqref{varpiepkappa}. In fact,
recalling \eqref{def:Phi_0} and \eqref{def:ctePhi_0} and applying \eqref{vapi:background+perturbation} we get
\begin{equation}\label{def:Phi_0new}
\phi(r)= C_{\varpi_{\e,\k}}\log\left(\frac{r}{r_1}\right) -A\left(\frac{r^2-r_1^2}{2}-r_1^2\log\left(\frac{r}{r_1}\right)\right)- \int_{r_1}^r \frac{1}{s}\left(\int_{r_1}^{s} t \varpi_{\e,\k}(t)\mbox{d}t\right)\mbox{d}s,
\end{equation}
with
\begin{equation}\label{def:Cvapinew}
C_{\varpi_{\e,\k}}=\log^{-1}\left(\frac{r_2}{r_1}\right)\left[\gamma + A\left(\frac{r_2^2-r_1^2}{2}-r_1^2\log\left(\frac{r_2}{r_1}\right)\right)+\int_{r_1}^{r_2}\frac{1}{s}\left(\int_{r_1}^{s} t \varpi_{\e,\k}(t)\mbox{d}t\right)\mbox{d}s\right].
\end{equation}
In addition, the operator \eqref{def:Lnoperator} is reduced to the following expression
\begin{multline}\label{def:Lnggpre}
\cL_n[\l]\gg(r)=\left(\phi'(r)+\l r\right)\gg(r)\\
-\frac{\sn(r/r_1)}{\sn(r_2/r_1)}\frac{1}{n}\int_{r_1}^{r_2} s\varpi_{\e,\k}'(s)\sn(r_2/s)\gg(s)\ds +\frac{1}{n}\int_{r_1}^{r} s\varpi_{\e,\k}'(s)\sn(r/s)\gg(s)\ds.
\end{multline}

\section{Rescaling of the decomposition of the linear operator}\label{s:rescaling}
Recall that one of the main task of the paper will be to study the kernel and co-kernel of the linear operator $\cL[\l]\equiv \sum_{n\geq 1} \cos( n\theta)\cL_n[\l]\Pi_n$ over $D_\e$ where $\Pi_n$ is just the projector onto $\cos(n\theta)$. 

Here, it is important to emphasize that by definition of $D_\e$ and $\varpi'_{\e,\k}$ (see \eqref{dom:dynamics} and \eqref{varpiepkappa}) we have the following relation
\[
D_\e=\supp(\varpi')\times\T=\supp(\varpi'_{\e,\k})\times\T,
\]
where
\begin{equation}\label{supportvapi'}
\supp(\varpi'_{\e,\k})=I_\e^{R_1}\cup I_\e^{R_2},
\end{equation}
with
\[
I_\e^{R_1}=(R_1-\e,R_1+\e), \qquad  I_\e^{R_2}=(R_2-\e,R_2+\e).
\]

Our procedure will be to do an asymptotic analysis of the linear operator in terms of the smallness parameter $\e$. In order to do that, we will start writing the operator $\cL_n[\l]$ in a more tractable manner. 
Firstly, we note that we are considering a domain divided into two disjoint pieces. Then, taking \eqref{supportvapi'} into account we can just write
\[
\gg(r)=
\begin{cases}
\gg_1(r), \qquad r\in I_\e^{R_1},\\
\gg_2(r), \qquad r\in I_\e^{R_2}.
\end{cases}
\]
Using that information via \eqref{def:Lnggpre}, we obtain that linear operator $\cL_n[\l]$ admits an expression such as
\[
\cL_n[\l]
\begin{pmatrix}
\gg_1\\
\gg_2
\end{pmatrix}
(r)=
\begin{cases}
\cL_n^1[\l]
\begin{pmatrix}
\gg_1\\
\gg_2
\end{pmatrix}
(r), \qquad r\in I_\e^{R_1},\vspace{0.25 cm}\\
\cL_n^2[\l]
\begin{pmatrix}
\gg_1\\
\gg_2
\end{pmatrix}
(r), \qquad r\in I_\e^{R_2},
\end{cases}
\]
with
\begin{multline}\label{def:LnggR1}
\cL_n^1[\l]
\begin{pmatrix}
\gg_1\\
\gg_2
\end{pmatrix}
(r):=\left(\phi'(r)+\l r\right)\gg_1(r)\\
-\frac{\sn(r/r_1)}{\sn(r_2/r_1)}\frac{1}{n} \left(\int_{I_\e^{R_1}} s\varpi_{\e,\k}'(s)\sn(r_2/s)\gg_1(s)\ds +\int_{I_\e^{R_2}} s\varpi_{\e,\k}'(s)\sn(r_2/s)\gg_2(s)\ds\right)\\
 +\frac{1}{n}\int_{R_1-\e}^{r} s\varpi_{\e,\k}'(s)\sn(r/s)\gg_1(s)\ds, \qquad \text{for } r\in I_\e^{R_1},
\end{multline}
and
\begin{multline}\label{def:LnggR2}
\cL_n^2[\l]
\begin{pmatrix}
\gg_1\\
\gg_2
\end{pmatrix}
(r):=\left(\phi'(r)+\l r\right)\gg_2(r)\\
-\frac{\sn(r/r_1)}{\sn(r_2/r_1)}\frac{1}{n}\left(\int_{I_\e^{R_1}} s\varpi_{\e,\k}'(s)\sn(r_2/s)\gg_1(s)\ds + \int_{I_\e^{R_2}} s\varpi_{\e,\k}'(s)\sn(r_2/s)\gg_2(s)\ds\right)\\
 +\frac{1}{n}\left(\int_{I_\e^{R_1}} s\varpi_{\e,\k}'(s)\sn(r/s)\gg_1(s)\ds +  \int_{R_2-\e}^{r} s\varpi_{\e,\k}'(s)\sn(r/s)\gg_2(s)\ds\right), \qquad\text{for } r\in I_\e^{R_2}.
\end{multline}

\subsection{Re-scaling over $I_{\e}^{R_1}$ and $I_{\e}^{R_2}$}
Here, we are going to focus on write the previous operators $\cL_n^{1}[\l]$ and $\cL_n^{2}[\l]$ given respectively by \eqref{def:LnggR1} and \eqref{def:LnggR2} in a more convenient way that will allow us to handle them later. 

In first place, we take into account the change of variables
\[
r\in I_\e^{R_i}\mapsto z=(r-R_i)\e^{-1}\in (-1,1), \qquad \text{for } i=1,2.
\]
Secondly, recalling the definition of the profile function $\varpi_{\e,\k}$ given by \eqref{varpiepkappa} we obtain
\begin{equation*}
\varpi_{\e,\k}'(R_1+\e z)=-\varphi_\k' (-z), \qquad \varpi_{\e,\k}'(R_2+\e z)=+\varphi_\k' (+z).
\end{equation*}
Thirdly, we simplify the notation using  the auxiliary functions 
\begin{equation}\label{def:ggRi}
\bar{\gg}_i(z):=\gg(R_i+\e z),
\end{equation}
and
\begin{equation}\label{def:UpsRi}
\Upsilon_{\e,\k}^{R_i}(z):=(R_i+\e z)\phi'(R_i+\e z).
\end{equation}
Finally, combining all of the above, multiplying the resulting equation by  $(R_i+\e z)$ and using notation  \eqref{def:ggRi} and \eqref{def:UpsRi} we obtain
\begin{multline}\label{finalLnR1}
(R_1+\e z)\cL_n^1[\l]
\begin{pmatrix}
\gg_1\\
\gg_2
\end{pmatrix}
(R_1+\e z)=\left(\Upsilon_{\e,\k}^{R_1}(z)+\l (R_1+\e z)^2\right)\bar{\gg}_1(z)\\
+(R_1+\e z)\frac{\sn((R_1+\e z)/r_1)}{\sn(r_2/r_1)}\frac{\e}{n}\int_{-1}^{+1} (R_1+\e s)\varphi_\k' (-s)\sn(r_2/(R_1+\e s))\bar{\gg}_1(s)\ds\\
-(R_1+\e z)\frac{\sn((R_1+\e z)/r_1)}{\sn(r_2/r_1)}\frac{\e}{n}\int_{-1}^{+1} (R_2+\e s)\varphi_\k' (s)\sn(r_2/(R_2+\e s))\bar{\gg}_2( s)\ds\\
-(R_1+\e z)\frac{\e}{n}\int_{-1}^{z} (R_1+\e s)\varphi_\k' (-s)\sn((R_1+\e z)/(R_1+\e s))\bar{\gg}_1(s)\ds,
\end{multline}
and
\begin{multline}\label{finalLnR2}
(R_2+\e z)\cL_n^2[\l]\begin{pmatrix}
\gg_1\\
\gg_2
\end{pmatrix}
(R_2+\e z)=\left(\Upsilon_{\e,\k}^{R_2}(z)+\l (R_2+\e z)^2\right)\bar{\gg}_2(z)\\
+(R_2+\e z)\frac{\sn((R_2+\e z)/r_1)}{\sn(r_2/r_1)}\frac{\e}{n}\int_{-1}^{+1} (R_1+\e s)\varphi_\k' (-s)\sn(r_2/(R_1+\e s))\bar{\gg}_1(s)\ds \\
-(R_2+\e z)\frac{\sn((R_2+\e z)/r_1)}{\sn(r_2/r_1)}\frac{\e}{n}\int_{-1}^{+1} (R_2+\e s)\varphi_\k' (s)\sn(r_2/(R_2+\e s))\bar{\gg}_2(s)\ds\\
-(R_2+\e z)\frac{\e}{n}\int_{-1}^{+1} (R_1+\e s)\varphi_\k' (-s)\sn((R_2+\e z)/(R_1+\e s))\bar{\gg}_1(s)\ds\\
+(R_2+\e z)\frac{\e}{n}\int_{-1}^{ z} (R_2+\e s)\varphi_\k' (s)\sn((R_2+\e z)/(R_2+\e s))\bar{\gg}_2(s)\ds,
\end{multline}
with $z\in(-1,1).$

\begin{remark}
It is important to emphasize that after make these change of variables we get a factor $\varepsilon$ in front of each integral term. This is an immediate consequence of the fact that $|I_\e^{R_i}|=O(\e).$
\end{remark}

\begin{remark}
Note that each of the above expressions \eqref{finalLnR1} and \eqref{finalLnR2} are defined on the unit interval.
\end{remark}

\begin{remark}
We have multiplied \eqref{def:LnggR1} and \eqref{def:LnggR2} by $(R_i+\e z)$ respectively only for technical reasons. The main motive is that it will be easier to do the asymptotic analysis of the auxiliary function introduced in \eqref{def:UpsRi} instead of $\phi'(R_i+\e z)$.
\end{remark}

\subsection{The re-scaled linear operator }

Finally, for $z\in(-1,1)$ we define the auxiliary functional
\begin{align}
\cL_{n,\e}^{R_1}[\l]\begin{pmatrix}
\bar{\gg}_1\\
\bar{\gg}_2
\end{pmatrix}(z):&=(R_1+\e z)\cL_n^1[\l] \begin{pmatrix}
\gg_1\\
\gg_2
\end{pmatrix}(R_1+\e z), \label{LR1final}\\
\cL_{n,\e}^{R_2}[\l]\begin{pmatrix}
\bar{\gg}_1\\
\bar{\gg}_2
\end{pmatrix}(z):&=(R_2+\e z)\cL_n^2[\l] \begin{pmatrix}
\gg_1\\
\gg_2
\end{pmatrix}(R_2+\e z), \label{LR2final}
\end{align}
which are given by the right hand side of \eqref{finalLnR1} and \eqref{finalLnR2} respectively.\\

Here, it is important to write explicitly how we can recover the full expression of the linear operator by going backwards. Thus, let $\hh\in X(D_\e)$ given by the expansion
\[
\hh(r,\theta)=\sum_{n\geq 1}h_n(r) \cos( n\theta), \qquad (r,\theta)\in D_\e,
\]
with
\[
h_n(r)=
\begin{cases}
h_n^{R_1}(r), \qquad r\in I_\e^{R_1},\\
h_n^{R_2}(r), \qquad r\in I_\e^{R_2}.
\end{cases}
\]
Hence, for $(r,\theta)\in D_\e=\left(I_\e^{R_1}\cup I_\e^{R_2}\right)\times\T$, we finally get
\begin{align*}
\cL[\l]\hh(r,\theta)=\sum_{n\geq 1} \cos(n\theta)    \cL_n[\l] h_n(r)&=
\begin{cases}
\sum_{n\geq 1} \cos(n\theta)  \cL_n^1[\l]
\begin{pmatrix}
h_n^{R_1}\\
h_n^{R_2}
\end{pmatrix}
(r), \qquad r\in I_\e^{R_1},\vspace{0.25 cm}\\
\sum_{n\geq 1} \cos(n\theta)  \cL_n^2[\l]
\begin{pmatrix}
h_n^{R_1}\\
h_n^{R_2}
\end{pmatrix}
(r), \qquad r\in I_\e^{R_2},
\end{cases}   
\end{align*}
or equivalently
\begin{align*}
\cL[\l]\hh(r,\theta)=
\begin{cases}
\sum_{n\geq 1} \frac{\cos(n\theta)}{r}  \cL_{n,\e}^{R_1}[\l]
\begin{pmatrix}
\bar{h}_n^{R_1}\\
\bar{h}_n^{R_2}
\end{pmatrix}
((r-R_1)\e^{-1}), \qquad r\in I_\e^{R_1},\vspace{0.25 cm}\\
\sum_{n\geq 1} \frac{\cos(n\theta)}{r}  \cL_{n,\e}^{R_2}[\l]
\begin{pmatrix}
\bar{h}_n^{R_1}\\
\bar{h}_n^{R_2}
\end{pmatrix}
((r-R_2)\e^{-1}), \qquad r\in I_\e^{R_2},
\end{cases}   
\end{align*}
with $\cL_{n,\e}^{R_1}[\l], \cL_{n,\e}^{R_2}[\l]$ given  respectively by \eqref{LR1final}-\eqref{LR2final} and $\bar{h}_n^{R_i}(z):=h_n^{R_i}(R_i+\e z)$ for $z\in(-1,1).$

\subsection{Asymptotic analysis of $\Upsilon_{\e,\k}^{R_i}(z)$}
The main task of this section will be to get the asymptotic analysis of $\Upsilon_{\e,\k}^{R_i}(z)$ in terms of the parameter $\e$. In order to do that, we start recalling expression \eqref{def:Phi_0new}, which give us that
\[
\phi'(r)=\frac{C_{\varpi_{\e,\k}}}{r}-A\left(r-\frac{r_1^2}{r} \right)-\frac{1}{r}\int_{r_1}^{r}t\varpi_{\e,\k}(t)\dt.
\]
Then, as an inmediate consequence we get
\begin{align*}
\Upsilon_{\e,\k}^{R_i}(z)&=(R_i+\e z)\phi'(R_i+\e z)\\
&=C_{\varpi_{\e,\k}} -A((R_i+\e z)^2-r_1^2) -\int_{r_1}^{R_i+\e z} t \varpi_{\e,\k}(t)\mbox{d}t,
\end{align*}
with (see \eqref{def:Cvapinew})
\begin{align*}
C_{\varpi_{\e,\k}}&=\log^{-1}\left(\frac{r_2}{r_1}\right)\left[\gamma + A\left(\frac{r_2^2-r_1^2}{2}-r_1^2\log\left(\frac{r_2}{r_1}\right)\right)+\int_{r_1}^{r_2}\frac{1}{s}\left(\int_{r_1}^{s} t \varpi_{\e,\k}(t)\mbox{d}t\right)\mbox{d}s\right].
\end{align*}
Here, for the interested reader, we refer to the Appendix \ref{app:upsilon} for the details of the computations involved to prove that
\begin{equation}\label{decompoupsilon}
\Upsilon_{\e,\k}^{R_i}(z)=\tilde{\Upsilon}_0^{R_i}+\e \tilde{\Upsilon}_1^{R_i}(z)+\e^2 \tilde{\Upsilon}_{\e,\k}^{R_i}(z),
\end{equation}
with (see \eqref{auxUpsilon0}, \eqref{auxUpsilon11}, \eqref{auxUpsilon12}, \eqref{auxUpsilon21} and \eqref{auxUpsilon22}) 
\begin{equation}\label{Upsilon0}
\tilde{\Upsilon}_0^{R_i}=\log^{-1}\left(\frac{r_2}{r_1}\right)\left[\gamma + A\left(\frac{r_2^2-r_1^2}{2}-r_1^2\log\left(\frac{r_2}{r_1}\right)\right)\right]-A(R_i^2 - r_1^2), 
\end{equation}
\begin{align}
\tilde{\Upsilon}_1^{R_1}(z)&=\log^{-1}\left(\frac{r_2}{r_1}\right)\left[\frac{R_2^2-R_1^2}{2}\left(\log(r_2) + \frac{1}{2}\right)+ \frac{R_1^2}{2}\log(R_1)-\frac{R_2^2}{2}\log(R_2) \right] - 2AR_1 z, \label{Upsilon11}\\
\tilde{\Upsilon}_1^{R_2}(z)&=\tilde{\Upsilon}_1^{R_1}(0)- 2AR_2 z -\frac{R_2^2-R_1^2}{2}, \label{Upsilon12}
\end{align}
and
\begin{align}
\tilde{\Upsilon}_{\e,\k}^{R_1}(z)&=\frac{\mathsf{Remainder}_{\e,\k}}{\log(r_2/r_1)}-A z^2   -\int_{-1}^{z}(R_1+\e t) \varphi_{\k}(-t)\dt, \label{Upsilon21}\\
\tilde{\Upsilon}_{\e,\k}^{R_2}(z)&=\tilde{\Upsilon}_{\e,\k}^{R_1}(1)+A(1-z^2)- \left[\int_{-1}^{z}(R_2+\e t)\varphi_{\k}(+t)\dt- (R_1+R_2)\right], \label{Upsilon22}
\end{align}
with $\mathsf{Remainder}_{\e,\k}$ given explicitly by \eqref{auxremainder}.

\begin{remark}
Here, we emphasize that $\mathsf{Remainder}_{\e,\k}$ is a constant that depends only on $\e$ and $\k$.
\end{remark}

\begin{remark}
    Note that $\tilde{\Upsilon}_{\e,\k}^{R_i}(z)=O(1)$ for $i=1,2$ in terms of the smallness parameters $\e$ and $\k$.
\end{remark}

\section{One dimensionality of the kernel of the linear operator}\label{s:dim1kernel}
The following section consists on two well-differentiated parts. On the one hand, we will prove that there exists an element in the kernel of the linear operator. On the other hand, we will prove that the kernel is exactly equal to the span of this
element. In summary, the main result of this section is to prove the following result:

\begin{theorem}\label{mainthmkernel}
For any $M>1,$ there exist positive $\kappa_0=\kappa_0(M)$ and positive $\ep_0=\ep_0(M)$ such that for all $0<\ep<\ep_0$, $0<\kappa<\kappa_0$ and  $m\in \N$ such that $ m<M$, we can find $\lambda_{\ep,\kappa,m}\in \R$ and a $\frac{2\pi}{m}$-periodic and non-identically zero function $h_{\ep,\kappa,m}\in X(D_\ep)$  solving 
\begin{align*}
\cL_{\ep,\kappa}[\lambda_{\ep,\kappa,m}] h_{\ep,\kappa,m}=0,
\end{align*}
where the functional $\cL_{\ep,\kappa}[\lambda]$ is given in \eqref{def:LsumLn}.
In addition, the kernel $\cL_{\ep,\kappa}[\lambda_{\ep,\kappa,m}]$ on $X(D_\ep)$ is the span of $h_{\ep,\kappa,m}$ and the regularity of the solution is in fact $h_{\ep,\kappa,m}\in C^\infty(D_{\ep}$).
\end{theorem}

\begin{remark}
Importantly, $h_{\ep,\kappa,m}(r,\theta)$, with $\theta\in\T$ and $r\in I_{\ep}^{R_1,R_2}$, depends non trivially on $\theta$.
\end{remark}

\begin{remark} It is important to emphasize that the smallness of $\ep$ does not depend on $\kappa$. \end{remark}

\begin{remark}\label{remark}The speed of the rotation $\lambda_{\ep,\kappa,m}$ satisfies the expansion in terms of $\ep$-parameter:
\begin{align}\label{speedsuperindex}
\lambda_{\ep,\kappa,m}=\l_0+\lambda^1_{\kappa,m}\ep+\lambda^2_{\ep,\kappa,m}\ep^2,
\end{align}
where 
\[
\l_0 = \frac{u^\theta_{\textsf{TC}}(R_2)}{R_2}=\frac{1}{R_2}\left(A R_2+ \frac{B}{R_2}\right),
\]
\begin{align*}
\lambda^1_{\k,m}=O(1),
\end{align*}
and
\begin{align*}
|\lambda^2_{\ep,\kappa,m}|\leq C(M).
\end{align*}

\end{remark}

\begin{remark}
In the rest of the section we avoid to use subscripts $\ep,\k,m$ in expression \eqref{speedsuperindex}, which we will ignore in favor of readability. We will write with abuse of notation $\lambda=\l_0+\lambda_1\ep+\lambda_2^\ep \ep^2$. Similarly, in some cases where it is clear from the context, we can simply write $h=h_{\e,\k,m}\in X(D_\e)$ with expansion
\[
h(r,\theta)=\sum_{n\geq 1} h_n(r)\cos(n\theta), \qquad (r,\theta)\in D_\e.
\] 
\end{remark}

\subsection{Proof of existence.} \label{s:existencedim1}
In order to show the existence of $(\lambda_{\ep,\kappa,m},h_{\ep,\kappa,m})$ solving
$$\mathcal{L}[\lambda_{\ep,\kappa,m}]h_{\ep,\kappa,m}=0,$$
we fix  $m\in \N$ with $1\leq m<M$, and suppose that $h_{\e,\k,m}(r,\theta)=h_{\e,\k,m}(r)\cos(m\theta)$ consists only on the $m$-th mode of the expansion. Thus, we have reduced our infinite-dimensional problem for every $n\geq 1$ to a single problem for $n=m$.
That is, using the notation introduced previously, we take
\begin{align}
\bar{h}_n^{R_1}(z)=\left\{\begin{array}{cc} 0 & n\neq m,\\ \aa(z) & n=m, \end{array}\right.\label{hR1sol}\\
\bar{h}_n^{R_2}(z)=\left\{\begin{array}{cc} 0 & n\neq m,\\ \bb(z) & n=m, \end{array}\right.\label{hR2sol}
\end{align}
where functions $\aa$ and $\bb$ depend on $\ep,\kappa$ and $m$ but we do not make this dependence explicit for sake of simplicity. Thus, our  problem translates into finding $(\lambda,\aa,\bb)\in \R\times C^{2+\alpha}(-1,1)\times C^{2+\alpha}(-1,1)$ solving
\begin{align}
&\cL_{m,\e}^{R_1}[\lambda]\left(\begin{array}{cc}\aa \\ \bb\end{array}\right)=0, \label{LR1cero} \\
&\cL_{m,\e}^{R_2}[\lambda]\left(\begin{array}{cc}\aa \\ \bb\end{array}\right)=0,  \label{LR2cero}
\end{align}
where $\cL^{R_i}_{m,\e}$ are given in \eqref{LR1final} and \eqref{LR2final} or equivalently by the right-hand side of \eqref{finalLnR1} and \eqref{finalLnR2}. After that, the solution $(\lambda_{\ep,\kappa,m},\,h_{\ep,\kappa,m})$ will be given by
\begin{align*}
&\lambda_{\ep,\kappa,m}=\lambda,\\
&h_{\ep,\kappa,m}(r,\theta)=\aa\left(\frac{r-R_1}{\ep}\right)\cos(m\theta)\quad \text{for $r\in I_\e^{R_1}$},\\
&h_{\ep,\kappa,m}(r,\theta)=\bb\left(\frac{r-R_2}{\ep}\right)\cos(m\theta)\quad \text{for $r\in I_\e^{R_2}$}.
\end{align*}

In order to solve \eqref{LR1cero} and \eqref{LR2cero} we introduce the ansatzs
\begin{equation}\label{ansatzab}
\left\{
\begin{aligned}
\aa(z)&= \hspace{1.26cm} \aa_1(z)\e + \aa_2^\ep(z) \ep^2,\\
\bb(z)&=\bb_0(z)+\bb^\ep_1(z)\ep,
\end{aligned}
\right.
\end{equation}
and
\begin{equation}\label{ansatzlambda}
\lambda=\l_0+\lambda_1\ep+\lambda_2^\ep \ep^2,
\end{equation}
with 
\begin{equation}\label{def:l0}
\l_0=-\frac{\tilde{\Upsilon}_0^{R_2}}{R_2^2}=A\left(1-\frac{r_2^2-r_1^2}{2 R_2^2 \log\left(\frac{r_2}{r_1}\right)}\right)-\frac{\gamma}{R_2^2\log(r_2/r_1)},
\end{equation}
where $\aa_1$, $\bb_0$ and $\lambda_1$ will depend on $\kappa$ and $m$ but they will not depend on $\ep$. In addition, the remaining terms $\aa^\ep_2 $, $\bb^\ep_1$ and $\l^\ep_2$ depend on $\ep$, $\kappa$ and $m$.

\begin{remark}
Here, it is important to emphasize that the speed $\l$ of the rotation that we construct is related with the localization of the ``legs'' of the trapezoid given by the profile $\varpi_{\e,\k}$ and the circulation $\gamma$ of the Taylor-Couette flow. More specifically, taking into account (see \eqref{def:gamma}) that
\[
\gamma=(-1)\left(\frac{A}{2}(r_2^2-r_1^2) + B \log (r_2/r_1) \right),
\]
and applying that into \eqref{def:l0}, we finally arrive to the desired expression
\begin{equation}\label{eq:l0TC}
\l_0 = \frac{u^\theta_{\textsf{TC}}(R_2)}{R_2}.
\end{equation}
\end{remark}

\begin{remark}
As a consequence of the above expression we take $R_2\in(r_1,r_2)$ such that $u^\theta_{\textsf{TC}}(R_2)\neq 0$ to construct time-periodic solutions with angular velocity of order $O(1)$ in terms of $\e$.
\end{remark}

\begin{remark}
It is important to note that in the particular case $u^\theta_{\textsf{TC}}(R_1)=u^\theta_{\textsf{TC}}(R_2)$ our procedure give us that at the first order in $\e$, if $(\aa,\bb)$ is in the kernel of $\cL[\l_{\e,\k,m}]$ also $(\bb,\aa)$ is in the kernel of $\cL[\l_{\e,\k,m}]$ for the same value of $\l$.
For that reason, to avoid  technicalities, we will assume that $R_1,R_2\in (r_1,r_2)$ such that $u^\theta_{\textsf{TC}}(R_1)\neq u^\theta_{\textsf{TC}}(R_2)$.
\end{remark}

Putting our ansatz \eqref{ansatzab}, \eqref{ansatzlambda} into \eqref{LR1cero}, \eqref{LR2cero} reduces the problem to solving the following system of equations
\begin{equation*}
\left\{
\begin{aligned}
\cL_{m,\e}^{R_1}[\l_0+\l_1 \e + \l_2^\e \e^2]\left(\begin{array}{cc}\hspace*{1.2 cm} \aa_1(z)\ep+\aa_2^\ep(z) \ep^2 \\ \hspace*{-1.8 cm} \bb_0(z)+ \bb^\ep_1(z)\ep\end{array}\right)(z)&=0,\\
\cL_{m,\e}^{R_2}[\l_0+\l_1 \e + \l_2^\e \e^2]\left(\begin{array}{cc}\hspace*{1.2 cm} \aa_1(z)\ep+\aa_2^\ep(z) \ep^2 \\ \hspace*{-1.8 cm} \bb_0(z)+ \bb^\ep_1(z)\ep\end{array}\right)(z)&=0.
\end{aligned}
\right.
\end{equation*}

\subsubsection{Asymptotic analysis of the linear operator}
In order to get a tractable manner to solve the above system, we start introducing the ansatz for the eigenfunction given by \eqref{ansatzab} into \eqref{finalLnR1}-\eqref{finalLnR2}. This implies that
\begin{multline}\label{LR1ansatz}
\cL_{m,\e}^{R_1}[\l]\begin{pmatrix}
\mathsf{a} \\
\mathsf{b}
\end{pmatrix}(z)=\left[\Upsilon_{\e,\k}^{R_1}(z)+ \l (R_1+\e z)^2\right]\left(\aa_1(z)\ep+\aa_2^\ep(z) \ep^2\right)\\
+(R_1+\e z)\frac{\sm((R_1+\e z)/r_1)}{\sm(r_2/r_1)}\frac{\e}{m}\int_{-1}^{+1} (R_1+\e s)\varphi_\k' (-s)\sm(r_2/(R_1+\e s))\left(\aa_1(s)\ep+\aa_2^\ep(s) \ep^2\right)\ds\\
-(R_1+\e z)\frac{\sm((R_1+\e z)/r_1)}{\sm(r_2/r_1)}\frac{\e}{m}\int_{-1}^{+1} (R_2+\e s)\varphi_\k' (s)\sm(r_2/(R_2+\e s))\left(\bb_0(s)+ \bb^\ep_1(s)\ep\right) \ds\\
-(R_1+\e z)\frac{\e}{m}\int_{-1}^{z} (R_1+\e s)\varphi_\k' (-s)\sm((R_1+\e z)/(R_1+\e s))\left(\aa_1(s)\ep+\aa_2^\ep(s) \ep^2\right)\ds=0,
\end{multline}
and
\begin{multline}\label{LR2ansatz}
\cL_{m,\e}^{R_2}[\l]\begin{pmatrix}
\mathsf{a} \\
\mathsf{b}
\end{pmatrix}(z)=\left[\Upsilon_{\e,\k}^{R_2}(z)+\l (R_2+\e z)^2\right]\left(\bb_0(z)+ \bb^\ep_1(z)\ep\right)\\
+(R_2+\e z)\frac{\sm((R_2+\e z)/r_1)}{\sm(r_2/r_1)}\frac{\e}{m}\int_{-1}^{+1} (R_1+\e s)\varphi_\k' (-s)\sm(r_2/(R_1+\e s))\left(\aa_1(s)\ep+\aa_2^\ep(s) \ep^2\right)\ds \\
-(R_2+\e z)\frac{\sm((R_2+\e z)/r_1)}{\sm(r_2/r_1)}\frac{\e}{m}\int_{-1}^{+1} (R_2+\e s)\varphi_\k' (s)\sm(r_2/(R_2+\e s))\left(\bb_0(s)+ \bb^\ep_1(s)\ep\right)\ds\\
-(R_2+\e z)\frac{\e}{m}\int_{-1}^{+1} (R_1+\e s)\varphi_\k' (-s)\sm((R_2+\e z)/(R_1+\e s))\left(\aa_1(s)\ep+\aa_2^\ep(s) \ep^2\right)\ds\\
+(R_2+\e z)\frac{\e}{m}\int_{-1}^{ z} (R_2+\e s)\varphi_\k' (s)\sm((R_2+\e z)/(R_2+\e s))\left(\bb_0(s)+ \bb^\ep_1(s)\ep\right)\ds=0.
\end{multline}

The next step will be to introduce also the ansantz for the eigenvalue given by \eqref{ansatzlambda} into the previous system of equations. But before that, it will be convenient to use \eqref{decompoupsilon} to write
\begin{align*}
\Upsilon_{\e,\k}^{R_i}(z)+\l (R_i+\e z)^2&=(\tilde{\Upsilon}_0^{R_i}+\e \tilde{\Upsilon}_1^{R_i}(z)+\e^2 \tilde{\Upsilon}_{\e,\k}^{R_i}(z))+(\l_0+\l_1\e +\l_2^\e \e^2)(R_i+\e z)^2\\
&=\left(\l_0 R_i^2+\tilde{\Upsilon}_0^{R_i}\right)\\
&\quad +\left(\l_0 2 R_i z + \l_1 R_i^2+\tilde{\Upsilon}_1^{R_i}(z)\right)\e\\
&\quad +\left(\l_0 z^2 +\l_1 2 z R_i + \l_2^\e R_i^2+\tilde{\Upsilon}_{\e,\k}^{R_i}(z)+\l_2^\e 2z R_i\e+\l_2^\e z^2\e^2\right)\e^2.
\end{align*}
For simplicity, in the rest of this section we will use the notation
\begin{equation}\label{decominalphas}
\Upsilon_{\e,\k}^{R_i}(z)+\l (R_i+\e z)^2=\a_0^{R_i}[\l_0]+ \a_1^{R_i}[\l_0,\l_1](z)\e+\a_{2,\e}^{R_i}[\l_0,\l_1,\l_2^\e](z)\e^2,
\end{equation}
with
\begin{align}
\a_0^{R_i}[\l_0]:&=\l_0 R_i^2+\tilde{\Upsilon}_0^{R_i}, \label{def:alpha0}\\
\a_1^{R_i}[\l_0,\l_1](z):&=\l_0 2 R_i z + \l_1 R_i^2+\tilde{\Upsilon}_1^{R_i}(z), \label{def:alpha1}\\
\a_{2,\e}^{R_i}[\l_0,\l_1,\l_2^\e](z):&=\l_0 z^2 +\l_1 2 z R_i + \l_2^\e R_i^2+\tilde{\Upsilon}_{\e,\k}^{R_i}(z)+\l_2^\e 2z R_i\e+\l_2^\e z^2\e^2. \label{def:alpha2}
\end{align}

At this point, we have all the ingredients to write \eqref{LR1ansatz} and \eqref{LR2ansatz} in a much more manageable way. In first place, we note that we can decompose $\cL_{m,\e}^{R_1}[\l]$ as follows
\begin{equation}\label{LR1T1T2T3}
\cL_{m,\e}^{R_1}[\l]\begin{pmatrix}
\aa \\
\bb
\end{pmatrix}=\cT_{m,\e}^1[\aa_1,\bb_0,\l_0] \e+ \cT_{m,\e}^2[\aa_1,\aa_2^\e,\bb_1^\e,\l_0,\l_1] \e^2 +\cT_{m,\e}^3[\aa_1,\aa_2^\e,\l_0,\l_1,\l_2^\e]\e^3,
\end{equation}
with
\begin{multline}\label{T1e}
\cT_{m,\e}^1[\aa_1,\bb_0,\l_0](z)=\a_0^{R_1}[\l_0] \aa_1(z) \\ 
-(R_1+\e z)\frac{\sm((R_1+\e z)/r_1)}{\sm(r_2/r_1)}\frac{1}{m}\int_{-1}^{+1} (R_2+\e s)\varphi_\k' (s)\sm(r_2/(R_2+\e s))\bb_0(s) \ds,
\end{multline}
\begin{multline}\label{T2e}
\cT_{m,\e}^2[\aa_1,\aa_2^\e,\bb_1^\e,\l_0,\l_1](z)=\a_1^{R_1}[\l_0,\l_1] \aa_1(z)+\a_0^{R_1}[\l_0]\aa_2^{\e}(z)\\
+(R_1+\e z)\frac{\sm((R_1+\e z)/r_1)}{\sm(r_2/r_1)}\frac{1}{m}\int_{-1}^{+1} (R_1+\e s)\varphi_\k' (-s)\sm(r_2/(R_1+\e s))\aa_1(s)\ds\\
-(R_1+\e z)\frac{\sm((R_1+\e z)/r_1)}{\sm(r_2/r_1)}\frac{1}{m}\int_{-1}^{+1} (R_2+\e s)\varphi_\k' (s)\sm(r_2/(R_2+\e s))\bb^\ep_1(s) \ds\\
-(R_1+\e z)\frac{1}{m}\int_{-1}^{z} (R_1+\e s)\varphi_\k' (-s)\sm((R_1+\e z)/(R_1+\e s))\aa_1(s)\ds,
\end{multline}
and
\begin{multline}\label{T3e}
\cT_{m,\e}^3[\aa_1,\aa_2^\e,\l_0,\l_1,\l_2^\e](z)=\a_1^{R_1}[\l_0,\l_1]\mathsf{a}_2^{\e}(z) +\a_{2,\e}^{R_1}[\l_0,\l_1,\l_2^\e] \left(\mathsf{a_1}(z)+\mathsf{a}_2^{\e}(z)\e\right)\\
+(R_1+\e z)\frac{\sm((R_1+\e z)/r_1)}{\sm(r_2/r_1)}\frac{1}{m}\int_{-1}^{+1} (R_1+\e s)\varphi_\k' (-s)\sm(r_2/(R_1+\e s))\aa_2^\ep(s) \ds\\
-(R_1+\e z)\frac{1}{m}\int_{-1}^{z} (R_1+\e s)\varphi_\k' (-s)\sm((R_1+\e z)/(R_1+\e s)) \aa_2^\ep(s) \ds.
\end{multline}

Proceeding in a similar manner, we can also decompose $\cL_{m,\e}^{R_2}[\l]$ as follows
\begin{multline}\label{LR2Q1Q2Q3}
\cL_{m,\e}^{R_2}[\l]\begin{pmatrix}
\aa \\
\bb
\end{pmatrix}=\a_0^{R_2}[\l_0] \mathsf{b}_0(z)\\
+\mathcal{Q}_{m,\e}^1[\bb_0,\bb_1^\e,\l_0,\l_1]\e+ \mathcal{Q}_{m,\e}^2[\aa_1,\aa_2^\e,\bb_0,\bb_1^\e,\l_0,\l_1,\l_2^\e] \e^2 +\mathcal{Q}_{m,\e}^3[\aa_2^\ep,\bb_1^\e,\l_0,\l_1,\l_2^\e]\e^3,
\end{multline}
with
\begin{multline}\label{Q1e}
\mathcal{Q}_{m,\e}^1[\bb_0,\bb_1^\e,\l_0,\l_1](z)=\a_0^{R_2}[\l_0] \bb_1^\e(z)+\a_1^{R_2}[\l_0,\l_1]\bb_0(z)\\
-(R_2+\e z)\frac{\sm((R_2+\e z)/r_1)}{\sm(r_2/r_1)}\frac{1}{m}\int_{-1}^{+1} (R_2+\e s)\varphi_\k' (s)\sm(r_2/(R_2+\e s))\bb_0(s)\ds\\
+(R_2+\e z)\frac{1}{m}\int_{-1}^{ z} (R_2+\e s)\varphi_\k' (s)\sm((R_2+\e z)/(R_2+\e s))\bb_0(s)\ds,
\end{multline}
\begin{multline}\label{Q2e}
\mathcal{Q}_{m,\e}^2[\aa_1,\bb_0,\bb_1^\e,\l_0,\l_1,\l_2^\e](z)=\a_1^{R_2}[\l_0,\l_1]\bb_1^\e(z)+ \a_{2,\e}^{R_2}[\l_0,\l_1,\l_2^\e]\mathsf{b}_0(z)\\
+(R_2+\e z)\frac{\sm((R_2+\e z)/r_1)}{\sm(r_2/r_1)}\frac{1}{m}\int_{-1}^{+1} (R_1+\e s)\varphi_\k' (-s)\sm(r_2/(R_1+\e s))\aa_1(s)\ds \\
-(R_2+\e z)\frac{\sm((R_2+\e z)/r_1)}{\sm(r_2/r_1)}\frac{1}{m}\int_{-1}^{+1} (R_2+\e s)\varphi_\k' (s)\sm(r_2/(R_2+\e s)) \bb^\ep_1(s)\ds\\
-(R_2+\e z)\frac{1}{m}\int_{-1}^{+1} (R_1+\e s)\varphi_\k' (-s)\sm((R_2+\e z)/(R_1+\e s)) \aa_1(s)\ds\\
+(R_2+\e z)\frac{1}{m}\int_{-1}^{ z} (R_2+\e s)\varphi_\k' (s)\sm((R_2+\e z)/(R_2+\e s)) \bb^\ep_1(s)\ds,
\end{multline}
and
\begin{multline}\label{Q3e}
\mathcal{Q}_{m,\e}^3[\aa_2^\ep,\bb_1^\e,\l_0,\l_1,\l_2^\e](z)=\a_{2,\e}^{R_2}[\l_0,\l_1,\l_2^\e] \mathsf{b}_1^\e(z)\\
+(R_2+\e z)\frac{\sm((R_2+\e z)/r_1)}{\sm(r_2/r_1)}\frac{1}{m}\int_{-1}^{+1} (R_1+\e s)\varphi_\k' (-s)\sm(r_2/(R_1+\e s))\aa_2^\ep(s) \ds\\
-(R_2+\e z)\frac{1}{m}\int_{-1}^{+1} (R_1+\e s)\varphi_\k' (-s)\sm((R_2+\e z)/(R_1+\e s))\aa_2^\ep(s) \ds.
\end{multline}

\subsubsection{Dividing the problem into main and lower order terms.}
Taking into account these two decompositions, our problem translates into solve the system
\begin{align*}
\cT_{m,\e}^1[\aa_1,\bb_0,\l_0] \e+ \cT_{m,\e}^2[\aa_1,\aa_2^\e,\bb_1^\e,\l_0,\l_1] \e^2 +\cT_{m,\e}^3[\aa_1,\aa_2^\e,\l_0,\l_1,\l_2^\e]\e^3&=0,\\
\a_0^{R_2}[\l_0] \mathsf{b}_0(z) +\mathcal{Q}_{m,\e}^1[\bb_0,\bb_1^\e,\l_0,\l_1]\e+ \mathcal{Q}_{m,\e}^2[\aa_1,\bb_0,\bb_1^\e,\l_0,\l_1,\l_2^\e] \e^2 +\mathcal{Q}_{m,\e}^3[\aa_2^\ep,\bb_1^\e,\l_0,\l_1,\l_2^\e]\e^3 &=0.
\end{align*} 
In fact, since $\l_0$ is defined by \eqref{def:l0} in such a way that $\a_0^{R_2}[\l_0]=0$ we just have to solve
\begin{equation}\label{systemTsQs}
\left\{
\begin{aligned}
\cT_{m,\e}^1[\aa_1,\bb_0,\l_0] + \cT_{m,\e}^2[\aa_1,\aa_2^\e,\bb_1^\e,\l_0,\l_1] \e +\cT_{m,\e}^3[\aa_1,\aa_2^\e,\l_0,\l_1,\l_2^\e]\e^2&=0,\\
\mathcal{Q}_{m,\e}^1[\bb_0,\l_0,\l_1]+ \mathcal{Q}_{m,\e}^2[\aa_1,\bb_0,\bb_1^\e,\l_0,\l_1,\l_2^\e] \e +\mathcal{Q}_{m,\e}^3[\aa_2^\ep,\bb_1^\e,\l_0,\l_1,\l_2^\e]\e^2 &=0,
\end{aligned}
\right.
\end{equation}
where we have use  that $\mathcal{Q}_{m,\e}^1[\bb_0,\bb_1^\e,\l_0,\l_1]\equiv \mathcal{Q}_{m,\e}^1[\bb_0,\l_0,\l_1]$ due to the fact that the term depending of $\bb_1^\e$ is multiplied by the factor $\a_0^{R_2}[\l_0]=0$.

In order to start solving \eqref{systemTsQs} we need to use the fact that
\begin{align*}
\mathcal{T}_{m,\e}^1[\aa_1,\bb_0,\l_0](z)&=\mathcal{T}_{m,0}^1[\aa_1,\bb_0,\l_0](z)+\int_0^\e \left(\p_\d\mathcal{T}_{m,\d}^1\right)[\bb_0](z)\mbox{d}\delta,\\
\mathcal{Q}_{m,\e}^1[\bb_0,\l_0,\l_1](z)&=\mathcal{Q}_{m,0}^1[\bb_0,\l_0,\l_1](z)+\int_0^\e (\p_\d\mathcal{Q}_{m,\d}^1)[\bb_0](z)\mbox{d}\delta,
\end{align*}
where (see \eqref{T1e}, \eqref{Q1e})
\begin{equation}\label{T10}
\cT_{m,0}^1[\aa_1,\bb_0,\l_0](z)=\a_0^{R_1}[\l_0] \aa_1(z)  
-R_1 R_2\frac{\sm(R_1/r_1)\sm(r_2/R_2)}{\sm(r_2/r_1)}\frac{1}{m}\int_{-1}^{+1} \varphi_\k' (s)\bb_0(s) \ds,
\end{equation}
and
\begin{equation}\label{Q10}
\mathcal{Q}_{m,0}^1[\bb_0,\l_0,\l_1](z)=\a_1^{R_2}[\l_0,\l_1]\bb_0(z)-R_2^2\frac{\sm(R_2/r_1)\sm(r_2/R_2)}{\sm(r_2/r_1)}\frac{1}{m}\int_{-1}^{+1} \varphi_\k' (s)\bb_0(s)\ds.
\end{equation}

Consequently, the use of the asymptotic expansion of the terms $\cT_{m,\e}^1, \mathcal{Q}_{m,\e}^1$ near $\e=0$ allow us to split the problem of solving the most involved system \eqref{systemTsQs} into the study of solving the following two system of equations given by \eqref{systemT0Q0} and \eqref{systemTeQe}.

Now, we know explicitly the value of $\l_0$. Thus, we start solving
\begin{equation}\label{systemT0Q0}
\left\{
\begin{aligned}
\cT_{m,0}^1[\aa_1,\bb_0;\l_0]&=0,\\
\mathcal{Q}_{m,0}^1[\bb_0,\l_1;\l_0] &=0.
\end{aligned}
\right.
\end{equation}
After solve the above system, we will have an explicity expression for $\aa_1, \bb_0, \l_1$ and $\l_0$. Hence, we have all the necessary ingredients  to solve via a fixed point argument the system
\begin{equation}\label{systemTeQe}
\left\{
\begin{aligned}
\left(\cT_{m,\e}^2[\aa_2^\e,\bb_1^\e;\aa_1,\l_0,\l_1]+\frac{1}{\e}\int_0^\e \left(\p_\d\mathcal{T}_{m,\d}^1\right)[\bb_0]\mbox{d}\delta\right)  +\cT_{m,\e}^3[\aa_2^\e,\l_2^\e;\aa_1,\l_0,\l_1]\e&=0,\\
\left(\mathcal{Q}_{m,\e}^2[\bb_1^\e,\l_2^\e;\aa_1,\bb_0,\l_0,\l_1]+\frac{1}{\e}\int_0^\e (\p_\d\mathcal{Q}_{m,\d}^1)[\bb_0]\mbox{d}\delta\right)  +\mathcal{Q}_{m,\e}^3[\aa_2^\ep,\bb_1^\e,\l_2^\e;\l_0,\l_1]\e &=0.
\end{aligned}
\right.
\end{equation}
Here, it is important emphasize that we just want to solve system \eqref{systemTeQe} for unknowns $\aa_2^\e,\bb_1^\e,\l_2^\e$.

\begin{remark}
Just before and from now on we will repeatedly use the notation $\mathcal{T}[\pp;\mathsf{q}]$ for an operator $\mathcal{T}$ dependent on variable functions $\pp,\mathsf{q}$ where $\pp$ is an unknown and $\mathsf{q}$ is previously known.
\end{remark}

\subsubsection{Solving the main order}
In this section, we are going to solve the system \eqref{systemT0Q0}.
Equivalently, we just have to solve for $\aa_1,\bb_0$ and $\l_1$ the following system of equations
\begin{align}
\a_0^{R_1}[\l_0] \aa_1(z)-\pp_1(m)\int_{-1}^{1} \varphi_\k'(s)\bb_0(s)\ds&=0,\label{a1equation}\\
\a_1^{R_2}[\l_0,\l_1](z)\bb_0(z)-\pp_2(m)\int_{-1}^{1} \varphi_\k'( s)\bb_0(s)\ds&=0,\label{b0equation}
\end{align}
with
\begin{equation}\label{def:ppi}
\pp_i(m):=R_i R_2\frac{\sm(R_i/r_1) \sm(r_2/R_2)}{\sm(r_2/r_1)}\frac{1}{m}.
\end{equation}

Note that \eqref{a1equation} implies that $\aa_1(z)\equiv \aa_1$ is just a constant value  given by
\begin{equation}\label{def:a1}
\aa_1:=\frac{1}{\a_0^{R_1}[\l_0]}\pp_1(m)\int_{-1}^{1} \varphi_\k'(s)\bb_0(s)\ds.
\end{equation}
Recalling the circulation of the Taylor-Couette flow (see \eqref{def:gamma}) and the definition of $\tilde{\Upsilon}_0^{R_i}$ given by \eqref{Upsilon0} we have that $\tilde{\Upsilon}_0^{R_i}=-(A R_i^2+B)$. Moreover, taking into account \eqref{def:alpha0} and the fact that $B\neq 0$, we get
\begin{equation}\label{aux:a1definedok}
\alpha_0^{R_1}[\l_0]=\l_0 R_1^2+\tilde{\Upsilon}_0^{R_1}=B\frac{R_1^2-R_2^2}{R_2^2}\neq 0.
\end{equation}

Before solve \eqref{b0equation} we have to pick $\l_1$ appropriately. 
Notice that from \eqref{b0equation} we see that $\bb_0$ must be of the form
\begin{equation}\label{def:b0}
\bb_0(z)=\frac{C}{\a_1^{R_2}[\l_0,\l_1](z)},
\end{equation}
for some free constant $C$. Plugging the above expression into \eqref{b0equation}  one finds that
\begin{equation}\label{lambda1eq}
1=\pp_2(m)\int_{-1}^{1} \frac{\varphi_\k'( s)}{\a_1^{R_2}[\l_0,\l_1](s)}\ds,
\end{equation}
which is an equation for $\l_1$.\\

Therefore, taking into account all the above, the procedure to solve \eqref{systemT0Q0} or equivalently \eqref{a1equation}, \eqref{b0equation} follows the next steps. Firstly, we determine the value of $\l_1$ solving \eqref{lambda1eq}. Secondly, we determine $\bb_0$ using expression \eqref{def:b0}. Lastly, we obtain the constant value of $\aa_1$ via \eqref{def:a1}. 

\subsubsection{Determining the value of $\l_1$}
Here, we are going to determine the value of $\l_1$ satisfying \eqref{lambda1eq}. In order to do that, we observe (see \eqref{def:alpha1} and \eqref{Upsilon12}) that $\a_1^{R_2}[\l_0,\l_1](z)$ is nothing more than a linear function that can be written as
\[
\a_1^{R_2}[\l_0,\l_1](z)=(\l_0-A)2 R_2 z+\l_1 R_2^2+\tilde{\Upsilon}_1^{R_2}(0).
\]
In addition, using \eqref{eq:l0TC} we note that 
\begin{equation}\label{auxIlamba1}
\a_1^{R_2}[\l_0,\l_1]'(z)=(\l_0-A)2 R_2=\frac{2B}{R_2}.
\end{equation} 
Therefore, we have that $\sign(\a_1^{R_2}[\l_0,\l_1]'(z))=\sign(B)$ for all $|z|\leq 1$. Consequently, the function $\a_1^{R_2}[\l_0,\l_1](z)$ is strictly increasing/decreasing depending only of the sign of parameter $B$.\medskip

The task of the next lemma will be to identify the value of $\lambda_1$ solving \eqref{lambda1eq}.
\begin{lemma}\label{lema:lambda1} For any $M>1,$ there exist $\kappa'_0=\kappa'_0(M)$ such that for all $0< \kappa \leq \kappa'_0$ and $1\leq m <M$, we can find $\l_1$  (depending on $\kappa$ and $m$) such that \eqref{lambda1eq} holds.
In addition, fixed $m$ and $\kappa$ this solution is unique and satisfies 
$$\lambda_1=\lambda^*-\beta(M), \qquad 0< \beta(M)\ll 1,$$
where $\l^*$ is given by \eqref{def:lambdastar} and $\beta(M)$ just depends on $M$ (it does not depend on either $\kappa$ or $m$).
\end{lemma}



We start by focusing our attention on the case $B>0$ since the details for the other case follow along the same lines. Therefore, we will leave to the interested reader the details for the case $B<0$.

First, it is important to recall that for the case $B>0$ we know that $\a_1^{R_2}[\l_0,\l_1](z)$ is a strictly increasing linear function. In particular, we have that
\[
\a_1^{R_2}[\l_0,\l_1](z)\leq \a_1^{R_2}[\l_0,\l_1](1)=\frac{2B}{R_2}+\l_1 R_2^2+\tilde{\Upsilon}_1^{R_2}(0),
\]
with
\[
\tilde{\Upsilon}_1^{R_2}(0)=\log^{-1}\left(\frac{r_2}{r_1}\right)\left[\frac{R_2^2-R_1^2}{2}\left(\log(r_2) + \frac{1}{2}\right)+ \frac{R_1^2}{2}\log(R_1)-\frac{R_2^2}{2}\log(R_2) \right]-\frac{R_2^2-R_1^2}{2}.
\]

A priori we do not know what is the sign of $\a_1^{R_2}[\l_0,\l_1](1)$, but we can assure that
\begin{equation}\label{alpha1R2negative}
\a_1^{R_2}[\l_0,\l_1](z)\leq 0, \qquad  \text{for all } |z|\leq 1,
\end{equation}
taking $\l_1\in(-\infty, \l^\ast)$, where $\l^\ast$ is defined via the following relation
\begin{equation}\label{def:lambdastar}
\l^{\ast} = -\frac{1}{R_2^2}\left(\frac{2B}{R_2}+\tilde{\Upsilon}_1^{R_2}(0)\right).
\end{equation}

Finally, we also need to obtain a lower bound of $\pp_2(m)$ for $1\leq m < M$. Recalling the definition of $\pp_2(m)$ thought \eqref{def:ppi} and the fact that $\sm(\cdot):=\sinh(m \log (\cdot))$ with $\sinh(\cdot)$  strictly increasing function, we obtain
\begin{equation}\label{LBpp2}
\pp_2(m)\geq  \frac{R_2^2}{M}\frac{\cS_{1}(R_2/r_1) \cS_{1}(r_2/R_2)}{\cS_M(r_2/r_1)}=:\mathfrak{m}(M) > 0.
\end{equation}

\begin{proof}[Proof of Lemma \ref{lema:lambda1}]
Let me define the auxiliary function
\[
I(\tilde{\l}):=\pp_2(m)\int_{-1}^{1} \frac{\varphi_\k'( s)}{\a_1^{R_2}[\l_0,\tilde{\l}](s)}\ds.
\]
Applying the fact that $\varphi_\k'(s)\leq 0$  (see Lemma \ref{propiedadesvarphi}) and \eqref{alpha1R2negative} we have that $I:(-\infty,\l^{\ast})\rightarrow [0,\infty)$ is a non-negative function. In addition, it is continuous, strictly increasing with $I'(\tilde{\l})> 0$ and satisfying
\begin{equation}\label{limitl1}
\lim_{\tilde{\l}\to -\infty} I(\tilde{\l})=0.
\end{equation}
In fact, we obtain the following lower bound
\begin{align*}
\frac{I(\tilde{\l})}{\mathfrak{m}(M)}&\geq \int_{-1}^{1} \frac{\varphi_\k'( s)}{\a_1^{R_2}[\l_0,\tilde{\l}](s)}\ds \geq \int_{-1}^{1-\k} \frac{\varphi_\k'( s)}{\a_1^{R_2}[\l_0,\tilde{\l}](s)}\ds\\
&=\int_{-1}^{1-\k} \frac{\varphi_\k'( s)+1/2}{\a_1^{R_2}[\l_0,\tilde{\l}](s)}\ds-\frac{1}{2} \int_{-1}^{1-\k} \frac{1}{\a_1^{R_2}[\l_0,\tilde{\l}](s)}\ds\\
&\geq -\frac{\|\varphi_\k'+\tfrac{1}{2}\|_{L^1([-1,1-\k])}}{|\a_1^{R_2}[\l_0,\tilde{\l}](1-\k)|}-\frac{1}{2}\frac{1}{2 R_2 (\l_0-A)}\log\left(\frac{|\a_1^{R_2}[\l_0,\tilde{\l}](1-\k)|}{|\a_1^{R_2}[\l_0,\tilde{\l}](-1)|}\right).
\end{align*}
Then, applying again Lemma \ref{propiedadesvarphi} and  using \eqref{auxIlamba1}  we have
\[
\frac{I(\tilde{\l})}{\mathfrak{m}(M)} \geq -\frac{C \k}{|\a_1^{R_2}[\l_0,\tilde{\l}](1-\k)|} - \frac{1}{2} \frac{R_2}{2B}\log\left(\frac{|\a_1^{R_2}[\l_0,\tilde{\l}](1-\k)|}{|\a_1^{R_2}[\l_0,\tilde{\l}](-1)|}\right).
\]

As an intermediate step, we study each of the above terms separately. In first place, we have
\begin{align*}
\a_1^{R_2}[\l_0,\tilde{\l}](1-\k)&=\a_1^{R_2}[\l_0,\tilde{\l}](1)-\k (\l_0-A)2 R_2\\
&=\a_1^{R_2}[\l_0,\tilde{\l}](1)-\k \frac{2B}{R_2}\leq -\k \frac{2B}{R_2}\ ,
\end{align*}
where we have used \eqref{auxIlamba1}. Moreover, we have seen that $\a_1^{R_2}[\l_0,\l_1](z)\leq 0$  for all $|z|\leq 1$ thanks to \eqref{alpha1R2negative}. This implies that
\begin{equation}\label{getl1aux1}
-\frac{\k}{|\a_1^{R_2}[\l_0,\tilde{\l}](1-\k)|} \geq -\frac{R_2}{2B}.
\end{equation}
For the other term, we can ensure that for any $0<\beta \ll 1$ there exist $\k_0'(\beta)$ small enough such that
\begin{equation}\label{getl1aux2}
0 < \frac{|\a_1^{R_2}[\l_0,\tilde{\l}](1-\k)|}{|\a_1^{R_2}[\l_0,\tilde{\l}](-1)|}<\beta,
\end{equation}
for all $\tilde{\l}$ satisfying $0<\l^\ast-\tilde{\l}<\k_0'$ with $0<\k<\k_0'$.

Now, we have all the ingredients to conclude. We just need to combine \eqref{getl1aux1} and \eqref{getl1aux2} to get
\[
\frac{I(\tilde{\l})}{\mathfrak{m}(M)} \geq -\frac{R_2}{2B}\left(C + \frac{1}{2}\log(\beta)\right).
\]
Then, taking $\beta$ small enough such that
\begin{equation}\label{lwgamma}
-\frac{R_2}{2B}\left(C + \frac{1}{2}\log(\beta)\right) > \frac{1}{\mathfrak{m}(M)},
\end{equation}
we have proved that $I(\tilde{\l})>1$ for all $\tilde{\l}$ satisfying $0<\l^\ast-\tilde{\l}<\k_0'(\beta)$ with $0<\k<\k_0'(\beta)$ and $1\leq m< M$. Here, it is important to emphasize that $\beta$ given by \eqref{lwgamma} only depends on $M$.

Combining the lower bound with \eqref{limitl1} and the fact that $I(\tilde{\l})$ is strictly increasing, we 
conclude that there exists a unique $\l_1=\l^\ast-\beta(M)$  solving \eqref{lambda1eq} for all $0 < \k < \k_0'$ and $1 \leq m < M$.
\end{proof}

After choose a suitable $\l_1<\l^\ast$ by Lemma \ref{lema:lambda1}, we obtain $\bb_0$ through equation \eqref{def:b0}. Now,
we have all the ingredients to prove the following result.

\begin{lemma}\label{propo1} For any $M>1,$ there exists $\kappa_0=\kappa_0(M)$ such that for $1\leq m < M$ and $0< \kappa<\kappa_0$, we can find a solution $(\l_1,\, \bb_0)\in \R\times C^\infty([-1,1])$  of equation \eqref{b0equation}. In addition,  fixed $m$ and $\kappa$, this solution is unique (modulo multiplication by constant) and satisfies
\begin{align}
\bb_0(z)&=\frac{1}{\a_1^{R_2}[\l_0,\l_1](z)},\label{b0formula} \\
\lambda_1&=\frac{2B}{R_2^3}\left(1+\frac{2}{\exp\left[\frac{2B}{R_2}\frac{2}{\pp_2(m)}\right]-1}\right)-\frac{\tilde{\Upsilon}_1^{R_2}(0)}{R_2^2}+\mathsf{error}\label{lambda1formula},
\end{align}
where
\begin{align*}
|\mathsf{error}|\leq C(M)\kappa.
\end{align*}
Moreover, we get
\begin{align*}
\|\bb_0\|_{L^\infty([-1,1])}\leq C(M), && \|\pa_z^{k)}\bb_0\|_{L^\infty([-1,\,1])}\leq C(k,M),\quad k=0,1,2,... && \aa_1=\frac{1}{\a_0^{R_1}[\l_0]}\frac{R_1 \sm(R_1/r_1)}{R_2 \sm(R_2/r_2)}.
\end{align*}
\end{lemma}
\begin{remark}
It is important to emphasize that constant $C(M)$ depend on $M$  but it does not depend on neither $m$ or $\kappa$, for $0 < \kappa<\kappa_0$ and $1\leq m<M$.
\end{remark}


\begin{proof}
We start the proof by explaining that, as before, we are considering only the case $B>0$. The remaining  case, i.e. $B<0$ follows my straightforward modifications. 

In first place, we note that 
\[
\bb_0(z)=\frac{1}{\a_1^{R_2}[\l_0,\l_1](z)},
\]
solves \eqref{b0equation} with linear function $\a_1^{R_2}[\l_0,\l_1](z)$ given by
\[
\a_1^{R_2}[\l_0,\l_1](z)=(\l_0-A)2 R_2 z+\l_1 R_2^2+\tilde{\Upsilon}_1^{R_2}(0).
\]
Since $\a_1^{R_2}[\l_0,\l_1](z)<0$ for all $|z|\leq 1$ we obtain  that its inverse, that is, $\bb_0(z)$ is a  $C^\infty$-function on the unit interval $[-1,1]$. In addition, since $\a_1^{R_2}[\l_0,\l^\ast](1)=0$ by definition of $\l^\ast$, we have that $\tilde{\Upsilon}_1^{R_2}(0)=-\l_\ast R_2^2-(\l_0-A)2 R_2$ and consequently we can just write
\begin{equation}\label{auxIIlamba1}
\a_1^{R_2}[\l_0,\l_1](z)=(\l_0-A)2 R_2 (z-1)+(\l_1-\l_\ast) R_2^2.
\end{equation}
Thus, a direct calculation yields 
\[
\|\bb_0\|_{L^\infty([-1,1])}^2=\frac{1}{|\l_1-\l_\ast| R_2^2}.
\]
Recalling that $\l_1$ given by Lemma \ref{lema:lambda1} satisfies  $\l_\ast-\l_1=\beta(M)$ we get that $\|\bb_0\|_{L^2([-1,1])}\leq C(M)$. Proceeding in a similar way, we obtain the same type of bounds for its derivatives, i.e.  
$$\|\pa_z^{k)}\bb_0\|_{L^\infty([-1,\,1])}\leq C(k,M).$$

From \eqref{def:a1} and \eqref{lambda1eq} we have that
\[
\aa_1=\frac{1}{\a_0^{R_1}[\l_0]}\pp_1(m)\int_{-1}^{1} \varphi_\k'(s)\bb_0(s)\ds=\frac{1}{\a_0^{R_1}[\l_0]}\frac{\pp_1(m)}{\pp_2(m)}=\frac{1}{\a_0^{R_1}[\l_0]}\frac{R_1 \sm(R_1/r_1)}{R_2 \sm(R_2/r_2)}.
\]

Finally, in order to get \eqref{lambda1formula} we proceed as follow. From the identity \eqref{lambda1eq} we get
\begin{multline}\label{eq:LAMBDA1}
1=\pp_2(m)\int_{-1}^{1} \frac{-\varphi_\k'( s)-1/2}{|\a_1^{R_2}[\l_0,\l_1](s)|}\ds + \frac{\pp_2(m)}{2}\int_{-1}^{1} \frac{1}{|\a_1^{R_2}[\l_0,\l_1](s)|}\ds\\
=\frac{\pp_2(m)}{2}\left(\Lambda_1+\frac{1}{(\l_0-A)2 R_2}\log\left(\frac{\a_1^{R_2}[\l_0,\l_1](+1)}{\a_1^{R_2}[\l_0,\l_1](-1)}\right)\right),
\end{multline}
where
\[
\Lambda_1:=2\int_{-1}^{1} \frac{-\varphi_\k'( s)-1/2}{|\a_1^{R_2}[\l_0,\l_1](s)|}\ds.
\]
Solving \eqref{eq:LAMBDA1} we get
\[
\l_1 R_2^2 +\tilde{\Upsilon}_1^{R_2}(0)=(\l_0-A)2 R_2\left(1+\frac{2}{\exp\left[(\l_0-A)2 R_2\left(\frac{2}{\pp_2(m)}-\Lambda_1\right)\right]-1}\right),
\]
and using \eqref{auxIlamba1} we finally arrive to
\[
\l_1=\frac{2B}{R_2^3}\left(1+\frac{2}{\exp\left[\frac{2B}{R_2}\left(\frac{2}{\pp_2(m)}-\Lambda_1\right)\right]-1}\right)-\frac{\tilde{\Upsilon}_1^{R_2}(0)}{R_2^2}.
\]
Consequently
\begin{multline*}
\left|\l_1-\frac{2B}{R_2^3}\left(1+\frac{2}{\exp\left[\frac{2B}{R_2}\frac{2}{\pp_2(m)}\right]-1}\right)+\frac{\tilde{\Upsilon}_1^{R_2}(0)}{R_2^2}\right|\\
=\frac{2B}{R_2^3 }\frac{2\exp\left[\frac{2B}{R_2}\frac{2}{\pp_2(m)}\right]\left(1-\exp\left[-\frac{2B}{R_2}\Lambda_1 \right]\right)}{\left(\exp\left[\frac{2B}{R_2}\left(\frac{2}{\pp_2(m)}-\Lambda_1\right)\right]-1\right)\left(\exp\left[\frac{2B}{R_2}\frac{2}{\pp_2(m)}\right]-1\right)}=:\mathsf{error}.
\end{multline*}
Now, we can estimate each of the terms of the last factor that involve the term $\Lambda_1$. Note that the term appearing in the numerator can be bounded by
\[
\left|1-\exp\left[-\frac{2B}{R_2}\Lambda_1 \right]\right|\leq \frac{2B}{R_2}\exp\left[\frac{2B}{R_2}|\Lambda_1|\right]|\Lambda_1|.
\]
In addition, using the fact that $\a_1^{R_2}[\l_0,\l_\ast](1)=0$ we get
\[
|\Lambda_1|\leq 2\int_{-1}^{1} \frac{|\varphi_\k'( s)+1/2|}{|\a_1^{R_2}[\l_0,\l_1](s)|}\ds \leq \frac{C\k}{|\a_1^{R_2}[\l_0,\l_1](1)|} =\frac{C \k}{(\l_\ast-\l_1)R_2^2}
\]
where $C$ is a universal constant coming from Lemma \ref{propiedadesvarphi}. Proceeding as before, since $\l_1$ given by Lemma \ref{lema:lambda1} satisfies  $0<\l_\ast-\l_1=\beta(M)$ we have proved that
\[
|\Lambda_1|\leq \frac{C \k}{\beta(M) R_2^2}.
\]
For the denominator it is sufficient to note that
\begin{equation*}
\pp_2(m)\leq  R_2^2\frac{\cS_{M}(R_2/r_1) \cS_{M}(r_2/R_2)}{\cS_1(r_2/r_1)}=:\mathfrak{M}(M), \qquad \text{for all}\quad 1\leq m < M,
\end{equation*}
and
\begin{multline*}
\exp\left[\frac{2B}{R_2}\left(\frac{2}{\pp_2(m)}-\Lambda_1\right)\right]-1\\
\geq \exp\left[\frac{2B}{R_2}\left(\frac{2}{\mathfrak{M}(M)}-\frac{C \k}{\beta(M) R_2^2}\right)\right]-1>C(M), \qquad \text{if} \quad \frac{2}{\mathfrak{M}(M)}>\frac{C \k}{\beta(M) R_2^2}.
\end{multline*}
Then, taking $\k<\k_0\equiv \min\{\k_0',\k_0''\}$, where $\k_0'$  is given by Lemma \ref{lema:lambda1} and $\k_0''=\frac{2\beta(M)R_2^2}{C \,\mathfrak{M}(M)}$, we have proved that $|\mathsf{error}|\leq C(M)\k$ for all $0<\k<\k_0$ and $1\leq m<M$.
\end{proof}

\subsubsection{Solving the lower order terms via a fixed point argument}
The main task of this section is to solve \eqref{systemTeQe} for the unknowns $\aa_2^\e, \bb_1^\e$ and $\l_2^\e$. We rewrite it here for convenience.
\begin{equation*}
\left\{
\begin{aligned}
\left(\cT_{m,\e}^2[\aa_2^\e,\bb_1^\e;\aa_1,\l_0,\l_1]+\frac{1}{\e}\int_0^\e \left(\p_\d\mathcal{T}_{m,\d}^1\right)[\bb_0]\mbox{d}\delta\right)  +\cT_{m,\e}^3[\aa_2^\e,\l_2^\e;\aa_1,\l_0,\l_1]\e&=0,\\
\left(\mathcal{Q}_{m,\e}^2[\bb_1^\e,\l_2^\e;\aa_1,\bb_0,\l_0,\l_1]+\frac{1}{\e}\int_0^\e (\p_\d\mathcal{Q}_{m,\d}^1)[\bb_0]\mbox{d}\delta\right)  +\mathcal{Q}_{m,\e}^3[\aa_2^\ep,\bb_1^\e,\l_2^\e;\l_0,\l_1]\e &=0.
\end{aligned}
\right.
\end{equation*}
We will proceed as before by considering the Taylor expansion of $\cT_{m,\e}^2$ and $\mathcal{Q}_{m,\e}^2$ around $\e=0$.
Before that, we need to take care of the fact that term $\mathsf{Remainder}_{\e,\k}$, which appears on $\a_{2,\e}^{R_2}[\l_0,\l_1,\l_2^\e]$ through the term $\tilde{\Upsilon}_{\e,\k}^{R_2}(z)$, satisfies that $\mathsf{Remainder}_{\e,\k}=O(1)$ but $\left(\p_\e\mathsf{Remainder}_{\e,\k}\right)(z)=O(\e^{-1})$ in terms of $\e$. 

With a little abuse of notation, in the rest of this section, we redefine term $\a_{2,\e}^{R_2}[\l_0,\l_1,\l_2^\e](z)$ as (see \eqref{def:alpha2})
\begin{equation}\label{def:alpha2new}
\a_{2,\e}^{R_2}[\l_0,\l_1,\l_2^\e](z):=\l_0 z^2 +\l_1 2 z R_2 + \l_2^\e R_2^2+\l_2^\e 2z R_2\e+\l_2^\e z^2\e^2,
\end{equation}
where the term $\tilde{\Upsilon}_{\e,\k}^{R_2}(z)$ will be study separately.

After this, the main task of this section is to solve the coupled system
\begin{equation*}
\left(\cT_{m,\e}^2[\aa_2^\e,\bb_1^\e;\aa_1,\l_0,\l_1]+\frac{1}{\e}\int_0^\e \left(\p_\d\mathcal{T}_{m,\d}^1\right)[\bb_0]\mbox{d}\delta\right)  +\cT_{m,\e}^3[\aa_2^\e,\l_2^\e;\aa_1,\l_0,\l_1]\e=0
\end{equation*}
together with
\begin{multline*}
\left(\tilde{\Upsilon}_{\e,\k}^{R_2}(z)\bb_0(z)+ \mathcal{Q}_{m,\e}^2[\bb_1^\e,\l_2^\e;\aa_1,\bb_0,\l_0,\l_1](z)+\frac{1}{\e}\int_0^\e (\p_\d\mathcal{Q}_{m,\d}^1)[\bb_0](z)\mbox{d}\delta\right)\\
  +\mathcal{Q}_{m,\e}^3[\aa_2^\ep,\bb_1^\e,\l_2^\e;\l_0,\l_1](z)\e =0.
\end{multline*}
Now, we can use Taylor's expansion to obtain
\begin{align*}
\cT_{m,\e}^2[\aa_2^\e,\bb_1^\e;\aa_1,\l_0,\l_1]&=\cT_{m,0}^2[\aa_2^\e,\bb_1^\e;\aa_1,\l_0,\l_1]+\int_0^\e  \left(\p_\d \cT_{m,\d}^2 \right) [\bb_1^\e;\aa_1] \mbox{d}\delta,\\
\mathcal{Q}_{m,\e}^2[\bb_1^\e,\l_2^\e;\aa_1,\bb_0,\l_0,\l_1]&=\mathcal{Q}_{m,0}^2[\bb_1^\e,\l_2^\e;\aa_1,\bb_0,\l_0,\l_1]+\int_0^\e \left(\p_\d \mathcal{Q}_{m,\d}^2 \right)[\bb_1^\e,\l_2^\e;\aa_1,\bb_0] \mbox{d}\d,
\end{align*}
where (see \eqref{T2e}, \eqref{Q2e})
\begin{multline}\label{T20}
\cT_{m,0}^2[\aa_2^\e,\bb_1^\e;\aa_1,\l_0,\l_1]=\a_1^{R_1}[\l_0,\l_1] \aa_1(z)+\a_0^{R_1}[\l_0]\aa_2^{\e}(z)\\
+ R_1^2\frac{\sm(R_1/r_1)\sm(r_2/R_1)}{\sm(r_2/r_1)}\frac{1}{m}\int_{-1}^{+1}  \varphi_\k' (-s)\aa_1(s)\ds\\
-R_1R_2 \frac{\sm(R_1/r_1)\sm(r_2/R_2)}{\sm(r_2/r_1)}\frac{1}{m}\int_{-1}^{+1} \varphi_\k' (s)\bb^\ep_1(s) \ds,
\end{multline}
and
\begin{multline}\label{Q20}
\mathcal{Q}_{m,0}^2[\bb_1^\e,\l_2^\e;\aa_1,\bb_0,\l_0,\l_1]=\a_1^{R_2}[\l_0,\l_1]\bb_1^\e(z)+ \a_{2,0}^{R_2}[\l_0,\l_1,\l_2^\e]\mathsf{b}_0(z)\\
+R_1 R_2\frac{\sm(R_2/r_1)\sm(r_2/R_1)-\sm(R_2/R_1)\sm(r_2/r_1)}{\sm(r_2/r_1)}\frac{1}{m}\int_{-1}^{+1} \varphi_\k' (-s)\aa_1(s)\ds\\
-R_2^2\frac{\sm(R_2/r_1)\sm(r_2/R_2)}{\sm(r_2/r_1)}\frac{1}{m}\int_{-1}^{+1} \varphi_\k' (s) \bb^\ep_1(s)\ds.
\end{multline}
\begin{remark}
Here, it is important to emphasize that term $\a_{2,\e}^{R_2}[\l_0,\l_1,\l_2^\e]$ is defined via \eqref{def:alpha2new}.
\end{remark}
In fact, we have
\begin{align*}
\cT_{m,0}^2[\aa_2^\e,\bb_1^\e;\aa_1,\l_0,\l_1]&=\widehat{\cT}_m^2[\aa_2^\e,\bb_1^\e;\l_0]+\widetilde{\cT}_m^2[\aa_1,\l_0,\l_1],\\
\mathcal{Q}_{m,0}^2[\bb_1^\e,\l_2^\e;\aa_1,\bb_0,\l_0,\l_1]&=\widehat{\cQ}_m^2[\bb_1^\e,\l_2^\e ;\bb_0,\l_0,\l_1]+\widetilde{\cQ}_m^2[\aa_1]
\end{align*}
with
\begin{align*}
\widehat{\cT}_m^2[\aa_2^\e,\bb_1^\e;\l_0]&:=\a_0^{R_1}[\l_0]\aa_2^{\e}(z)-R_1R_2 \frac{\sm(R_1/r_1)\sm(r_2/R_2)}{\sm(r_2/r_1)}\frac{1}{m}\int_{-1}^{+1} \varphi_\k' (s)\bb^\ep_1(s) \ds,\\
\widetilde{\cT}_m^2[\aa_1,\l_0,\l_1]&:=\a_1^{R_1}[\l_0,\l_1] \aa_1(z) + R_1^2\frac{\sm(R_1/r_1)\sm(r_2/R_1)}{\sm(r_2/r_1)}\frac{1}{m}\int_{-1}^{+1}  \varphi_\k' (-s)\aa_1(s)\ds,
\end{align*}
and
\begin{multline}\label{Q2hat}
\widehat{\cQ}_m^2[\bb_1^\e,\l_2^\e ;\bb_0,\l_0,\l_1]:=\a_1^{R_2}[\l_0,\l_1]\bb_1^\e(z)+ \a_{2,0}^{R_2}[\l_0,\l_1,\l_2^\e]\mathsf{b}_0(z)\\
 -R_2^2\frac{\sm(R_2/r_1)\sm(r_2/R_2)}{\sm(r_2/r_1)}\frac{1}{m}\int_{-1}^{+1} \varphi_\k' (s) \bb^\ep_1(s)\ds,
\end{multline}
\begin{equation*}
\widetilde{\cQ}_m^2[\aa_1]:=R_1 R_2\frac{\sm(R_2/r_1)\sm(r_2/R_1)-\sm(R_2/R_1)\sm(r_2/r_1)}{\sm(r_2/r_1)}\frac{1}{m}\int_{-1}^{+1} \varphi_\k' (-s)\aa_1(s)\ds.
\end{equation*}

\begin{remark}
From now on we will repeatedly use the decomposition $\mathcal{T}[\mathsf{p};\mathsf{q}]=\widehat{\mathcal{T}}[\mathsf{p};\mathsf{q}]+\widetilde{\mathcal{T}}[\mathsf{q}]$ to emphasize that unknown appear only on the term $\widehat{\mathcal{T}}[\mathsf{p};\mathsf{q}]$ and the other is fully determined. 
\end{remark}

Taking into account all the above considerations, the system \eqref{systemTeQe} can be written as
\begin{multline*}
\widehat{\cT}_m^2[\aa_2^\e,\bb_1^\e;\l_0](z)+ \left( \widetilde{\cT}_m^2[\aa_1,\l_0,\l_1](z)+\frac{1}{\e}\int_0^\e \left(\p_\d\mathcal{T}_{m,\d}^1\right)[\bb_0](z)\mbox{d}\delta \right)\\
+\e\left(\frac{1}{\e}\int_0^\e  \left(\p_\d \cT_{m,\d}^2 \right) [\bb_1^\e;\aa_1](z) \mbox{d}\delta + \cT_{m,\e}^3[\aa_2^\e,\l_2^\e;\aa_1,\l_0,\l_1,](z)\right)=0,
\end{multline*}
and
\begin{multline*}
\widehat{\cQ}_m^2[\bb_1^\e,\l_2^\e ;\bb_0,\l_0,\l_1](z)
+\left(\widetilde{\cQ}_m^2[\aa_1](z)+\tilde{\Upsilon}_{\e,\k}^{R_2}(z)\bb_0(z)+\frac{1}{\e}\int_0^\e (\p_\d\mathcal{Q}_{m,\d}^1)[\bb_0](z)\mbox{d}\delta\right)\\
+ \e \left(\frac{1}{\e}\int_0^\e \left(\p_\d \mathcal{Q}_{m,\d}^2 \right)[\bb_1^\e,\l_2^\e;\aa_1,\bb_0](z) \mbox{d}\d +\mathcal{Q}_{m,\e}^3[\aa_2^\ep,\bb_1^\e,\l_2^\e;\l_0,\l_1](z)\right)=0.
\end{multline*}
Here, it is important to note that each of the terms in parentheses is of order $O(1)$ in terms of $\e$.
Consequently, we finally get for $\aa_2^\e, \bb_1^\e$ and $\l_2^\e$ the system
\begin{align}
\widehat{\cT}_m^2[\aa_2^\e,\bb_1^\e;\l_0](z)&=A_0[\aa_1,\bb_0,\l_0,\l_1](z)+\e A_1[\aa_2^\e,\bb_1^\e,\l_2^\e;\aa_1,\l_1,\l_0](z), \label{invertT}\\
\widehat{\cQ}_m^2[\bb_1^\e,\l_2^\e ;\bb_0,\l_0,\l_1](z)&=B_0[\aa_1,\bb_0,\l_0,\l_1](z)+\e B_1[\aa_2^\e,\bb_1^\e,\l_2^\e;\aa_1,\bb_0,\l_0,\l_1](z), \label{invertQ}
\end{align}
where
\begin{align*}
A_0[\aa_1,\bb_0,\l_0,\l_1](z):&=-\widetilde{\cT}_m^2[\aa_1,\l_0,\l_1](z)-\frac{1}{\e}\int_0^\e \left(\p_\d\mathcal{T}_{m,\d}^1\right)[\aa_1,\bb_0,\l_0](z)\mbox{d}\delta,\\
B_0[\aa_1,\bb_0,\l_0,\l_1](z):&=-\widetilde{\cQ}_m^2[\aa_1](z)-\tilde{\Upsilon}_{\e,\k}^{R_2}(z)\bb_0(z)-\frac{1}{\e}\int_0^\e (\p_\d\mathcal{Q}_{m,\d}^1)[\bb_0,\l_0,\l_1](z)\mbox{d}\delta,
\end{align*}
and
\begin{equation*}
A_1[\aa_2^\e,\bb_1^\e,\l_2^\e;\aa_1,\l_0,\l_1](z):=-\frac{1}{\e}\int_0^\e  \left(\p_\d \cT_{m,\d}^2 \right) [\bb_1^\e;\aa_1](z) \mbox{d}\delta
- \cT_{m,\e}^3[\aa_2^\e,\l_2^\e;\aa_1,\l_0,\l_1](z),
\end{equation*}
\begin{equation*}
B_1[\aa_2^\e,\bb_1^\e,\l_2^\e;\aa_1,\bb_0,\l_0,\l_1](z):=-\frac{1}{\e}\int_0^\e \left(\p_\d \mathcal{Q}_{m,\d}^2 \right)[\bb_1^\e,\l_2^\e;\aa_1,\bb_0](z) \mbox{d}\d
 -\mathcal{Q}_{m,\e}^3[\aa_2^\ep,\bb_1^\e,\l_2^\e;\l_0,\l_1](z).
\end{equation*}
Here, we note that Lemma \ref{propo1} together with \eqref{DFT1Q1} yields
\begin{align*}
\|A_0\|_{L^{2}}&\leq C(M),\\
\|B_0\|_{L^{2}}&\leq C(M).
\end{align*}
To handle $A_1$ and $B_1$ we have to make an intermediate step. Applying Lemma \ref{lema:lambda1} into definition \eqref{T3e} and \eqref{Q3e} we get
\begin{align*}
\|\cT^3_{m,\e}\|_{L^2}\leq &C(M)\left(\|\aa_2^\e\|_{L^2} + |\l_2^\e|+|\l_2^\e|\|\aa_2^\e\|_{L^2}\right),\\
\|\cQ^3_{m,\e}\|_{L^2}\leq &C(M)\left(|\l_2^\e|\|\bb_1^\e\|_{L^2} + \|\aa_2^\e\|_{L^2}\right).
\end{align*}
Combining the above with Lemma \ref{propo1} together with \eqref{DFT2} and \eqref{DFQ2} we obtain
\begin{align*}
\|A_1\|_{L^{2}}&\leq C(M)\left(\|\bb_1^\e\|_{L^2}+ \|\aa_2^\e\|_{L^2} + |\l_2^\e|+|\l_2^\e|\|\aa_2^\e\|_{L^2}\right),\\
\|B_1\|_{L^{2}}&\leq C(M)\left(|\l_2^\e|+\|\bb_1^\e\|_{L^2}+|\l_2^\e|\|\bb_1^\e\|_{L^2} + \|\aa_2^\e\|_{L^2}\right).
\end{align*}

\subsubsection{Inverting the operator $\widehat{\cQ}_m^2$}
In the following two lemmas we are going to learn how to invert the left hand side of \eqref{invertQ}. In order to do that, we write explicitly \eqref{Q2hat}, that is the left hand side of \eqref{invertQ}:
\begin{equation*}
\widehat{\cQ}_m^2[\bb_1^\e,\l_2^\e ;\bb_0,\l_0,\l_1]:=\a_1^{R_2}[\l_0,\l_1]\bb_1^\e(z)+ \a_{2,0}^{R_2}[\l_0,\l_1,\l_2^\e]\mathsf{b}_0(z) -\pp_2(m)\int_{-1}^{+1} \varphi_\k' (s) \bb^\ep_1(s)\ds,
\end{equation*}
where in the last term we have used \eqref{def:ppi}.

In order to invert the above operator we are going to split the argument in two parts. Firstly, we are going to start inverting a simplified expression as follows. We omit their proof because it follows by direct computation.
\begin{lemma}\label{auxlemmainvert}
Let $\l_0$ given by expression \eqref{def:l0}, $\l_1$ be as in Lemma \ref{lema:lambda1} and let $\mathsf{F}\in L^2([-1,1])$ such that $\int_{-1}^{1}\mathsf{F}(s)\bb_0(s)\varphi_\k'(s)\ds=0$. Then
\[
\gg(z):=\mathsf{F}(z)\bb_0(z)=\frac{\mathsf{F}(z)}{\a_1^{R_2}[\l_0,\l_1](z)},
\]
solves
\[
\a_1^{R_2}[\l_0,\l_1](z)\gg(z)-\pp_2(m)\int_{-1}^{1} \varphi_\k'(s)\gg(s)\emph{d}s =\mathsf{F}(z).
\]
\end{lemma}

Here, we recall expression \eqref{def:alpha2new} that  allows us to write
\begin{align*}
\a_{2,0}^{R_2}[\l_2^\e;\l_0,\l_1](z)=\l_0 z^2 +\l_1 2 z R_2 + \l_2^\e R_2^2=:\l_2^\e R_2^2+\beta[\l_0,\l_1](z).
\end{align*}
Now, we have all the ingredients to  invert $\widehat{\cQ}_m^2$.

\begin{lemma}\label{lemmainvertQ2}
Let $\l_0, \l_1$ be as in Lemma \ref{auxlemmainvert} and let $\mathsf{G}\in L^2([-1,1])$. Then, taking
\[
\mu :=\frac{1}{R_2^2}\frac{\int_{-1}^{1} \left(\mathsf{G}(s)-\beta[\l_0,\l_1](s)\right)\bb_0(s)\varphi_\k'(s)\ds}{\int_{-1}^{1}|\bb_0(s)|^2\varphi_\k'(s)\ds},
\]
the function
\[
\gg(z):=\left(\mathsf{G}(z)-\mu R_2^2 \bb_0(z)\right)\bb_0(z)=\frac{\mathsf{G}(z)-\mu R_2^2 \bb_0(z)}{\a_1^{R_2}[\l_0,\l_1](z)},
\]
solves
\[
 \a_1^{R_2}[\l_0,\l_1](z)\gg(z)+\left(\mu  R_2^2+\beta[\l_0,\l_1](z)\right)\bb_0(z) - \pp_2(m)\int_{-1}^{1} \varphi_\k'(s)\gg(s)\emph{d}s=\mathsf{G}(z).
\]
That is, we  solve
\begin{equation}\label{solinvertQ2hat}
\widehat{\cQ}_m^2[\gg,\mu;\bb_0,\l_0,\l_1](z)=\mathsf{G}(z),
\end{equation}
and we have the bound
\begin{equation}\label{boundmu}
|\mu|\leq  C(M)\|\mathsf{G}-\beta[\l_0,\l_1]\|_{L^2([-1,1])}.
\end{equation}
\end{lemma}
\begin{proof}
It is immediate to verify that $\mathsf{F}(z):=\mathsf{G}(z)-\left(\mu  R_2^2+\beta[\l_0,\l_1](z)\right)\bb_0(z)$ satisfies that 
\[
\int_{-1}^{1}F(s)\bb_0(s)\varphi_\k'(s)\ds=0
\]
by the definition of $\mu$. Then, we can apply directly Lemma \ref{auxlemmainvert} to obtain \eqref{solinvertQ2hat}.

To conclude \eqref{boundmu}, on the one hand, we note that
\begin{align*}
\left|\int_{-1}^{1} \left(\mathsf{G}(s)-\beta[\l_0,\l_1](s)\right)\bb_0(s)\varphi_\k'(s)\ds\right| &\leq \|\mathsf{G}-\beta[\l_0,\l_1]\|_{L^2([-1,1])} \|\bb_0\|_{L^2([-1,1])} \|\varphi_\k'\|_{L^\infty([-1,1])}\\
&\leq C(M)\|\mathsf{G}-\beta[\l_0,\l_1]\|_{L^2([-1,1])},
\end{align*}
where in the last step we have used Lemmas \ref{propo1} and \ref{propiedadesvarphi}.

On the other hand, since $\bb_0(z)=\left(\a_1^{R_2}[\l_0,\l_1](z)\right)^{-1}$ and $\a_1^{R_2}[\l_0,\l_1](z)\leq \a_1^{R_2}[\l_0,\l_1](1)< 0$ we have
\[
\left|\int_{-1}^{1}|\bb_0(s)|^2\varphi_\k'(s)\ds \right| \geq \frac{1}{|\a_1^{R_2}[\l_0,\l_1](1)|}\int_{-1}^{1}\frac{\varphi_\k'(s)}{\a_1^{R_2}[\l_0,\l_1](s)}\ds=\frac{\pp_2(m)}{|\a_1^{R_2}[\l_0,\l_1](1)|},
\]
where in the last step we have used \eqref{lambda1eq}. In fact, applying the lower bound \eqref{LBpp2} and the fact that $\a_1^{R_2}[\l_0,\l_1](1)=(\l_1-\l_\ast)R_2^2$ with $\l_1$ given by Lemma \ref{lema:lambda1} we arrive to
\[
\left|\int_{-1}^{1}|\bb_0(s)|^2\varphi_\k'(s)\ds \right| \geq \frac{\mathfrak{m}(M)}{\beta(M)R_2^2}.
\]
Combining both we have
\[
|\mu|\leq \frac{1}{R_2^2} \frac{\beta(M)R_2^2}{\mathfrak{m}(M)} C(M)\|\mathsf{G}-\beta[\l_0,\l_1]\|_{L^2([-1,1])}\leq C(M)\|\mathsf{G}-\beta[\l_0,\l_1]\|_{L^2([-1,1])}.
\]
\end{proof}

\subsubsection{The fixed point argument}
At this point we have all the required ingredients to solve the system \eqref{invertT}, \eqref{invertQ} for $\aa_2^\e,\bb_1^\e$ and $\l_2^\e$. We rewrite the system here for convenience:
\begin{align*}
\widehat{\cT}_m^2[\aa_2^\e,\bb_1^\e;\l_0](z)&=A_0[\aa_1,\bb_0,\l_0,\l_1](z)+\e A_1[\aa_2^\e,\bb_1^\e,\l_2^\e;\aa_1,\l_1,\l_0](z), \\
\widehat{\cQ}_m^2[\bb_1^\e,\l_2^\e ;\bb_0,\l_0,\l_1](z)&=B_0[\aa_1,\bb_0,\l_0,\l_1](z)+\e B_1[\aa_2^\e,\bb_1^\e,\l_2^\e;\aa_1,\bb_0,\l_0,\l_1](z).
\end{align*}
With a little abuse of notation, since $\aa_1,\bb_0,\l_0,\l_1$ are already determined at this point of the proof, we just write the system in terms of the unknowns $\aa_2^\e,\bb_1^\e,\l_2^\e$. That is,
\begin{align*}
\widehat{\cT}_m^2[\aa_2^\e,\bb_1^\e](z)&=A_0(z)+\e A_1[\aa_2^\e,\bb_1^\e,\l_2^\e](z), \\
\widehat{\cQ}_m^2[\bb_1^\e,\l_2^\e](z)&=B_0(z)+\e B_1[\aa_2^\e,\bb_1^\e,\l_2^\e](z).
\end{align*}
Now, inverting $\widehat{\cQ}_m^2$ via Lemma \ref{lemmainvertQ2} we get
\begin{align}
&\l_2^\e =\frac{1}{R_2^2}\frac{\int_{-1}^{1} \left(B_0(s)+\e B_1[\aa_2^\e,\bb_1^\e,\l_2^\e](s)-\beta[\l_0,\l_1](s)\right)\bb_0(s)\varphi_\k'(s)\ds}{\int_{-1}^{1}|\bb_0(s)|^2\varphi_\k'(s)\ds}, \label{eql2eaux}\\
&\bb_1^\e(z)=\left(B_0(z)+\e B_1[\aa_2^\e,\bb_1^\e,\l_2^\e](z)-\l_2^\e R_2^2 \bb_0(z)\right)\bb_0(z), \label{eqb1eaux}
\end{align}
and rewriting the equation for $\widehat{\cT}_m^2$ in an explicit way, we have
\begin{equation}\label{eqa2eaux}
\a_0^{R_1}[\l_0]\aa_2^{\e}(z)-R_1 R_2 \frac{\sm(r_1/R_1) \sm(R_2/r_2)}{m\sm(r_1/r_2)}\int_{-1}^{1} \varphi_\k'( s)\bb_1^\e(s)\ds=A_0(z)+\e A_1[\aa_2^\e,\bb_1^\e,\l_2^\e](z). 
\end{equation}
Then, solving \eqref{eql2eaux}, \eqref{eqb1eaux}, \eqref{eqa2eaux} will give us a solution of our original system \eqref{invertT}, \eqref{invertQ}.

Now, to obtain a solution of \eqref{eql2eaux}, \eqref{eqb1eaux} and \eqref{eqa2eaux} we firstly introduce \eqref{eql2eaux} and \eqref{eqb1eaux} in the second term of the left hand side of \eqref{eqa2eaux} and \eqref{eql2eaux} in the right hand side of \eqref{eqb1eaux}. In this way, a solution of
\begin{align}
\l_2^\e &=\frac{1}{R_2^2}\frac{\int_{-1}^{1} \left(B_0(s)+\e B_1[\aa_2^\e,\bb_1^\e,\l_2^\e](s)-\beta[\l_0,\l_1](s)\right)\bb_0(s)\varphi_\k'(s)\ds}{\int_{-1}^{1}|\bb_0(s)|^2\varphi_\k'(s)\ds}, \label{eql2eauxbis}\\
\bb_1^\e(z)&=\widetilde{B}_0(z) +\e\widetilde{B}_1[\aa_2^\e,\bb_1^\e,\l_2^\e](z), \label{eqb1eauxbis}\\
\a_0^{R_1}[\l_0]\aa_2^{\e}(z)&=\widetilde{A}_0(z)+\e \widetilde{A}_1[\aa_2^\e,\bb_1^\e,\l_2^\e](z),\label{eqa2eauxbis}
\end{align}
with
\begin{align*}
\widetilde{B}_0(z)&:=\bb_0(z)\left(B_0(z) -\bb_0(z)\frac{\int_{-1}^{1} \left(B_0(s)-\beta[\l_0,\l_1](s)\right)\bb_0(s)\varphi_\k'(s)\ds}{\int_{-1}^{1}|\bb_0(s)|^2\varphi_\k'(s)\ds}\right),\\
\widetilde{B}_1[\aa_2^\e,\bb_1^\e,\l_2^\e](z)&:=\bb_0(z)\left(B_1[\aa_2^\e,\bb_1^\e,\l_2^\e](z)-\bb_0(z)\frac{\int_{-1}^{1}  B_1[\aa_2^\e,\bb_1^\e,\l_2^\e](s)\bb_0(s)\varphi_\k'(s)\ds}{\int_{-1}^{1}|\bb_0(s)|^2\varphi_\k'(s)\ds} \right),
\end{align*}
and
\begin{multline*}
\widetilde{A}_0(z):=A_0(z)+R_1 R_2 \frac{\sm(r_1/R_1) \sm(R_2/r_2)}{m\sm(r_1/r_2)}\int_{-1}^{1} \varphi_\k'( s)B_0(s)\bb_0(s)\ds\\
-\left(\frac{\int_{-1}^{1} \left(B_0(s)-\beta[\l_0,\l_1](s)\right)\bb_0(s)\varphi_\k'(s)\ds}{\int_{-1}^{1}|\bb_0(s)|^2\varphi_\k'(s)\ds}\right) R_1 R_2 \frac{\sm(r_1/R_1) \sm(R_2/r_2)}{m\sm(r_1/r_2)}\int_{-1}^{1} \varphi_\k'( s)|\bb_0(s)|^2\ds,
\end{multline*}
\begin{multline*}
\widetilde{A}_1[\aa_2^\e,\bb_1^\e,\l_2^\e](z):=A_1[\aa_2^\e,\bb_1^\e,\l_2^\e](z)+R_1 R_2 \frac{\sm(r_1/R_1) \sm(R_2/r_2)}{m\sm(r_1/r_2)}\int_{-1}^{1} \varphi_\k'( s)B_1[\aa_2^\e,\bb_1^\e,\l_2^\e](s)\bb_0(s)\ds\\
-\left(\frac{\int_{-1}^{1} B_1[\aa_2^\e,\bb_1^\e,\l_2^\e](s)\bb_0(s)\varphi_\k'(s)\ds}{\int_{-1}^{1}|\bb_0(s)|^2\varphi_\k'(s)\ds}\right)\int_{-1}^{1} \varphi_\k'( s)|\bb_0(s)|^2\ds,
\end{multline*}
will give us a solution of \eqref{eql2eaux}, \eqref{eqb1eaux} and \eqref{eqa2eaux}.\\

In order to solve \eqref{eql2eauxbis}, \eqref{eqb1eauxbis} and \eqref{eqa2eauxbis} we can use a contraction argument on parameter $\ep$. Indeed, let us call $F_\l[\aa_2^\e , \bb_1^\e, \l_2^\e](z)$, $F_B[\aa_2^\e , \bb_1^\e, \l_2^\e](z)$ and $F_A[\aa_2^\e , \bb_1^\e, \l_2^\e](z)$ to the right hand side of \eqref{eql2eauxbis}, \eqref{eqb1eauxbis} and \eqref{eqa2eauxbis} respectively. We define the constants values
\[
C_{\lambda}:=\left|\frac{1}{R_2^2}\frac{\int_{-1}^{1} \left(B_0(s)-\beta[\l_0,\l_1](s)\right)\bb_0(s)\varphi_\k'(s)\ds}{\int_{-1}^{1}|\bb_0(s)|^2\varphi_\k'(s)\ds} \right|,
\]
and 
\[
C_B:= \|\tilde{B}_0\|_{L^2([-1,1])}, \quad C_{A}:=\|\tilde{A}_0\|_{L^2([-1,1])}.
\]
\begin{remark}
Let us emphasize that $C_A$, $C_B$ and $C_\lambda$ are uniformly bounded by some $C(M).$
\end{remark}

Note that if  $|\lambda_{2}^\ep|<2C_{\lambda}$, $\|\bb_1^\ep\|_{L^2([-1,1])}< 2C_{B}$ and $\|\a_0^{R_1}[\l_0] \aa_2^\ep \|_{L^2([-1,1])}< 2 C_A$, taking  $\e$ small enough (with respect to a constant that just depends on $M$) we have
\begin{align*}
|F_{\lambda}[\aa_2^\e , \bb_1^\e, \l_2^\e]|<2C_{\lambda},\\
\|F_B[\aa_2^\e , \bb_1^\e, \l_2^\e]\|_{L^2([-1,1])}<2 C_B,\\
\|F_A[\aa_2^\e , \bb_1^\e, \l_2^\e]\|_{L^2([-1,1])}<2 C_A.
\end{align*}
In addition, for any pair 
$$u=(\a_0^{R_1}[\l_0] \aa_2^\ep, \bb_1^\ep, \lambda_2^\ep), \qquad \tilde{u}=(\a_0^{R_1}[\l_0]\tilde{\aa}_2^\ep, \tilde{\bb}_1^\ep, \tilde{\lambda}_2^\ep)$$
satisfying $|\lambda_{2}^\ep, \tilde{\lambda}^\ep_2|<2C_{\lambda}$, $\|\bb_1^\ep, \tilde{\bb}_1^\ep \|_{L^2([-1,1])}< 2C_{B}$ and $\|\a_0^{R_1}[\l_0]\aa_2^\ep, \a_0^{R_1}[\l_0]\tilde{\aa}^\ep\|_{L^2([-1,1])}< 2 C_A$ we have that
\begin{align*}
|F_\lambda[u]-F_\lambda[\tilde{u}]|\leq \ep C(M)\|u-\tilde{u}\|_{L^2([-1,1])\times L^2([-1,1])\times \R},\\
\|F_A[u]-F_A[\tilde{u}]\|_{L^2([-1,1])}\leq \ep C(M)\|u-\tilde{u}\|_{L^2([-1,1])\times L^2([-1,1])\times \R},\\
\|F_B[u]-F_B[\tilde{u}]\|_{L^2([-1,1])}\leq \ep C(M)\|u-\tilde{u}\|_{L^2([-1,1])\times L^2([-1,1])\times \R}.
\end{align*}
Taking $\e_0:=C(M)^{-1},$ we have that right hand side of \eqref{eql2eaux}, \eqref{eqb1eaux}, \eqref{eqa2eaux}  is a contraction mapping for all $0<\ep<\e_0(M).$ This implies that there exist $(\aa_2^\ep, \bb_1^\ep, \lambda_{2}^\ep)\in L^2([-1,1])\times L^2([-1,1])\times \R$, with $\|\a_0^{R_1}[\l_0] \aa_2^\ep\|_{L^2([-1,1])}<2C_A$, $\|\bb_1^\ep \|_{L^2([-1,1])}<{2C_B}$ and $|\lambda_{2}^\ep|<2C_{\lambda}$ solving \eqref{eql2eauxbis}, \eqref{eqb1eauxbis} and \eqref{eqa2eauxbis} for all $0<\ep<\ep_0(M)$.

Therefore, we have shown the existence of a solution $(\lambda,\aa,\bb)\in \R\times L^2([-1,1])\times L^2([-1,1])$ solving \eqref{LR1cero} and \eqref{LR2cero}, with

\begin{align}
\aa(z)=&  \aa_1 \ep + \aa_2^\ep(z) \ep^2,\label{asolucion}\\
\bb(z)=& \, \bb_0(z)+\bb_1^\ep(z) \ep,\label{bsolucion}\\
\lambda  = & \, \l_0+\lambda_1\ep +\lambda_2^\ep \ep^2\label{lambdasolucion},
\end{align}
with $\l_0$ given by \eqref{def:l0} and  $\aa_1, \bb_0, \l_1$ given by Lemma \ref{propo1}
and satisfying
\begin{align}
\|\aa_2^\ep\|_{L^2([-1,1])}\leq C(M),\quad \|\bb_1^\ep\|_{L^2([-1,1])}\leq C(M),\quad |\lambda_2^\ep|\leq C(M)\label{cotas}.
\end{align}
Next we shall show that this solution is  smooth.

\subsection{Regularity}
To determine the exact regularity of the constructed solution $(\aa, \bb)$ we proceed as follows. First, we recall that the system \eqref{LR1cero}, \eqref{LR2cero} can be written using \eqref{finalLnR1}, \eqref{finalLnR2} as
\begin{multline}\label{auxregularityaa}
\frac{\Lambda_{m,\e}^{R_1}(z)}{R_1+\e z}\aa(z)=\frac{\e}{n}\int_{-1}^{z} (R_1+\e s)\varphi_\k' (-s)\sn((R_1+\e z)/(R_1+\e s))\aa(s)\ds\\
-\frac{\sn((R_1+\e z)/r_1)}{\sn(r_2/r_1)}\frac{\e}{n}\int_{-1}^{+1} (R_1+\e s)\varphi_\k' (-s)\sn(r_2/(R_1+\e s))\aa(s)\ds\\
+\frac{\sn((R_1+\e z)/r_1)}{\sn(r_2/r_1)}\frac{\e}{n}\int_{-1}^{+1} (R_2+\e s)\varphi_\k' (s)\sn(r_2/(R_2+\e s))\bb(s)\ds,
\end{multline}

\begin{multline}\label{auxregularitybb}
\frac{\Lambda_{m,\e}^{R_2}(z)}{R_2+\e z}\bb(z)=\frac{\e}{n}\int_{-1}^{+1} (R_1+\e s)\varphi_\k' (-s)\sn((R_2+\e z)/(R_1+\e s))\aa(s)\ds\\
-\frac{\sn((R_2+\e z)/r_1)}{\sn(r_2/r_1)}\frac{\e}{n}\int_{-1}^{+1} (R_1+\e s)\varphi_\k' (-s)\sn(r_2/(R_1+\e s))\aa(s)\ds \\
+\frac{\sn((R_2+\e z)/r_1)}{\sn(r_2/r_1)}\frac{\e}{n}\int_{-1}^{+1} (R_2+\e s)\varphi_\k' (s)\sn(r_2/(R_2+\e s))\bb(s)\ds\\
-\frac{\e}{n}\int_{-1}^{ z} (R_2+\e s)\varphi_\k' (s)\sn((R_2+\e z)/(R_2+\e s))\bb(s)\ds,
\end{multline}
with 
\begin{equation}\label{def:Lambdacapital}
\Lambda_{m,\e}^{R_i}(z):=\l_m (R_i+\e z)^2+\Upsilon_{\e,\k}^{R_i}(z), \qquad \text{for } i=1,2.
\end{equation}
Note that neither $\Lambda_{m,\e}^{R_1}(z)$ nor $\Lambda_{m,\e}^{R_2}(z)$ have zeros on the unit interval. More specifically, for the computed eigenvalue $\lambda=\l_0+\lambda_1\ep+\lambda_2^\ep \ep^2$, we have (see \eqref{decominalphas}) that
\begin{align*}
\Lambda_{m,\e}^{R_1}(z)&= \a_0^{R_1}[\l_0]+ \a_1^{R_1}[\l_0,\l_1](z)\e+\a_{2,\e}^{R_1}[\l_0,\l_1,\l_2^\e](z)\e^2,\\
\Lambda_{m,\e}^{R_2}(z)&= \phantom{\a_0^{R_2}[\l_0]+}\,\, \a_1^{R_2}[\l_0,\l_1](z)\e+\a_{2,\e}^{R_2}[\l_0,\l_1,\l_2^\e](z)\e^2.
\end{align*}
On the one hand, recalling \eqref{aux:a1definedok} we have seen that
\[
\alpha_0^{R_1}[\l_0]=B\frac{R_1^2-R_2^2}{R_2^2}\neq 0.
\]
On the other hand, using \eqref{auxIIlamba1} and Lemma \ref{lema:lambda1} we have
$$\a_1^{R_2}[\l_0,\l_1](z)\leq \a_1^{R_2}[\l_0,\l_1](1)=-\beta(M), \qquad \text{with }\quad \beta(M)>0.$$

Consequently, we have 
\[
\Lambda_{m,\e}^{R_1}(z)=\a_0^{R_1}[\l_0]+ O(\e), \qquad \Lambda_{m,\e}^{R_2}(z)=-\beta(M)\e +O(\e^2).
\]

Taking $\e$ small enough these two quantities are not zero and we are allow to divide both sides of \eqref{auxregularityaa}, \eqref{auxregularitybb} by these non-zero factors. In addition, most of the terms that appear on the right side of \eqref{auxregularityaa}, \eqref{auxregularitybb} are directly $C^{\infty}([-1,1])$-functions since they involved only $\sn(x)=\sinh(n\log(x))$ with $x>1$.

Finally, for the rest of the integral terms that appear as
$$\ff(z)=\int_{-1}^{z} (R_i+\e s)\varphi_\k' ((-1)^i s)\sm\left(\frac{R_i+\e z}{R_i+\e s}\right)\gg(s)\ds,$$
we are going to use the fact that $\ff\in H^2([-1,1])$ for $\gg\in L^2([-1,1])$. The result follows directly by taking two derivatives on $\ff$ to get that
\[
\ff'(z)=\frac{\e m}{R_i+\e z}\int_{-1}^{z} (R_i+\e s)\varphi_\k' ((-1)^i s)\cm\left(\frac{R_i+\e z}{R_i+\e s}\right)\gg(s)\ds,
\]
\begin{multline*}
\ff''(z)= \e m \varphi_\k' ((-1)^i z)\gg(z) + \left(\frac{\e m}{R_i+\e z}\right)^2\ff(z)\\
-\frac{\e^2 m}{(R_i+\e z)^2}\int_{-1}^{z} (R_i+\e s)\varphi_\k' ((-1)^i s)\cm\left(\frac{R_i+\e z}{R_i+\e s}\right)\gg(s)\ds.
\end{multline*}
Therefore, if $\gg\in H^k([-1,1])$ we have that  $\ff\in H^{k+2}([-1,1])$.

All this yields that if $(\aa,\bb)\in H^k([-1, 1])\times H^k([-1, 1])$ then $(\aa, \bb)\in H^{k+2}([-1, 1])\times H^{k+2}([-1, 1])$ for all $k \geq 0$. Then, we can conclude by a recursive argument that $(\aa, \bb) \in C^{\infty}([-1, 1])$.

\subsection{Uniqueness}
So far we have shown that there exist $\lambda$ and  $(\bar{h}_n^{R_1},\bar{h}_n^{R_2})$ given by
\begin{align*}
\bar{h}_n^{R_1}(z)=\left\{\begin{array}{cc} 0 & n\neq m,\\ \aa(z) & n=m, \end{array}\right. \quad \text{and} \quad  \bar{h}_n^{R_2}(z)=\left\{\begin{array}{cc} 0 & n\neq m,\\ \bb(z) & n=m, \end{array}\right.
\end{align*}
with $(\l,\aa,\bb)\in \R\times C^\infty([-1,1])\times C^\infty([-1,1])$  solving \eqref{LR1final}, \eqref{LR2final} for all $n\in\N$ and satisfying \eqref{asolucion}, \eqref{bsolucion}, \eqref{lambdasolucion} and \eqref{cotas}.

To finish the proof of Theorem \ref{mainthmkernel} we need to check that the kernel of $\cL[\lambda]$, fixed that $\lambda$, is one dimensional. Therefore we have to prove that the solution  of system \eqref{LR1final}, \eqref{LR2final}, given by \eqref{hR1sol}, \eqref{hR2sol}, \eqref{asolucion}, \eqref{bsolucion}, for $\lambda$ given by \eqref{lambdasolucion}, is unique modulo multiplication by a constant.
In order to prove our goal, we will distinguish between two cases:\medskip

\subsubsection{Case $n= m$}
Let $(\uu,\vv)\in L^2([-1,1])\times L^2([-1,1])$ be a solution of \eqref{LR1cero}, \eqref{LR2cero}, with $\lambda$ given by \eqref{lambdasolucion}.
Note that $(\uu, \vv)$ depends on $\e$ although we do not make explicit this dependence.
Then, let us see that for some constant C we have
\[
\begin{cases}
\uu(z)=C \aa(z),\\
\vv(z)=C \bb(z).
\end{cases}	
\]
The system \eqref{LR1cero}, \eqref{LR2cero} is linear in $(\uu, \vv)$ and we can assume without loss of generality that
\[
\|\uu\|_{L^2([-1,1])}^2+\|\vv\|_{L^2([-1,1])}^2=1.
\]
If it is not the case we only need to normalized the solution and enter that value into the final constant $C$. Now, using \eqref{LR1cero} and \eqref{LR1ansatz}, we have that
\[
0=\cL_{m,\e}^{R_1}[\l]\begin{pmatrix}
\mathsf{u} \\
\mathsf{v}
\end{pmatrix}(z)=\Lambda_{m,\e}^{R_1}(z)\uu(z) + O(\e),
\]
but since we have seen  that $\Lambda_{m,\e}^{R_1}(z)=O(1)$ we get that $\uu(z)=\e \uu_1(z)$ with $\|\uu_1\|_{L^2([-1,1])}=O(1)$
in terms of $\e$. This information yields from \eqref{LR2cero} and \eqref{LR2ansatz} that
\begin{multline*}
0=\cL_{m,\e}^{R_2}[\l]\begin{pmatrix}
\mathsf{u} \\
\mathsf{v}
\end{pmatrix}(z)=\Lambda_{m,\e}^{R_2}(z)\vv(z) + O(\e^2)\\
-(R_2+\e z)\frac{\sm((R_2+\e z)/r_1)}{\sm(r_2/r_1)}\frac{\e}{m}\int_{-1}^{+1} (R_2+\e s)\varphi_\k' (s)\sm(r_2/(R_2+\e s))\vv(s)\ds\\
+(R_2+\e z)\frac{\e}{m}\int_{-1}^{ z} (R_2+\e s)\varphi_\k' (s)\sm((R_2+\e z)/(R_2+\e s))\vv(s)\ds.
\end{multline*}
Since $\Lambda_{m,\e}^{R_2}(z)=  \a_1^{R_2}[\l_0,\l_1](z)\e+\a_{2,\e}^{R_2}[\l_0,\l_1,\l_2^\e](z)\e^2$, we divide the above expression by $\e$ to get
\begin{multline}\label{eq:preauxvF}
\a_1^{R_2}[\l_0,\l_1](z)\vv(z)+(R_2+\e z)\frac{1}{m}\int_{-1}^{ z} (R_2+\e s)\varphi_\k' (s)\sm((R_2+\e z)/(R_2+\e s))\vv(s)\ds\\
-(R_2+\e z)\frac{\sm((R_2+\e z)/r_1)}{\sm(r_2/r_1)}\frac{1}{m}\int_{-1}^{+1} (R_2+\e s)\varphi_\k' (s)\sm(r_2/(R_2+\e s))\vv(s)\ds=O(\e).
\end{multline}
Now, we introduce the auxiliary functional
\begin{multline}\label{eq:auxR}
R_\e[\vv](z):=-(R_2+\e z)\frac{\sm((R_2+\e z)/r_1)}{\sm(r_2/r_1)}\frac{1}{m}\int_{-1}^{+1} (R_2+\e s)\varphi_\k' (s)\sm(r_2/(R_2+\e s))\vv(s)\ds\\
+(R_2+\e z)\frac{1}{m}\int_{-1}^{ z} (R_2+\e s)\varphi_\k' (s)\sm((R_2+\e z)/(R_2+\e s))\vv(s)\ds,
\end{multline}
which can be written as
\[
R_\e[\vv](z)=R_0[\vv](z)+\int_0^\e (\p_\d R_\d)[\vv](z) \mbox{d}\d,
\]
with
\[
R_0[\vv](z):=-R_2^2\frac{\sm(R_2/r_1)\sm(r_2/R_2)}{\sm(r_2/r_1)}\frac{1}{m}\int_{-1}^{+1}  \varphi_\k' (s)\vv(s)\ds.
\]
Therefore, using \eqref{def:ppi} we have that \eqref{eq:preauxvF} can be written as
\begin{equation}\label{eq:auxvF}
\a_1^{R_2}[\l_0,\l_1](z)\vv(z)-\pp_2(m)\int_{-1}^{+1}  \varphi_\k' (s)\vv(s)\ds=O(\e)-\int_0^\e (\p_\d R_\d)[\vv](z) \mbox{d}\d=:F(z),
\end{equation}
where $\|F\|_{L^2([-1,1])}=O(\e)$, see \eqref{DFRW} for details. Just dividing \eqref{eq:auxvF} by $\a_1^{R_2}[\l_0,\l_1](z)$ we have
\[
\vv(z)=\frac{\pp_2(m)}{\a_1^{R_2}[\l_0,\l_1](z)}\int_{-1}^{+1}  \varphi_\k' (s)\vv(s)\ds + \frac{F(z)}{\a_1^{R_2}[\l_0,\l_1](z)}.
\]
Thus, recalling the precise form of $\bb_0(z)$, given by \eqref{b0formula} we obtain that
\begin{equation}\label{eq:auxvC}
\vv(z)= C \bb_0(z) + \e \vv_1(z), 
\end{equation}
where $\|\vv_1\|_{L^2([-1,1])}=O(1)$ in terms of $\e$ and $C$ is a constant.

Taking into account all this information and looking again \eqref{LR1cero} and \eqref{LR1ansatz}, we obtain
\begin{multline*}
0=\cL_{m,\e}^{R_1}[\l]\begin{pmatrix}
\e\mathsf{u}_1 \\
C \bb_0+\e\mathsf{v}_1
\end{pmatrix}(z)=\e\Lambda_{m,\e}^{R_1}(z)\uu_1(z)\\
+(R_1+\e z)\frac{\sm((R_1+\e z)/r_1)}{\sm(r_2/r_1)}\frac{\e^2}{m}\int_{-1}^{+1} (R_1+\e s)\varphi_\k' (-s)\sm(r_2/(R_1+\e s))\uu_1(s)\ds\\
-(R_1+\e z)\frac{\sm((R_1+\e z)/r_1)}{\sm(r_2/r_1)}\frac{\e}{m}\int_{-1}^{+1} (R_2+\e s)\varphi_\k' (s)\sm(r_2/(R_2+\e s))(C \bb_0(s)+\e\vv_1(s)) \ds\\
-(R_1+\e z)\frac{\e^2}{m}\int_{-1}^{z} (R_1+\e s)\varphi_\k' (-s)\sm((R_1+\e z)/(R_1+\e s))\uu_1(s)\ds=0.
\end{multline*}
Then, using \eqref{def:ppi} we have that
\[
\a_0^{R_1}[\l_0] \uu_1(z)-C \pp_1(m)\int_{-1}^{+1} \varphi_\k' (s)\bb_0(s)\ds=G(z),
\]
where $\|G\|_{L^2([-1,1])}=O(\e)$. Recalling the precise form of $\aa_1$, see \eqref{def:a1}, forces to
\begin{equation}\label{eq:auxuC}
\uu_1(z)=C \aa_1 + \e \uu_2(z),
\end{equation}
with $\|\uu_2\|_{L^2([-1,1])}=O(1)$ in terms of $\e$ and the same constant $C$ that in \eqref{eq:auxvC}.

Finally, we get a coupled system for $(\uu_2, \vv_1)$ which is exactly the same as \eqref{eqb1eaux}, \eqref{eqa2eaux} for $(\aa_2^\e, \bb_1^\e)$ up to the multiplicative constant $C$. More specifically, we have
\begin{multline*}
0=\cL_{m,\e}^{R_1}[\l]\begin{pmatrix}
\e(C\aa_1 +\e \uu_2) \\
C \bb_0+\e\mathsf{v}_1
\end{pmatrix}(z)=\e\Lambda_{m,\e}^{R_1}(z)(C\aa_1 +\e \uu_2(z))\\
+(R_1+\e z)\frac{\sm((R_1+\e z)/r_1)}{\sm(r_2/r_1)}\frac{\e^2}{m}\int_{-1}^{+1} (R_1+\e s)\varphi_\k' (-s)\sm(r_2/(R_1+\e s))(C\aa_1 +\e \uu_2(s))\ds\\
-(R_1+\e z)\frac{\sm((R_1+\e z)/r_1)}{\sm(r_2/r_1)}\frac{\e}{m}\int_{-1}^{+1} (R_2+\e s)\varphi_\k' (s)\sm(r_2/(R_2+\e s))(C \bb_0(s)+\e\vv_1(s)) \ds\\
-(R_1+\e z)\frac{\e^2}{m}\int_{-1}^{z} (R_1+\e s)\varphi_\k' (-s)\sm((R_1+\e z)/(R_1+\e s))(C\aa_1 +\e \uu_2(s))\ds=0,
\end{multline*}
and
\begin{multline*}
0=\cL_{m,\e}^{R_2}[\l]\begin{pmatrix}
\e(C\aa_1 +\e \uu_2)\\
C\bb_0 + \e \vv_1
\end{pmatrix}(z)=\Lambda_{m,\e}^{R_2}(z)(C\bb_0(z) + \e \vv_1(z))\\
+(R_2+\e z)\frac{\sm((R_2+\e z)/r_1)}{\sm(r_2/r_1)}\frac{\e^2}{m}\int_{-1}^{+1} (R_1+\e s)\varphi_\k' (-s)\sm(r_2/(R_1+\e s))(C\aa_1 +\e \uu_2(s))\ds \\
-(R_2+\e z)\frac{\sm((R_2+\e z)/r_1)}{\sm(r_2/r_1)}\frac{\e}{m}\int_{-1}^{+1} (R_2+\e s)\varphi_\k' (s)\sm(r_2/(R_2+\e s))(C\bb_0(s) + \e \vv_1(s))\ds\\
-(R_2+\e z)\frac{\e^2}{m}\int_{-1}^{+1} (R_1+\e s)\varphi_\k' (-s)\sm((R_2+\e z)/(R_1+\e s))(C\aa_1 +\e \uu_2(s))\ds\\
+(R_2+\e z)\frac{\e}{m}\int_{-1}^{ z} (R_2+\e s)\varphi_\k' (s)\sm((R_2+\e z)/(R_2+\e s))(C\bb_0(s) + \e \vv_1(s))\ds=0.
\end{multline*}
Proceeding as before, we finally arrive to
\[
\a_0^{R_1}[\l_0]\uu_2(z)-\pp_1(m)\int_{-1}^{+1} \varphi_\k' (s)\vv_1(s) \ds= C A_0(z) + \e A_1[\uu_2,\vv_1;C\aa_1,C\bb_0],
\]
\[
\vv_1(z)=\left(C B_0(z)+\e B_1[\uu_1,\vv_1; C\bb_0,C\aa_1](z)-\l_2^\e R_2^2 C\bb_0(z)\right)C\bb_0(z).
\]
This system is linear in $(\uu_2,\vv1)$ (now $\l$ is fixed). One can check that $C(\aa_\e^2,\bb_\e^1)$ is a solution and by a fixed point argument must be unique. Then $(\uu_2,\vv_1) = C(\aa_\e^2,\bb_\e^1)$ (with the same constant $C$ as in \eqref{eq:auxuC} and \eqref{eq:auxvC}) and consequently we have proved our goal.

\subsubsection{Case $n\neq m$}   
Since $\Lambda_{m,\e}^{R_1}(z)=O(1)$ and $\Lambda_{m,\e}^{R_2}(z)=O(\e)$ we can just write
\begin{multline}\label{eq:mainunnotm}
\Lambda_{m,\e}^{R_1}(z)\uu(z)+(R_1+\e z)\frac{\sn((R_1+\e z)/r_1)}{\sn(r_2/r_1)}\frac{\e}{n}\int_{-1}^{+1} (R_1+\e s)\varphi_\k' (-s)\sn(r_2/(R_1+\e s))\uu(s)\ds\\
-(R_1+\e z)\frac{\sn((R_1+\e z)/r_1)}{\sn(r_2/r_1)}\frac{\e}{n}\int_{-1}^{+1} (R_2+\e s)\varphi_\k' (s)\sn(r_2/(R_2+\e s))\vv(s)\ds\\
-(R_1+\e z)\frac{\e}{n}\int_{-1}^{z} (R_1+\e s)\varphi_\k' (-s)\sn((R_1+\e z)/(R_1+\e s))\uu(s)\ds=0,
\end{multline}
and (dividing the full expression by $\e$)
\begin{multline}\label{eq:mainvnnotm}
\frac{\Lambda_{m,\e}^{R_2}(z)}{\e}\vv(z)+(R_2+\e z)\frac{\sn((R_2+\e z)/r_1)}{\sn(r_2/r_1)}\frac{1}{n}\int_{-1}^{+1} (R_1+\e s)\varphi_\k' (-s)\sn(r_2/(R_1+\e s))\uu(s)\ds \\
-(R_2+\e z)\frac{\sn((R_2+\e z)/r_1)}{\sn(r_2/r_1)}\frac{1}{n}\int_{-1}^{+1} (R_2+\e s)\varphi_\k' (s)\sn(r_2/(R_2+\e s))\vv(s)\ds\\
-(R_2+\e z)\frac{1}{n}\int_{-1}^{+1} (R_1+\e s)\varphi_\k' (-s)\sn((R_2+\e z)/(R_1+\e s))\uu(s)\ds\\
+(R_2+\e z)\frac{1}{n}\int_{-1}^{ z} (R_2+\e s)\varphi_\k' (s)\sn((R_2+\e z)/(R_2+\e s))\vv(s)\ds=0.
\end{multline}
Hence, we can divide \eqref{eq:mainunnotm} by $\Lambda_{m,\e}^{R_1}(z)$ to get
\begin{equation}\label{solunnotm}
\uu(z)=\e U[\uu,\vv](z),
\end{equation}
where
\begin{multline*}
U[\uu,\vv](z):=
\frac{R_1+\e z}{n \Lambda_{m,\e}^{R_1}(z)}\frac{\sn((R_1+\e z)/r_1)}{\sn(r_2/r_1)}\int_{-1}^{+1} (R_2+\e s)\varphi_\k' (s)\sn(r_2/(R_2+\e s))\vv(s)\ds\\
-\frac{R_1+\e z}{n \Lambda_{m,\e}^{R_1}(z)}\frac{\sn((R_1+\e z)/r_1)}{\sn(r_2/r_1)}\int_{-1}^{+1} (R_1+\e s)\varphi_\k' (-s)\sn(r_2/(R_1+\e s))\uu(s)\ds\\
+\frac{R_1+\e z}{n \Lambda_{m,\e}^{R_1}(z)}\int_{-1}^{z} (R_1+\e s)\varphi_\k' (-s)\sn((R_1+\e z)/(R_1+\e s))\uu(s)\ds.
\end{multline*}
The next step will be to introduce \eqref{solunnotm} into the equation \eqref{eq:mainvnnotm}. Hence, we have
\begin{multline}\label{eq:preauxW}
\frac{\Lambda_{m,\e}^{R_2}(z)}{\e}\vv(z)
-(R_2+\e z)\frac{\sn((R_2+\e z)/r_1)}{\sn(r_2/r_1)}\frac{1}{n}\int_{-1}^{+1} (R_2+\e s)\varphi_\k' (s)\sn(r_2/(R_2+\e s))\vv(s)\ds\\
+(R_2+\e z)\frac{1}{n}\int_{-1}^{ z} (R_2+\e s)\varphi_\k' (s)\sn((R_2+\e z)/(R_2+\e s))\vv(s)\ds\\
=(R_2+\e z)\frac{\e}{n}\int_{-1}^{+1} (R_1+\e s)\varphi_\k' (-s)\sn((R_2+\e z)/(R_1+\e s))U[\uu,\vv](s)\ds\\
-(R_2+\e z)\frac{\sn((R_2+\e z)/r_1)}{\sn(r_2/r_1)}\frac{\e}{n}\int_{-1}^{+1} (R_1+\e s)\varphi_\k' (-s)\sn(r_2/(R_1+\e s))U[\uu,\vv](s)\ds.
\end{multline}
Now, we introduce the auxiliary functional
\begin{multline}\label{eq:auxW}
W_\e[\vv](z):=-(R_2+\e z)\frac{\sn((R_2+\e z)/r_1)}{\sn(r_2/r_1)}\frac{1}{n}\int_{-1}^{+1} (R_2+\e s)\varphi_\k' (s)\sn(r_2/(R_2+\e s))\vv(s)\ds\\
+(R_2+\e z)\frac{1}{n}\int_{-1}^{ z} (R_2+\e s)\varphi_\k' (s)\sn((R_2+\e z)/(R_2+\e s))\vv(s)\ds,
\end{multline}
which can be written as
\[
W_\e[\vv](z)=W_0[\vv](z)+\int_0^\e (\p_\delta W_\delta)[\vv](z)\mbox{d}\delta,
\]
with (using \eqref{def:ppi})
\[
W_0[\vv](z)=-\pp_2(n)\int_{-1}^{+1}  \varphi_\k' (s)\vv(s)\ds.
\]
Then, using \eqref{eq:auxW} we have that \eqref{eq:preauxW} can be written as
\begin{multline*}
\a_1^{R_2}[\l_0,\l_1](z)\vv(z)    -\pp_2(n)\int_{-1}^{+1}  \varphi_\k' (s)\vv(s)\ds=-\alpha_{2,\e}^{R_2}[\l_0,\l_1,\l_2^{\e}](z)\vv(z)\e-\int_0^\e (\p_\delta W_\delta)[\vv](z)\mbox{d}\delta\\
+(R_2+\e z)\frac{\e}{n}\int_{-1}^{+1} (R_1+\e s)\varphi_\k' (-s)\sn((R_2+\e z)/(R_1+\e s))U[\uu,\vv](s)\ds\\
-(R_2+\e z)\frac{\sn((R_2+\e z)/r_1)}{\sn(r_2/r_1)}\frac{\e}{n}\int_{-1}^{+1} (R_1+\e s)\varphi_\k' (-s)\sn(r_2/(R_1+\e s))U[\uu,\vv](s)\ds.
\end{multline*}
Introducing a new auxiliary functional $\e V[\uu,\vv](z)$ for the right hand side of the above expression, we arrive to
\begin{equation}\label{solvingvnneqm}
\a_1^{R_2}[\l_0,\l_1](z)\vv(z)    -\pp_2(n)\int_{-1}^{+1}  \varphi_\k' (s)\vv(s)\ds=\e V[\uu,\vv](z),
\end{equation}
with
\begin{multline*}
V[\uu,\vv](z):=
-\alpha_{2,\e}^{R_2}[\l_0,\l_1,\l_2^{\e}](z)\vv(z)-\frac{1}{\e}\int_0^\e (\p_\delta W_\delta)[\vv](z)\mbox{d}\delta\\
+(R_2+\e z)\frac{1}{n}\int_{-1}^{+1} (R_1+\e s)\varphi_\k' (-s)\sn((R_2+\e z)/(R_1+\e s))U[\uu,\vv](s)\ds\\
-(R_2+\e z)\frac{\sn((R_2+\e z)/r_1)}{\sn(r_2/r_1)}\frac{1}{n}\int_{-1}^{+1} (R_1+\e s)\varphi_\k' (-s)\sn(r_2/(R_1+\e s))U[\uu,\vv](s)\ds.
\end{multline*}

In the following lemma we are going to learn how to invert the left hand side of \eqref{solvingvnneqm}.
\begin{lemma}\label{lemma:invertnneqm}
Let $\mathsf{G}\in L^2([-1,1])$. Then, the unique solution for
\begin{equation}\label{auxlemmahatP}
\a_1^{R_2}[\l_0,\l_1](z)\gg(z)-\pp_2(n)\int_{-1}^{1} \varphi_\k'( s)\gg(s)\ds=\mathsf{G}(z),
\end{equation}
is given by
\begin{equation}
\gg(z)=\frac{1}{\a_1^{R_2}[\l_0,\l_1](z)}\left(\mathsf{G}(z)+\frac{\pp_2(m)\pp_2(n)}{\pp_2(m)-\pp_2(n)}\int_{-1}^{1} \varphi_\k'( s)\frac{\mathsf{G}(s)}{\a_1^{R_2}[\l_0,\l_1](s)}\ds\right)
\end{equation}
\end{lemma}

Before their proof, we need to note that the following equivalence holds:
\begin{equation}\label{1to1pp2}
\pp_2(n)=\pp_2(m) \qquad \Longleftrightarrow \qquad  n= m.
\end{equation}
\begin{proof}[Proof of \eqref{1to1pp2}]
One of the implication follows trivially. In order to get the other one, we just need to prove that $\pp_2(n)$ is strictly decreasing for $n\in\N$.

Since $r_1<R_1<R_2<r_2$, taking $\mathsf{A}:=R_2/r_1$ and $\mathsf{B}:=r_2/R_2$, we can write
\[
\pp_2(n)=\frac{\sn(\mathsf{A})\sn(\mathsf{B})}{n\sn(\mathsf{A}\cdot\mathsf{B})}, \qquad \mathsf{A}, \mathsf{B}>1.
\]
By calculating its derivative, we obtain
\begin{multline*}
\sn^2(\mathsf{A}\cdot \mathsf{B})\frac{n}{R_2^2}\pp_2'(n)=\cn(\mathsf{A})\sn(\mathsf{B})\sn(\mathsf{A}\cdot\mathsf{B} )\log(\mathsf{A}) + \cn(\mathsf{B})\sn(\mathsf{A})\sn(\mathsf{A}\cdot \mathsf{B})\log(\mathsf{B})\\
-\sn(\mathsf{A})\sn(\mathsf{B}) \cn(\mathsf{A}\cdot \mathsf{B})\log(\mathsf{A}\cdot \mathsf{B})- \frac{1}{n}\sn(\mathsf{A})\sn(\mathsf{B}).
\end{multline*}
Maximizing the function of the right-hand side of the previous expression over the region $\mathsf{A,B}\geq 1$ we obtain that its absolute maximum is zero and is reached on the boundary, i.e. $\mathsf{A}=1$ or $\mathsf{B}=1$.
Since the above is independent of the parameter $n\in\N$, we get $\pp_2'(n)<0$ and conclude our goal. 
\end{proof}

Now, we have all the ingredients to prove the previous lemma.
\begin{proof}[Proof of Lemma \ref{lemma:invertnneqm}]
If $\gg(z)$ is a solution of \eqref{auxlemmahatP} then
\begin{equation}\label{def:gintermsofC}
\gg(z)=\frac{C}{\a_1^{R_2}[\l_0,\l_1](z)}+\frac{\mathsf{G}(z)}{\a_1^{R_2}[\l_0,\l_1](z)},
\end{equation}
for some constant $C$. Plugging this expression of $\gg$ into \eqref{auxlemmahatP} yields
\begin{equation*}
C\left(1 - \pp_2(n)\int_{-1}^{1}\frac{\varphi_\k'( s)}{\a_1^{R_2}[\l_0,\l_1](s)}\ds\right)-\pp_2(n)\int_{-1}^{1} \varphi_\k'( s)\frac{\mathsf{G}(s)}{\a_1^{R_2}[\l_0,\l_1](s)}\ds=0.
\end{equation*}
Then, using \eqref{lambda1eq} we get
\[
C\left(1-\frac{\pp_2(n)}{\pp_2(m)}\right)-\pp_2(n)\int_{-1}^{1} \varphi_\k'( s)\frac{\mathsf{G}(s)}{\a_1^{R_2}[\l_0,\l_1](s)}\ds=0,
\]
or equivalently
\begin{equation}\label{def:C}
C=\frac{\pp_2(m)\pp_2(n)}{\pp_2(m)-\pp_2(n)}\int_{-1}^{1} \varphi_\k'( s)\frac{\mathsf{G}(s)}{\a_1^{R_2}[\l_0,\l_1](s)}\ds,
\end{equation}
where in the last step we have used \eqref{1to1pp2} to ensure that $\pp_2(m)-\pp_2(n)\neq 0$.

Conversely, it is easy to check that $\gg$ in \eqref{def:gintermsofC}, with $C$ given by \eqref{def:C}, solves the equation \eqref{auxlemmahatP}.
\end{proof}
To finish, let us introduce the linear functional
\[
I[\mathsf{G}](z):=\frac{1}{\a_1^{R_2}[\l_0,\l_1](z)}\left(\mathsf{G}(z)+\frac{\pp_2(m)\pp_2(n)}{\pp_2(m)-\pp_2(n)}\int_{-1}^{1} \varphi_\k'( s)\frac{\mathsf{G}(s)}{\a_1^{R_2}[\l_0,\l_1](s)}\ds\right).
\]
By using Lemma \ref{lemma:invertnneqm}, we have that solution of \eqref{solvingvnneqm} is given by
\begin{equation}\label{solvnotm}
\vv(z)=\e I[V[\uu,\vv]](z).
\end{equation}
\begin{remark}
To apply Lemma \ref{lemma:invertnneqm} and arrive at the above expression \eqref{solvnotm} we must first check that the right-hand side of  \eqref{solvingvnneqm} belongs to $L^2([-1,1])$, which follows immediately from \eqref{DFRW}.
\end{remark}

The coupled system given by \eqref{solunnotm} and \eqref{solvnotm} is a linear contraction on $L^2([-1, 1])$ for $\e$ small
enough. Hence, there exists a unique $(\uu,\vv)\in L^2([-1,1])\times L^2([-1,1])$ solving \eqref{solunnotm}, \eqref{solvnotm}: 
\begin{align*}
\uu(z)&=\e U[\uu,\vv](z),\\
\vv(z)&=\e I[V[\uu,\vv]](z).
\end{align*}

Therefore, the trivial solution $(\uu, \vv)=(0,0)$ is the unique solution of \eqref{eq:mainunnotm} and \eqref{eq:mainvnnotm}. 
Consequently, we have achieved all the conclusions of Theorem \ref{mainthmkernel}.

\section{Co-dimension of the image of the linear operator}\label{s:codimension}
Let $\l_{\e,\k,m}$ given by Theorem \ref{mainthmkernel}. In order to prove that the co-dimension in $Y(D_\e)$ of the linear operator $\cL[\l_{\e,\k,m}]\equiv D_\ff F[\l_{\e,\k,m},0]$ over $X(D_\e)$ is one-dimensional, we will borrow the argument used for example in \cite{GHS}, see also \cite{G2,GHM,GHM2,GHM3}. Indeed, we will prove that $\cL[\l_{\e,\k,m}]$ is a Fredholm operator of zero index. Since
\[
\text{dim}(\mathcal{N}(\cL[\l_{\e,\k,m}]))=\text{dim}(Y(D_\e)/\mathcal{R}(\cL[\l_{\e,\k,m}])),
\]
for this kind of operators, we will find that
\begin{equation}\label{codim=1}
\text{dim}(Y(D_\e)/\mathcal{R}(\cL[\l_{\e,\k,m}]))=1,
\end{equation}
due to Theorem \ref{mainthmkernel}, where we have proved that kernel $\mathcal{N}(\cL[\l_{\e,\k,m}])$ of the linearized operator is one-dimensional. Hence, we have seen the third condition of the Crandall-Rabinowitz Theorem.

\subsection{Isomorphism plus compact operator}
The main task of this section is to write the linear operator $\cL[\l_{\e,\k,m}]$ as an isomorphism plus a compact operator.
In order to prove this, it will be convenient to recall that  linear operator $\cL[\l_{\e,\k,m}]$ can be decomposed as
\[
\cL[\l_{\e,\k,m}]\hh(r,\theta)=\sum_{n\geq 1}\cos(n\theta)\cL_n[\l_{\e,\k,m}]h_n(r), \qquad (r,\theta)\in D_\e,
\]
or equivalently
\begin{align*}
\cL[\l_{\e,\k,m}]\hh(r,\theta)=
\begin{cases}
\sum_{n\geq 1} \cos(n\theta)  \cL_n^1[\l_{\e,\k,m}]
\begin{pmatrix}
h_n^{R_1}\\
h_n^{R_2}
\end{pmatrix}
(r), \qquad r\in I_\e^{R_1},\vspace{0.25 cm}\\
\sum_{n\geq 1} \cos(n\theta)  \cL_n^2[\l_{\e,\k,m}]
\begin{pmatrix}
h_n^{R_1}\\
h_n^{R_2}
\end{pmatrix}
(r), \qquad r\in I_\e^{R_2},
\end{cases}   
\end{align*}
with 
\[
h_n(r)=
\begin{cases}
h_n^{R_1}(r), \qquad r\in I_\e^{R_1},\\
h_n^{R_2}(r), \qquad r\in I_\e^{R_2}, \end{cases}
\]
and (see \eqref{def:LnggR1})
\begin{multline*}
\cL_n^{1}[\l_{\e,\k,m}]\left( \gg_1, \gg_2 \right)(r)=\left(\phi'(r)+\l r\right)\gg_1(r)+\frac{1}{n}\int_{R_1-\e}^{r} s\varpi_{\e,\k}'(s)\sn(r/s)\gg_1(s)\ds\\
-\frac{1}{n}\frac{\sn(r/r_1)}{\sn(r_2/r_1)}\left(\int_{I_\e^{R_1}} s\varpi_{\e,\k}'(s)\sn(r_2/s)\gg_1(s)\ds + \int_{I_\e^{R_2}} s\varpi_{\e,\k}'(s)\sn(r_2/s)\gg_2(s)\ds\right),
\end{multline*}
and (see \eqref{def:LnggR2})
\begin{multline*}
\cL_n^{2}[\l_{\e,\k,m}]\left( \gg_1, \gg_2 \right)(r)=\left(\phi'(r)+\l r\right)\gg_2(r)\\
+\frac{1}{n}\left(\int_{I_\e^{R_1}} s\varpi_{\e,\k}'(s)\sn(r/s)\gg_1(s)\ds +  \int_{R_2-\e}^{r} s\varpi_{\e,\k}'(s)\sn(r/s)\gg_2(s)\ds.\right)\\
-\frac{1}{n}\frac{\sn(r/r_1)}{\sn(r_2/r_1)}\left(\int_{I_\e^{R_1}} s\varpi_{\e,\k}'(s)\sn(r_2/s)\gg_1(s)\ds + \int_{I_\e^{R_2}} s\varpi_{\e,\k}'(s)\sn(r_2/s)\gg_2(s)\ds\right).
\end{multline*}

\begin{remark}
Just above and throughout all this section we will write (with a little abuse of notation) $\cL_n^{1}[\l_{\e,\k,m}]\left( \gg_1, \gg_2 \right)$ and $\cL_n^{2}[\l_{\e,\k,m}]\left( \gg_1, \gg_2 \right)$ instead of
\[
\cL_n^{1}[\l_{\e,\k,m}]\begin{pmatrix}
\gg_1\\
\gg_2
\end{pmatrix}, \qquad \cL_n^{2}[\l_{\e,\k,m}]\begin{pmatrix}
\gg_1\\
\gg_2
\end{pmatrix},
\]
to alleviate the notation.
\end{remark}

Now, we have remembered all the necessary ingredients to obtain the following decomposition:
\begin{equation}\label{eq:isopluscompact}
\cL_n^{i}[\l_{\e,\k,m}]\left( \gg_1, \gg_2 \right)(r)=I^{R_i}[\l_{\e,\k,m}]\gg_i(r) + K_n^{R_i}(\gg_1,\gg_2)(r), \qquad \text{for } i=1,2,
\end{equation}
where we defined the auxiliary operator $I^{R_i}[\l_{\e,\k,m}]:C^{2,\alpha}(I_\e^{R_i})\longrightarrow C^{2,\alpha}(I_\e^{R_i})$ by the expression
\[
I^{R_i}[\l_{\e,\k,m}]\gg(r):=\left(\phi'(r)+\l_{\e,\k,m} r\right)\gg(r),
\]
and the operators $K_n^{R_i}:C^{2,\alpha}(I_\e^{R_1})\times C^{2,\alpha}(I_\e^{R_2})\longrightarrow C^{2,\alpha}(I_\e^{R_i})$ via the relations
\begin{multline*}
K_n^{R_1}\left( \gg_1, \gg_2 \right)(r):=\frac{1}{n}\int_{R_1-\e}^{r} s\varpi_{\e,\k}'(s)\sn(r/s)\gg_1(s)\ds\\
-\frac{1}{n}\frac{\sn(r/r_1)}{\sn(r_2/r_1)}\left(\int_{I_\e^{R_1}} s\varpi_{\e,\k}'(s)\sn(r_2/s)\gg_1(s)\ds+\int_{I_\e^{R_2}} s\varpi_{\e,\k}'(s)\sn(r_2/s)\gg_2(s)\ds\right),
\end{multline*}
and
\begin{multline*}
K_n^{R_2}\left( \gg_1, \gg_2 \right)(r):=\frac{1}{n}\left(\int_{I_\e^{R_1}} s\varpi_{\e,\k}'(s)\sn(r/s)\gg_1(s)\ds + \int_{R_2-\e}^{r} s\varpi_{\e,\k}'(s)\sn(r/s)\gg_2(s)\ds\right)\\
-\frac{1}{n}\frac{\sn(r/r_1)}{\sn(r_2/r_1)}\left(\int_{I_\e^{R_1}} s\varpi_{\e,\k}'(s)\sn(r_2/s)\gg_1(s)\ds+\int_{I_\e^{R_2}} s\varpi_{\e,\k}'(s)\sn(r_2/s)\gg_2(s)\ds\right).
\end{multline*}

On one hand, we shall check that $I^{R_i}[\l_{\e,\k,m}]:C^{2,\alpha}(I_\e^{R_i})\longrightarrow C^{2,\alpha}(I_\e^{R_i})$ defines an isomorphism. The continuity of this operator follows from the regularity of the function
\begin{equation}\label{aux:iso}
\phi'(r)+\l_{\e,\k,m} r \in C^{\infty}(I_\e^{R_i}),
\end{equation}
combined with the fact that $C^{2,\alpha}(I_\e^{R_i})$ is an algebra.
\begin{proof}[Proof of \eqref{aux:iso}]
Recall (see \eqref{def:UpsRi} and \eqref{def:Lambdacapital}) that
\[
\phi'(r)+\l_{\e,\k,m} r=r^{-1}\Lambda_{m,\e}^{R_i}\left((r-R_i)\e^{-1}\right), \qquad \forall r\in I_\e^{R_i}.
\]
Considering the change of variables $r\in I_\e^{R_i}\mapsto z:=(r-R_i)\e^{-1}\in(-1,1)$ and applying the fact that we have proved earlier that $\Lambda_{m,\e}^{R_i}\in C^{\infty}(-1,1)$ does not vanish since
\begin{align*}
\Lambda_{m,\e}^{R_1}(z)= \a_0^{R_1}[\l_0]+ O(\e),\qquad
\Lambda_{m,\e}^{R_2}(z)= -\beta(M)\e +O(\e^2),
\end{align*}
we can conclude our result.
\end{proof}
Moreover, since $\phi'(r)+\l_{\e,\k,m} r$ is not vanishing over $I_\e^{R_i}$, one has that $(\phi'(r)+\l_{\e,\k,m} r)\text{Id}$ is injective. In order to check that such an operator is an isomorphism, it is enough to check that it is surjective, as a consequence of the Banach isomorphism theorem. Take $\mathsf{k}\in C^{2,\alpha}(I_\e^{R_i})$, we will find $\gg\in C^{2,\alpha}(I_\e^{R_i})$ such that $(\phi'(r)+\l_{\e,\k,m} r)\gg(r)=\mathsf{k}(r)$. Indeed, $\gg$ is given by $(\phi'(r)+\l_{\e,\k,m} r)^{-1}\mathsf{k}(r).$ Using the regularity of $(\phi'(r)+\l_{\e,\k,m} r)$ and the fact that it is not vanishing over $I_\e^{R_i}$, it is easy to check that its inverse $(\phi'(r)+\l_{\e,\k,m} r)^{-1}$ still belongs to $C^{2,\alpha}(I_\e^{R_i}).$ The fact that $C^{2,\alpha}(I_\e^{R_i})$ is an algebra implies the desired result.\medskip

On the other hand, we have to prove that operator $$K_n^{R_i}:C^{2,\alpha}(I_\e^{R_1})\times C^{2,\alpha}(I_\e^{R_2})\longrightarrow C^{2,\alpha}(I_\e^{R_i})$$ is compact. To do that, we will prove that
\begin{equation}\label{smoothing}
K_n^{R_i}:C^{2,\alpha}(I_\e^{R_1})\times C^{2,\alpha}(I_\e^{R_2})\longrightarrow C^{2,\beta}(I_\e^{R_i}), \qquad \beta\in(\alpha,1).    
\end{equation}
That is,  we prove that for any $\beta\in(\alpha,1)$ one has the smoothing effect
\[
\|K_n^{R_i}(\gg_1,\gg_2)\|_{C^{2,\beta}(I_\e^{R_i})}\leq C \left(\|\gg_1\|_{C^{2,\alpha}(I_\e^{R_1})},\|\gg_2\|_{C^{2,\alpha}(I_\e^{R_2})}\right), \qquad \forall (\gg_1,\gg_2)\in C^{2,\alpha}(I_\e^{R_1})\times C^{2,\alpha}(I_\e^{R_2}).
\]
In particular, taking $\beta>\alpha$ we have that operator
$$K_n^{R_i}:C^{2,\alpha}(I_\e^{R_1})\times C^{2,\alpha}(I_\e^{R_2})\longrightarrow C^{2,\beta}(I_\e^{R_i}) \hookrightarrow C^{2,\alpha}(I_\e^{R_i})$$
is compact since the embedding $C^{2,\beta}(I_\e^{R_i}) \hookrightarrow C^{2,\alpha}(I_\e^{R_i})$ is compact. 

\begin{proof}[Proof of \eqref{smoothing}]
First of all, it is important to note that just by definition, we have 
\[
K_n^{R_i}(\gg_1,\gg_2)=\cL_n^{i}[\lambda_{\e,\k,m}](\gg_1,\gg_2)-I^{R_i}[\lambda_{\e,\k,m}]\gg_i\in C^{2,\alpha}(I_\e^{R_i}).
\]
Using the relations (see \eqref{def:auxhyper})
\[
\sn(r/s)'=\cn(r/s) n/r, \qquad \cn(r/s)'=\sn(r/s) n/r,
\]
we directly get 
\begin{equation}\label{eq:Kn''}
r(r K_n^{R_i}(\gg_1,\gg_2)'(r))'=r^2\varpi_{\e,\k}'(r)\gg_i(r) +n^2 K_n^{R_i}(\gg_1,\gg_2)(r).
\end{equation}
Since $\gg_i, K_n^{R_i}(\gg_1,\gg_2)\in C^{2,\alpha}(I_\e^{R_i})$ and $\varpi_{\e,\k}'\in C^{\infty}(I_\e^{R_i})$, applying directly the fact that  H\"older space is an algebra we can conclude that the right hand side of \eqref{eq:Kn''} belongs to $C^{0,\beta}(I_\e^{R_i})$.

Therefore, we have proved that $K_n^{R_i}(\gg_1,\gg_2)\in C^{2,\beta}(I_\e^{R_i}).$ 
\end{proof}

Since Fredholm operator remains Fredholm with the same index modulo compact operators and any isomorphism is Fredholm of zero index, we have all the ingredients to deduce \eqref{codim=1}.

Therefore, to apply the bifurcation argument it only remains to check the transversality condition in the Crandall-Rabinowitz Theorem \ref{th:CR}. With this in mind, we must first find a practical characterization for the range of the linearized operator $\cL[\l_{\e,\k,m}]$.

\section{Adjoint linear operator and their kernel}\label{s:adjoint}
In order to get a practical and  useful characterization for the range $\cR(\cL[\l_{\e,\k,m}])$ of the linearized operator $\cL\equiv\cL[\l_{\e,\k,m}]$ we are going to use the adjoint of the linear operator, i.e., $\cL^\ast\equiv \cL^\ast[\l_{\e,\k,m}]$. Here, it is important to recall and emphasize that $\l_{\e,\k,m}$ will be fixed from the beginning of the section and given by Theorem \ref{mainthmkernel}.\\

Note that, to study the range of the linear operator, we just have to understand the equation
\begin{equation}\label{fullproblemcod}
\cL[\l_{\e,\k,m}]\hh(r,\theta)=\HH(r,\theta), \qquad (r,\theta)\in  D_\e,
\end{equation}
with $\hh\in X(D_\e)$ and $\HH\in Y(D_\e)$, where we have assumed that $Y(D_\e)=X(D_\e)$, see \eqref{def:spaceX}-\eqref{def:spaceY}.

As we did before, we will use the expansions 
\[
\hh(r,\theta)=\sum_{n\geq 1}h_n(r)\cos( n \theta), \qquad  \HH(r,\theta)=\sum_{n\geq 1}H_n(r) \cos( n \theta),
\]
which allow us to translate \eqref{fullproblemcod} into the following infinity-dimensional system 
\begin{equation}\label{nproblemcod}
\cL_n[\l_{\e,\k,m}]h_n(r)=H_n(r), \quad r\in I_\e^{R_1}\cup I_\e^{R_2}, 
\end{equation}
for all $n\in\N$ with $\cL_n[\l_{\e,\k,m}]$ defined via \eqref{def:Lnggpre}. 

In order to solve \eqref{nproblemcod}, we will proceed as we did in Section \ref{s:rescaling}. That is, we re-scale each equation and introduce the auxiliary functions
\begin{equation*}
\bar{H}_{n}^{R_i}(z):=(R_i+\e z) H_n(R_i+\e z),\qquad
\bar{h}_{n}^{R_i}(z):= h_n(R_i+\e z), \qquad \text{for } z\in(-1,1).
\end{equation*}
Moreover, we define the auxiliary functional
\[
P_{n,\e}^{R_i}[\bar{\gg}_1,\bar{\gg}_2](z):=\cL_{n,\e}^{R_i}[\l_{\e,\k,m}]\begin{pmatrix}
\bar{\gg}_1 \\
\bar{\gg}_2
\end{pmatrix}(z),
\]
with $\cL_{n,\e}^{R_1}[\l_{\e,\k,m}]$ and $\cL_{n,\e}^{R_2}[\l_{\e,\k,m}]$ respectively given by \eqref{LR1final} and \eqref{LR2final}, for any $\bar{\gg}_1, \bar{\gg}_2\in C^{2,\a}([-1,1]).$

\begin{remark}
In what follows, to alleviate the notation, we will omit the dependence of $\l_{\e,\k,m}.$
\end{remark}

Therefore, for each $n\in\N$, the functional equation \eqref{nproblemcod} is equivalent to solve the system
\begin{equation}\label{sys:codT=W}
z\in(-1,1),\qquad\left\{
\begin{aligned}
P_{n,\e}^{R_1}[\bar{h}_n^{R_1}, \bar{h}_n^{R_2}](z)&=\bar{H}_n^{R_1}(z),\\
P_{n,\e}^{R_2}[\bar{h}_n^{R_1}, \bar{h}_n^{R_2}](z)&=\bar{H}_n^{R_2}(z).
\end{aligned}
\right.
\end{equation}

For convenience, we remind the reader of the explicit expression of the operators on the left-hand side of the above system. That is, for any $\uu,\vv\in C^{2,\a}([-1,1])$, we have
\begin{multline}\label{def:PR1}
P_{n,\e}^{R_1}[\uu,\vv](z)=\Lambda_{m,\e}^{R_1}(z)\uu(z)\\
+(R_1+\e z)\frac{\sn((R_1+\e z)/r_1)}{\sn(r_2/r_1)}\frac{\e}{n}\int_{-1}^{+1} (R_1+\e s)\varphi_\k' (-s)\sn(r_2/(R_1+\e s))\uu(s)\ds\\
-(R_1+\e z)\frac{\sn((R_1+\e z)/r_1)}{\sn(r_2/r_1)}\frac{\e}{n}\int_{-1}^{+1} (R_2+\e s)\varphi_\k' (s)\sn(r_2/(R_2+\e s))\vv(s)\ds\\
-(R_1+\e z)\frac{\e}{n}\int_{-1}^{z} (R_1+\e s)\varphi_\k' (-s)\sn((R_1+\e z)/(R_1+\e s))\uu(s)\ds,
\end{multline}
and
\begin{multline}\label{def:PR2}
P_{n,\e}^{R_2}[\uu,\vv](z)=\Lambda_{m,\e}^{R_2}(z)\vv(z)\\
+(R_2+\e z)\frac{\sn((R_2+\e z)/r_1)}{\sn(r_2/r_1)}\frac{\e}{n}\int_{-1}^{+1} (R_1+\e s)\varphi_\k' (-s)\sn(r_2/(R_1+\e s))\uu(s)\ds \\
-(R_2+\e z)\frac{\sn((R_2+\e z)/r_1)}{\sn(r_2/r_1)}\frac{\e}{n}\int_{-1}^{+1} (R_2+\e s)\varphi_\k' (s)\sn(r_2/(R_2+\e s))\vv(s)\ds\\
-(R_2+\e z)\frac{\e}{n}\int_{-1}^{+1} (R_1+\e s)\varphi_\k' (-s)\sn((R_2+\e z)/(R_1+\e s))\uu(s)\ds\\
+(R_2+\e z)\frac{\e}{n}\int_{-1}^{ z} (R_2+\e s)\varphi_\k' (s)\sn((R_2+\e z)/(R_2+\e s))\vv(s)\ds,
\end{multline}
with $\Lambda_{m,\e}^{R_1}(z)$ and $\Lambda_{m,\e}^{R_2}(z)$ given by \eqref{def:Lambdacapital}.

\subsection{Analysis in a weak function space}
Here, it is important to  mention that the  analysis about the characterization for the range of the linear operator can be performed in a weighted $L^2([-1,1])\times L^2([-1,1])$ space. This weak regularity would be enough to conclude our result.

\subsubsection{Weighted $L^2([-1,1])\times L^2([-1,1])$ space}
To simplify the notation throughout this section, we begin by introducing two auxiliary weight functions. Note that thanks to the properties (see Lemma \ref{propiedadesvarphi}) of $\varphi_\k'$, we directly have that
\begin{equation}\label{def:auxweights}
\sigma_{\pm}(z):=-\varphi_\k'(\pm z)\geq 0,\qquad z\in(-1,1).
\end{equation}
Moreover, we introduce the space $L^2([-1,1])$ with weights $\sigma_{\pm}(z).$ That is, we define $L^2_{\sigma_{\pm}}([-1,1])$ with inner-product and norm given by
\[
(\uu_1,\uu_2)_{L^2_{\sigma_{-}}}:=\int_{-1}^{1} \uu_1(z) \uu_2(z) \sigma_{-}(z)\dz, \qquad  (\vv_1,\vv_2)_{L^2_{\sigma_{+}}}:=\int_{-1}^{1} \vv_1(z) \vv_2(z) \sigma_{+}(z)\dz,
\]
and
\begin{align*}
\|\uu\|^2_{L^2_{\sigma_{-}}}&:=(\uu,\uu)_{L^2_{\sigma_{-}}}=\int_{-1}^{1} |\uu(z)|^2  \sigma_{-}(z)\dz,\\
\|\vv\|^2_{L^2_{\sigma_{+}}}&:=(\vv,\vv)_{L^2_{\sigma_{+}}}=\int_{-1}^{1} |\vv(z)|^2  \sigma_{+}(z)\dz.
\end{align*}
In addition, we also introduce the product space $L^2([-1,1])\times L^2([-1,1])$ with weights $\sigma_{-}(z)$ and $\sigma_{+}(z)$ respectively. That is, we define the Hilbert space
\[
\mathbb{L}^2([-1,1]):=L^2_{\sigma_{-}}([-1,1])\times L^2_{\sigma_{+}}([-1,1]),
\]
with inner-product and norm given by
\[
\langle (\uu_1,\vv_1);(\uu_2,\vv_2) \rangle_{\mathbb{L}^2}:=(\uu_1,\uu_2)_{L^2_{\sigma_{-}}}+(\vv_1,\vv_2)_{L^2_{\sigma_{+}}},
\]
and
\[
\|(\uu,\vv)\|^2_{\mathbb{L}^2}:=\langle (\uu,\vv);(\uu,\vv) \rangle_{\mathbb{L}^2}=\|\uu\|_{L^2_{\sigma_{-}}}+\|\vv\|_{L^2_{\sigma_{+}}}.
\]
Finally, going back in the re-scaling, we introduce in our original domain $I_\e^{R_1,R_2}$ the space
\[
L_{\tilde{\sigma}}^2(I_\e^{R_1,R_2}):=
L_{\tilde{\sigma}_{-}}^2(I_\e^{R_1})\times L_{\tilde{\sigma}_{+}}^2(I_\e^{R_2}),
\]
with
\[
\tilde{\sigma}_{-}(\cdot):=\e^{-1}\sigma_{-}((\cdot-R_1)\e^{-1}), \qquad  \tilde{\sigma}_{+}(\cdot):=\e^{-1}\sigma_{+}((\cdot-R_2)\e^{-1}).
\]

After all this, we are ready to introduce the weaker space $Z(D_\e)$ given by
\[
Z(D_\e):=\left\lbrace (r,\theta)\in D_\e \mapsto \gg(r,\theta)=\sum_{n\geq 1}g_n(r)\cos(n\theta): \sum_{n\geq 1}\norm{g_n}_{L_{\tilde{\sigma}}^2(I_\e^{R_1,R_2})}^2<\infty \right\rbrace,
\]
and study the equation
\begin{equation*}
\cL[\l_{\e,\k,m}]\hh(r,\theta)=\HH(r,\theta), \qquad (r,\theta)\in  D_\e,
\end{equation*}
with $\hh,\HH\in Z(D_\e)$.

\subsection{Adjoint operator} Let $\mathbb{L}^2([-1,1])=L^2_{\sigma_{-}}([-1,1])\times L^2_{\sigma_{+}}([-1,1])$ be the weighted Hilbert space introduce above and consider the operator:
\[
\left(P_{n,\e}^{R_1},P_{n,\e}^{R_2}\right) \colon \begin{array}{>{\displaystyle}c @{} >{{}}c<{{}} @{} >{\displaystyle}c} 
          L^2_{\sigma_{-}}([-1,1])\times L^2_{\sigma_{+}}([-1,1]) &\rightarrow& L^2_{\sigma_{-}}([-1,1])\times L^2_{\sigma_{+}}([-1,1]) \\ 
          (\uu,\vv) &\mapsto& \left(P_{n,\e}^{R_1}[\uu,\vv],P_{n,\e}^{R_2}[\uu,\vv]\right) 
         \end{array}
\]
The main task of this subsection will be to obtain an explicit expression for the adjoint operator. That is, (identifying a Hilbert space with its dual) 
\[
\left((P_{n,\e}^{R_1})^\ast,(P_{n,\e}^{R_2})^\ast \right) \colon \begin{array}{>{\displaystyle}c @{} >{{}}c<{{}} @{} >{\displaystyle}c} 
          L^2_{\sigma_{-}}([-1,1])\times L^2_{\sigma_{+}}([-1,1]) &\rightarrow& L^2_{\sigma_{-}}([-1,1])\times L^2_{\sigma_{+}}([-1,1]) \\ 
          (\ff,\gg) &\mapsto& \left((P_{n,\e}^{R_1})^\ast[\ff,\gg],(P_{n,\e}^{R_2})^\ast[\ff,\gg]\right) 
         \end{array}
\]
such that
\[
\langle\left(P_{n,\e}^{R_1}[\uu,\vv],P_{n,\e}^{R_2}[\uu,\vv]\right);(\ff,\gg)\rangle_{\mathbb{L}^2}=\langle (\uu,\vv);\left((P_{n,\e}^{R_1})^\ast[\ff,\gg],(P_{n,\e}^{R_2})^\ast[\ff,\gg]\right)\rangle_{\mathbb{L}^2},
\]
or equivalently
\begin{equation}\label{eq:operatorandadjoint}
(P_{n,\e}^{R_1}[\uu,\vv],\ff)_{L^2_{\sigma_{-}}}+(P_{n,\e}^{R_2}[\uu,\vv],\gg)_{L^2_{\sigma_{+}}}=(\uu,(P_{n,\e}^{R_1})^\ast[\ff,\gg])_{L^2_{\sigma_{-}}} + (\vv,(P_{n,\e}^{R_2})^\ast[\ff,\gg])_{L^2_{\sigma_{+}}}.
\end{equation}
It is only a matter of straightforward but tedious calculations, which we leave to the interested reader, to verify that $\left(P_{n,\e}^{R_1}\right)^{\ast}[\ff,\gg]$ and $\left(P_{n,\e}^{R_2}\right)^{\ast}[\ff,\gg]$ are given by the following expressions:
\begin{multline}\label{def:adjPR1}
\left(P_{n,\e}^{R_1}\right)^{\ast}[\ff,\gg](z):=\Lambda_{m,\e}^{R_1}(z)\ff(z)\\
+ (R_1+\e z)\frac{\sn(r_2/(R_1+\e z))}{\sn(r_2/r_1)}\frac{\e}{n}\int_{-1}^{1}(R_1+\e s)\varphi_\k'(-s)\sn((R_1+\e s)/r_1) \ff(s)\ds\\
+(R_1+\e z)\frac{\e}{n}\int_{z}^{1} (R_1+\e s)\varphi_\k'(-s)\sn((R_1+\e z)/(R_1+\e s)) \ff(s) \ds\\
+(R_1+\e z)\frac{\sn(r_2/(R_1+\e z))}{\sn(r_2/r_1)}\frac{\e}{n}\int_{-1}^{+1} (R_2+\e s)\varphi_\k'(s)\sn((R_2+\e s)/r_1)\gg(s)\ds\\
+(R_1+\e z)\frac{\e}{n}\int_{-1}^{+1} (R_2+\e s) \varphi_\k'(s)\sn((R_1+\e z)/(R_2+\e s))\gg(s)\ds,
\end{multline}
and
\begin{multline}\label{def:adjPR2}
\left(P_{n,\e}^{R_2}\right)^{\ast}[\ff,\gg](z):=\Lambda_{m,\e}^{R_2}(z)\gg(z)\\
-(R_2+\e z)\frac{\sn(r_2/(R_2+\e z))}{\sn(r_2/r_1)}\frac{\e}{n} \int_{-1}^{1}(R_1+\e s)\varphi_\k'(-s) \sn((R_1+\e s)/r_1)\ff(s) \ds\\
-(R_2+\e z)\frac{\sn(r_2/(R_2+\e z))}{\sn(r_2/r_1)}\frac{\e}{n}\int_{-1}^{+1} (R_2+\e s)\varphi_\k'(s)\sn((R_2+\e s)/r_1)\gg(s)\ds\\
-(R_2+\e z)\frac{\e}{n}\int_{z}^{1} (R_2+\e s)\varphi_\k'(s) \sn((R_2+\e z)/(R_2+\e s))\gg(s)\ds.
\end{multline}

\subsection{Kernel of the adjoint operator}
The main task of this section is to prove that there exists an element in the kernel of the adjoint operator $\cL^{\ast}[\l_{\e,\k,m}]$. Moreover, we will see that the kernel of the adjoint operator is the span of this element and we will identify the precise form of this element at first order in terms of parameter $\e$.

The key point to prove our objective will be to use the theory of Fredholm operators. In first place, since $\cL[\l_{\e,\k,m}]$ is a Fredholm operator, their adjoint operator $\cL^\ast[\l_{\e,\k,m}]$ is also Fredholm with
\[
\text{index}(\cL^\ast[\l_{\e,\k,m}])=-\text{index}(\cL[\l_{\e,\k,m}]).
\]
Moreover, since the linear operator $\cL[\l_{\e,\k,m}]$ can be written as an isomorphism plus a compact operator we have (see \cite[Theorem 5, p. 641]{E}) that
\[
\text{dim}(\cN(\cL^\ast[\l_{\e,\k,m}]))=\text{dim}(\cN(\cL[\l_{\e,\k,m}])).
\]
From Theorem \ref{mainthmkernel} we can conclude directly that 
\[
\text{dim}(\cN(\cL^\ast[\l_{\e,\k,m}]))=1.
\]
So far we have shown that there exists $\hh^\ast\in Z(D_\e)$ such that $\cL^\ast[\l_{\e,\k,m}]\hh^\ast=0$ over $D_\e=I_\e^{R_1,R_2}\times\T$, which admits an expansion 
\[
\hh^\ast(r,\theta)=\sum_{n\geq 1} h_n^\ast(r)\cos(n\theta),
\]
satisfying $\cL_n^\ast[\l_{\e,\k,m}]h_n^\ast=0$ for all $n\in\N$ over $I_\e^{R_1,R_2}=I_\e^{R_1}\cup I_\e^{R_2}$. 

More specifically, using the same convention as before, that is, re-scaling and decomposing the problem, we have proved that there exists $\{\bar{h}_n^\ast\}_{n\geq 1}\equiv \{((\bar{h}_n^{R_1})^\ast,(\bar{h}_n^{R_2})^\ast)\}_{n\geq 1}\in \mathbb{L}^2([-1,1])$ satisfying
\begin{equation}
    \begin{cases}
        \left(P_{n,\e}^{R_1}\right)^\ast[(\bar{h}_n^{R_1})^\ast, (\bar{h}_n^{R_2})^\ast](z)=0,\\
        \left(P_{n,\e}^{R_2}\right)^\ast[(\bar{h}_n^{R_1})^\ast, (\bar{h}_n^{R_2})^\ast](z)=0,
    \end{cases}\qquad z\in(-1,1),
\end{equation}
for all $n\in \N$, with $\left(P_{n,\e}^{R_1}\right)^\ast [\cdot,\cdot]$ and $\left(P_{n,\e}^{R_2}\right)^\ast [\cdot,\cdot]$ given respectively by \eqref{def:adjPR1}-\eqref{def:adjPR2}.\\

In order to identify precisely the form of the unique element in the kernel of the adjoint operator, we will distinguish between two cases:

\subsubsection{Case $n\neq m$}
Since $\Lambda_{m,\e}^{R_1}(z)=O(1)$ and $\Lambda_{m,\e}^{R_2}(z)=O(\e)$ we can just write
\begin{multline}\label{eq:mainfnnotm}
\Lambda_{m,\e}^{R_1}(z)\ff(z)\\
+ (R_1+\e z)\frac{\sn(r_2/(R_1+\e z))}{\sn(r_2/r_1)}\frac{\e}{n}\int_{-1}^{1}(R_1+\e s)\varphi_\k'(-s)\sn((R_1+\e s)/r_1) \ff(s)\ds\\
+(R_1+\e z)\frac{\e}{n}\int_{z}^{1} (R_1+\e s)\varphi_\k'(-s)\sn((R_1+\e z)/(R_1+\e s)) \ff(s) \ds\\
+(R_1+\e z)\frac{\sn(r_2/(R_1+\e z))}{\sn(r_2/r_1)}\frac{\e}{n}\int_{-1}^{+1} (R_2+\e s)\varphi_\k'(s)\sn((R_2+\e s)/r_1)\gg(s)\ds\\
+(R_1+\e z)\frac{\e}{n}\int_{-1}^{+1} (R_2+\e s) \varphi_\k'(s)\sn((R_1+\e z)/(R_2+\e s))\gg(s)\ds=0,
\end{multline}
and (dividing the full expression by $\e$)
\begin{multline}\label{eq:maingnnotm}
\frac{\Lambda_{m,\e}^{R_2}(z)}{\e}\gg(z)\\
-(R_2+\e z)\frac{\sn(r_2/(R_2+\e z))}{\sn(r_2/r_1)}\frac{1}{n} \int_{-1}^{1}(R_1+\e s)\varphi_\k'(-s) \sn((R_1+\e s)/r_1)\ff(s) \ds\\
-(R_2+\e z)\frac{\sn(r_2/(R_2+\e z))}{\sn(r_2/r_1)}\frac{1}{n}\int_{-1}^{+1} (R_2+\e s)\varphi_\k'(s)\sn((R_2+\e s)/r_1)\gg(s)\ds\\
-(R_2+\e z)\frac{1}{n}\int_{z}^{1} (R_2+\e s)\varphi_\k'(s) \sn((R_2+\e z)/(R_2+\e s))\gg(s)\ds=0.
\end{multline}
Hence, we can divide \eqref{eq:mainfnnotm} by $\Lambda_{m,\e}^{R_1}(z)$ to get
\begin{equation}\label{solfnnotm}
\ff(z)=-\e \mathrm{F}[\ff,\gg](z),
\end{equation}
where
\begin{multline*}
\mathrm{F}[\ff,\gg](z):= (R_1+\e z)\frac{\sn(r_2/(R_1+\e z))}{\sn(r_2/r_1)}\frac{\e}{n}\int_{-1}^{1}(R_1+\e s)\varphi_\k'(-s)\sn((R_1+\e s)/r_1) \ff(s)\ds\\
+(R_1+\e z)\frac{\e}{n}\int_{z}^{1} (R_1+\e s)\varphi_\k'(-s)\sn((R_1+\e z)/(R_1+\e s)) \ff(s) \ds\\
+(R_1+\e z)\frac{\sn(r_2/(R_1+\e z))}{\sn(r_2/r_1)}\frac{\e}{n}\int_{-1}^{+1} (R_2+\e s)\varphi_\k'(s)\sn((R_2+\e s)/r_1)\gg(s)\ds\\
+(R_1+\e z)\frac{\e}{n}\int_{-1}^{+1} (R_2+\e s) \varphi_\k'(s)\sn((R_1+\e z)/(R_2+\e s))\gg(s)\ds.
\end{multline*}
The next step will be to introduce \eqref{solfnnotm} into the equation \eqref{eq:maingnnotm}. Hence, we have
\begin{multline}\label{eq:preauxS}
\frac{\Lambda_{m,\e}^{R_2}(z)}{\e}\gg(z)\\
+(R_2+\e z)\frac{\sn(r_2/(R_2+\e z))}{\sn(r_2/r_1)}\frac{\e}{n} \int_{-1}^{1}(R_1+\e s)\varphi_\k'(-s) \sn((R_1+\e s)/r_1)\mathrm{F}[\ff,\gg](s) \ds\\
-(R_2+\e z)\frac{\sn(r_2/(R_2+\e z))}{\sn(r_2/r_1)}\frac{1}{n}\int_{-1}^{+1} (R_2+\e s)\varphi_\k'(s)\sn((R_2+\e s)/r_1)\gg(s)\ds\\
-(R_2+\e z)\frac{1}{n}\int_{z}^{1} (R_2+\e s)\varphi_\k'(s) \sn((R_2+\e z)/(R_2+\e s))\gg(s)\ds=0.
\end{multline}
Now, we introduce the auxiliary functional
\begin{multline}\label{eq:auxS}
S_\e[\gg](z):=-(R_2+\e z)\frac{\sn(r_2/(R_2+\e z))}{\sn(r_2/r_1)}\frac{1}{n}\int_{-1}^{+1} (R_2+\e s)\varphi_\k'(s)\sn((R_2+\e s)/r_1)\gg(s)\ds\\
-(R_2+\e z)\frac{1}{n}\int_{z}^{1} (R_2+\e s)\varphi_\k'(s) \sn((R_2+\e z)/(R_2+\e s))\gg(s)\ds,
\end{multline}
which can be written as 
\[
S_\e[\gg](z)=S_0[\gg](z)+\int_0^\e (\p_\delta S_\delta)[\gg](z)\mbox{d}\delta,
\]
with (using \eqref{def:ppi})
\[
S_0[\gg](z)=-\pp_2(n)\int_{-1}^{+1}  \varphi_\k' (s)\gg(s)\ds.
\]
Then, using \eqref{eq:auxS} we have that \eqref{eq:preauxS} can be written as
\begin{multline*}
\a_1^{R_2}[\l_0,\l_1](z)\gg(z)    -\pp_2(n)\int_{-1}^{+1}  \varphi_\k' (s)\gg(s)\ds=-\alpha_{2,\e}^{R_2}[\l_0,\l_1,\l_2^{\e}](z)\gg(z)\e-\int_0^\e (\p_\delta S_\delta)[\gg](z)\mbox{d}\delta\\
+(R_2+\e z)\frac{\sn(r_2/(R_2+\e z))}{\sn(r_2/r_1)}\frac{\e}{n} \int_{-1}^{1}(R_1+\e s)\varphi_\k'(-s) \sn((R_1+\e s)/r_1)\mathrm{F}[\ff,\gg](s) \ds.
\end{multline*}
Introducing a new auxiliary functional $\e \mathrm{G}[\ff,\gg](z)$ for the right hand side of the above expression, we arrive to
\begin{equation}\label{solvinggnneqm}
\a_1^{R_2}[\l_0,\l_1](z)\gg(z)    -\pp_2(n)\int_{-1}^{+1}  \varphi_\k' (s)\gg(s)\ds=\e \mathrm{G}[\ff,\gg](z),
\end{equation}
with
\begin{multline*}
\mathrm{G}[\ff,\gg](z):=  -\alpha_{2,\e}^{R_2}[\l_0,\l_1,\l_2^{\e}](z)\gg(z)-\frac{1}{\e}\int_0^\e (\p_\delta S_\delta)[\gg](z)\mbox{d}\delta\\
+(R_2+\e z)\frac{\sn(r_2/(R_2+\e z))}{\sn(r_2/r_1)}\frac{1}{n} \int_{-1}^{1}(R_1+\e s)\varphi_\k'(-s) \sn((R_1+\e s)/r_1)\mathrm{F}[\ff,\gg](s) \ds,  
\end{multline*}
where $\|\mathrm{G}\|_{L^2([-1,1])}=O(\e)$, see \eqref{DFSV} for details.
Now, applying Lemma \ref{lemma:invertnneqm}, we have that solution of \eqref{solvinggnneqm} is given by
\begin{equation}\label{solgnnotm}
\gg(z)=\e I[\mathrm{G}[\ff,\gg]](z).
\end{equation}

The coupled system given by \eqref{solfnnotm} and \eqref{solgnnotm} is a linear contraction on $\mathbb{L}
^2([-1, 1])$ for $\e$ small
enough. Hence, there exists a unique $(\ff,\gg)\in \mathbb{L}
^2([-1, 1])$ solving \eqref{solfnnotm}, \eqref{solgnnotm}: 
\begin{align*}
\ff(z)&=-\e \mathrm{F}[\ff,\gg](z),\\
\gg(z)&=\e I[\mathrm{G}[\ff,\gg]](z).
\end{align*}

Therefore, the trivial solution $(\ff, \gg)=(0,0)$ is the unique solution of \eqref{eq:mainfnnotm} and \eqref{eq:maingnnotm}.

\subsubsection{Case $n=m$}
To simplify the notation, in the following we will just write
\[
\bar{h}_m^\ast(z)\equiv \left((\bar{h}_m^{R_1})^\ast(z),(\bar{h}_m^{R_1})^\ast(z)\right)=:(\aa^\ast(z),\bb^\ast(z)).
\]
Taking into account this notation and what we have seen for the case $n\neq m$, we have that the unique element in the kernel of the adjoint operator $\cL_n^\ast[\l_{\e,\k,m}]$ has the form:
\begin{align}
(\bar{h}_n^{R_1})^\ast(z)=\left\{\begin{array}{cc} 0 & n\neq m,\\ \aa^\ast(z) & n=m, \end{array}\right.\label{hR1soladj}\\
(\bar{h}_n^{R_2})^\ast(z)=\left\{\begin{array}{cc} 0 & n\neq m,\\ \bb^\ast(z) & n=m, \end{array}\right.\label{hR2soladj}
\end{align}
where functions $\aa^\ast$ and $\bb^\ast$ depend on $\e$, $\k$ and $m$ but we do not make this dependence explicit for sake of simplicity.

To finish this section we need to obtain an explicit expression for $(\aa^\ast,\bb^\ast)$ at the first order in terms of parameter $\e$. More specifically, let us see that for some constant C we have
\begin{equation}\label{sol:adjexpansion}
\begin{cases}
\aa^\ast(z)=\phantom{C \bb_0(z)}\hspace{0.5 cm} O(\e),\\
\bb^\ast(z)=C \bb_0(z)+ O(\e).
\end{cases}	
\end{equation}
Here, $\bb_0(z)$ is given by \eqref{b0formula}. That is, the element in the kernel of the linear operator is exactly the same as (modulo a multiplicative constant) the element in the kernel of the adjoint operator at first order in terms of parameter $\e.$
\begin{remark}
The linear operator $\cL[\l_{\e,\k,m}]$ is self-adjoint at first order in terms of parameter $\e$.
\end{remark}
\begin{remark}
The information provided by \eqref{sol:adjexpansion} will be crucial to prove the transversality property in the Crandall-Rabinowitz's theorem.
\end{remark}

\begin{proof}[Proof of \eqref{sol:adjexpansion}]
We start by recalling what we know at this point. That is, $(\aa^\ast,\bb^\ast)$ solves 
\begin{equation*}
\begin{cases}
P_{m,\e}^{R_1}[\aa^\ast,\bb^\ast](z)=0,\\
P_{m,\e}^{R_2}[\aa^\ast,\bb^\ast](z)=0,
\end{cases}\qquad z\in(-1,1).
\end{equation*}
The above system is linear in $(\aa^\ast, \bb^\ast)$ and we can assume without loss of generality that 
\[
\|\aa^\ast\|_{L_{\sigma_{-}}^2}^2+\|\bb^\ast\|_{L_{\sigma_{+}}^2}^2=1.
\]
If it is not the case we only need to normalized the solution and enter that value into the final constant $C$. Now, using \eqref{def:adjPR1} we have that
\[
0=P_{m,\e}^{R_1}[\aa^\ast,\bb^\ast](z)=\Lambda_{m,\e}^{R_1}(z)\aa^\ast(z) + O(\e),
\]
but since we have seen  that $\Lambda_{m,\e}^{R_1}(z)=O(1)$ we get that $\aa^\ast(z)=\e \aa_1^\ast(z)$ with $\|\aa_1^\ast\|_{L_{\sigma_{-}}^2}=O(1)$
in terms of $\e$. This information yields from \eqref{def:adjPR2} that
\begin{multline*}
0=P_{m,\e}^{R_2}[\aa^\ast,\bb^\ast](z)=\Lambda_{m,\e}^{R_2}(z)\bb^\ast(z)+O(\e^2)\\
-(R_2+\e z)\frac{\sm(r_2/(R_2+\e z))}{\sm(r_2/r_1)}\frac{\e}{m}\int_{-1}^{+1} (R_2+\e s)\varphi_\k'(s)\sm((R_2+\e s)/r_1)\bb^\ast(s)\ds\\
-(R_2+\e z)\frac{\e}{m}\int_{z}^{1} (R_2+\e s)\varphi_\k'(s) \sm((R_2+\e z)/(R_2+\e s))\bb^\ast(s)\ds.
\end{multline*}
Since $\Lambda_{m,\e}^{R_2}(z)=  \a_1^{R_2}[\l_0,\l_1](z)\e+\a_{2,\e}^{R_2}[\l_0,\l_1,\l_2^\e](z)\e^2$, we divide the above expression by $\e$ to get
\begin{multline}\label{eq:preauxV}
\a_1^{R_2}[\l_0,\l_1](z)\bb^\ast(z) -(R_2+\e z)\frac{1}{m}\int_{z}^{1} (R_2+\e s)\varphi_\k'(s) \sm((R_2+\e z)/(R_2+\e s))\bb^\ast(s)\ds\\
-(R_2+\e z)\frac{\sm(r_2/(R_2+\e z))}{\sm(r_2/r_1)}\frac{1}{m}\int_{-1}^{+1} (R_2+\e s)\varphi_\k'(s)\sm((R_2+\e s)/r_1)\bb^\ast(s)\ds=O(\e).
\end{multline}
Now, we introduce the auxiliary functional
\begin{multline}\label{eq:auxV}
V_\e[\bb^\ast](z):=-(R_2+\e z)\frac{1}{m}\int_{z}^{1} (R_2+\e s)\varphi_\k'(s) \sm((R_2+\e z)/(R_2+\e s))\bb^\ast(s)\ds\\
-(R_2+\e z)\frac{\sm(r_2/(R_2+\e z))}{\sm(r_2/r_1)}\frac{1}{m}\int_{-1}^{+1} (R_2+\e s)\varphi_\k'(s)\sm((R_2+\e s)/r_1)\bb^\ast(s)\ds,
\end{multline}
which can be written as
\[
V_\e[\bb^\ast](z)=V_0[\bb^\ast](z)+\int_0^\e (\p_\d V_\d)[\bb^\ast](z) \mbox{d}\d,
\]
with (see \eqref{def:ppi})
\[
V_0[\bb^\ast](z)=-\pp_2(m)\int_{-1}^{+1} \varphi_\k'(s)\bb^\ast(s)\ds.
\]
Then, using \eqref{eq:auxV} we have that \eqref{eq:preauxV} can be written as
\begin{equation}\label{eq:auxG}
\a_1^{R_2}[\l_0,\l_1](z)\bb^\ast(z)-\pp_2(m)\int_{-1}^{+1} \varphi_\k'(s)\bb^\ast(s)\ds=O(\e)-\int_0^\e (\p_\d V_\d)[\bb^\ast](z) \mbox{d}\d:=G(z),
\end{equation}
where $\|G\|_{L^2([-1,1])}=O(\e)$, see \eqref{DFSV} for details. Just dividing \eqref{eq:auxG} by $\a_1^{R_2}[\l_0,\l_1](z)$ we have
\[
\bb^\ast(z)=\frac{\pp_2(m)}{\a_1^{R_2}[\l_0,\l_1](z)}\int_{-1}^{+1}  \varphi_\k' (s)\bb^\ast(s)\ds + \frac{G(z)}{\a_1^{R_2}[\l_0,\l_1](z)}.
\]
Thus, recalling the precise form of $\bb_0(z)$, given by \eqref{b0formula} we obtain that
\begin{equation*}
\bb^\ast(z)= C \bb_0(z) + \e \bb_1^\ast(z), 
\end{equation*}
where $\|\bb_1^\ast\|_{L_{\sigma_{+}}^2}=O(1)$ in terms of $\e$ and $C$ is a constant.
\end{proof}

\subsection{Image of the linear operator}
In the previous section we have seen that 
\[
\cN(\cL^\ast[\l_{\e,\k,m}])=\text{span}\{\hh^\ast_{\e,\k,m}\},
\]
where $\hh^\ast_{\ep,\kappa,m}(r,\theta)=h^\ast(r)\cos(m \theta)$ and with $h^\ast(\cdot)$ given by
\begin{equation*}
h^\ast(r)= \left\{
\begin{aligned}
\aa^\ast\left(\tfrac{r-R_1}{\e}\right) \qquad &\text{if} \quad  r\in I_\e^{R_1},\\
\bb^\ast\left(\tfrac{r-R_2}{\e}\right) \qquad &\text{if} \quad  r\in I_\e^{R_2}.
\end{aligned}\right.
\end{equation*}
Moreover, we have seen that $(\aa^\ast,\bb^\ast)=(\e\aa_1^\ast,\bb_0+\e\bb_1^\ast)$ with $\aa_1^\ast, \bb_1^\ast=O(1)$ in terms of parameter $\e$.

We also define the orthogonal complement of $(\aa^\ast,\bb^\ast)$ in $L^2_{\sigma_{-}}([-1,1])\times L^2_{\sigma_{+}}([-1,1])$:
\[
(\aa^\ast,\bb^\ast)^\perp:=\left\lbrace (\uu,\vv)\in L^2_{\sigma_{-}}([-1,1])\times L^2_{\sigma_{+}}([-1,1])\, \mid \, (\uu,\aa^\ast)_{L^2_{\sigma_{-}}} + (\vv,\bb^\ast)_{L^2_{\sigma_{+}}}=0 \right\rbrace.
\]
Since for this kind of operators, it is well known that $\cR(\cL[\l_{\e,\k,m}])=\cN(\cL^\ast[\l_{\e,\k,m}])^{\perp}$,  we are in a good position to prove the following lemma.

\begin{lemma}\label{necesario}
Let $\HH\in \cR(\cL[\l_{\e,\k,m}])\subset Z(D_\e)$ with expansion given by $\HH(r,\theta)=\sum_{n\geq 1} H_n(r)\cos(n\theta)$. Then, we have that $\bar{H}_m\equiv(\bar{H}_m^{R_1},\bar{H}_m^{R_2})\in (\aa^\ast,\bb^\ast)^\perp$.
\end{lemma}
\begin{proof}
Direct computation using \eqref{eq:operatorandadjoint}.
\end{proof}

\section{The transversality property}\label{s:transversality}
This section is devoted to the transversality assumption concerning the fourth and last hypothesis of the Crandall–Rabinowitz's theorem.

Up to this point, we have shown that the linearized operator $\cL[\lambda_{\e,\k,m}]\equiv D_\ff F[\lambda_{\e,\k,m},0]$ is a Fredholm operator of zero index with one-dimensional kernel. This property is not enough to bifurcate to non-trivial solutions for the nonlinear problem.
A sufficient condition for that, according to Theorem \ref{th:CR}, is the transversality property which  is equivalent to checking
\[
D^2_{\l,\ff}F[\lambda_{\e,\k,m}, 0]\hh_{\e,\k,m} \not\in \cR(\cL[\lambda_{\e,\k,m}]),
\]
where $\hh_{\e,\k,m}$ is the generator (unique modulo a multiplicative constant) of the kernel of $\cL[\lambda_{\e,\k,m}].$ 

Recall that for any $m\in \N$ we have proved that there exist an eigenvalue $\l_{\ep,\kappa,m}\in\R$ given by Theorem \ref{mainthmkernel} such that $\cN( \cL[\lambda_{\ep,\kappa,m}]) = \text{span}\{\hh_{\ep,\kappa,m}\},$ where $\hh_{\ep,\kappa,m}(r,\theta)=h(r)\cos(m \theta)$ and with $h(\cdot)$ given by
\begin{equation}\label{def:tranversality}
h(r)= \left\{
\begin{aligned}
\aa\left(\tfrac{r-R_1}{\e}\right) \qquad &\text{if} \quad  r\in I_\e^{R_1},\\
\bb\left(\tfrac{r-R_2}{\e}\right) \qquad &\text{if} \quad  r\in I_\e^{R_2}.
\end{aligned}\right.
\end{equation}
To finish, let us proceed by reduction to absurd. That is, we are going to assume that 
$$D_{\l,\ff}^2 F[\lambda_{\ep,\kappa,m},0]\hh_{\ep,\kappa,m}= r h(r)\cos( m \theta)\in \cR(\cL[\lambda_{\e,\k,m}]).$$
Then, there exists $\gg=g(r)\cos(m\theta)\in X(D_\e)$ such that $\cL_m[\lambda_{\e,\k,m}]g(r)=r h(r)$ for all $r\in I_\e^{R_1,R_2}.$

Now, applying Lemma \ref{necesario}  we have that if  $D_{\l,\ff}^2 F[\l_{\ep,\kappa,m},0]\hh_{\ep,\kappa,m} \in\cR(\cL_{\e,\k}[\l_{\ep,\kappa,m}]),$ then 
\[
\int_{-1}^1 (R_1+\e z)\aa(z)\aa^\ast(z)\sigma_{-}(z)\dz + \int_{-1}^{1}(R_2+\e z)\bb(z)\bb^\ast(z)\sigma_{+}(z)\dz=0.
\]
Taking into account the $\e$-expansion of each of the terms that appear above and using the fact that $\bb_0=\bb_0^\ast$ we get a contradiction because
\begin{align*}
\int_{-1}^1 (R_1+\e z)\aa(z)\aa^\ast(z)\sigma_{-}(z)\dz&=O(\e^2),\\
\int_{-1}^{1}(R_2+\e z)\bb(z)\bb^\ast(z)\sigma_{+}(z)\dz&=\underbrace{\int_{-1}^{1}(R_2+\e z)|\bb_0(z)|^2\sigma_{+}(z)\dz}_{O(1)}+O(\e).
\end{align*}
This concludes the last required condition of Crandall-Rabinowitz  theorem \ref{th:CR}.

\section{Main theorem}\label{s:mainthm}

In this section we shall provide a general statement that precise Theorem \ref{thmbasic} and give its proof using all the previous results.

After verifying  all the conditions for the application of the Crandall-Rabinowitz  Theorem \ref{th:CR} and the discussion in Section \ref{s:formulation} we obtain the following theorem:

\begin{theorem}\label{thmGOAL} Fixed $M>1$. There exist $\ep_0(M)$, $\kappa_0(M)$ such that, for every $0<\ep<\ep_0$, $0<\kappa<\kappa_0$ and $m\in \N$ with $1\leq m < M$, there exist a branch of solutions $\ff_{\ep,\kappa,m}^\sigma\in X(D_\e)$ parameterize by $\sigma$, of equation  \eqref{functionaleq}, with $|\sigma|<\sigma_0$, for some small number $\sigma_0>0$, $\varpi_{\ep,\kappa}$ as in Section \ref{s:profile} and $\lambda=\lambda^\sigma_{\ep,\kappa,m}.$ These solutions satisfy:
\begin{enumerate}
\item $\ff^\sigma_{\ep,\kappa,m}(r,\theta)$ is $\frac{2\pi}{m}-$periodic on $\theta$.
\item The branch $$\ff^\sigma_{\ep,\kappa,m}=\sigma h_{\ep,\kappa,m}+o(\sigma) \quad\text{in } X(D_\e),$$
and the frequency of the temporal periodicity is 
$$\lambda^\sigma_{\ep,\kappa,m}=\lambda_{\ep,\kappa,m}+o(1),$$ where $(\lambda_{\ep,\kappa,\sigma},\, h_{\ep,\kappa,m})$ are given in Theorem \ref{mainthmkernel} and Remark \ref{remark}.
\item  $\ff^\sigma_{\ep,\kappa,m}(r,\theta)$ depends on $\theta$ in a non-trivial way.
\end{enumerate}
In addition, the vorticity $\omega^\sigma_{\kappa,\sigma,m}\in C^{2,\alpha}(\mathsf{P}_{r_1,r_2})$,  given implicitly by 
\begin{align}\label{auxDL2_1}
&\omega^\sigma_{\ep,\kappa,m}(\tilde{r},\tilde{\theta})=2A+\varpi_{\e,\k}(r),
\end{align}
for $(\tilde{r},\tilde{\theta})=(r+\ff^\sigma_{\ep,\kappa,m}(r,\theta),\theta)$ with $\theta\in\T$ and $r\in I_\e^{R_1}\cup I_\e^{R_2}$,
\begin{align}\label{auxDL2_2}
\omega^\sigma_{\ep,\kappa,m}(\tilde{r},\tilde{\theta})=2A+\e,
\end{align}
for $\tilde{\theta}\in\T$ and $R_1+\e+\ff^\sigma_{\ep,\kappa,m}(R_1+\e,\tilde{\theta})<\tilde{r}<R_2-\e+\ff^\sigma_{\ep,\kappa,m}(R_2-\e,\tilde{\theta})$, and
\begin{align}\label{auxDL2_3}
\omega^\sigma_{\ep,\kappa,m}(\tilde{r},\tilde{\theta})=2A,
\end{align}
for $\tilde{\theta}\in \T$ and either $r_2\geq \tilde{r}>R_2+\ep+\ff^\sigma_{\ep,\kappa,m}(R_2+\ep,\tilde{\theta})$ or $r_1\leq \tilde{r}<R_1-\ep +\ff^{\sigma}_{\ep,\kappa,m}(R_1-\ep,\tilde{\theta})$,\\
yields a time periodic  solution for 2D Euler in the sense that $$\omega^\sigma_{\ep,\kappa,m}(\tilde{r},\tilde{\theta}+\lambda^\sigma_{\ep,\kappa,m}t), \qquad (\tilde{r},\tilde{\theta})\in\mathsf{P}_{r_1,r_2},$$
satisfies \eqref{eq:2dEvorticity}-\eqref{eq:BVPpsi}. Importantly, $\omega^\sigma_{\ep,\kappa,m}(\tilde{r},\tilde{\theta})$ depends non trivially on $\tilde{\theta}$.
\end{theorem}

Then, in order to prove Theorem \ref{thmbasic} it remains to prove that $H^{\sfrac{3}{2}-}(\mathsf{P}_{r_1,r_2})$-norm of $\omega^\sigma_{\ep,\kappa,m}$ can be made as small as we want and that $\omega^\sigma_{\ep,\kappa,\sigma}\in C^\infty(\mathsf{P}_{r_1,r_2})$. We do this in Theorems \ref{thm:distance} and \ref{thm:fullregularity}.

\subsection{Distance of the time periodic solution $\omega^\sigma_{\ep,\kappa,\sigma}$ to the Taylor-Couette flow}\label{s:distance}
The solution $\omega^\sigma_{\ep,\kappa,m}$ obtained in Theorem \ref{thmGOAL} satisfies the following statement: 
\begin{theorem}\label{thm:distance}
Fixed $M>1$, $0<\kappa<\kappa_0$ and $0<\gamma< \sfrac{3}{2}$, for all $\epsilon>0$ and $1\leq m<M$, there exist $\ep>0$ and $\sigma>0$ such that
\begin{align*}
\|\omega^\sigma_{\ep,\kappa,m}(\cdot,\cdot+\l^\sigma_{\ep,\kappa,m} t)-2A\|_{H^\gamma(\mathsf{P}_{r_1,r_2})}<\epsilon.
\end{align*}
\end{theorem}

\begin{proof}
Let us emphasis that we can make $\|\ff^\sigma_{\ep,\kappa,m}\|_{C^{2,\alpha}(D_\e)}$ arbitrarily small fixed $\ep$, $\kappa$ and $m$, taking $\sigma$ small. We have all the ingredients to obtain a quantitative estimate of the distance between the Taylor-Couette flow and the constructed time periodic solution.

We can compute the homogeneous $H^{\frac{3}{2}-}$ norm of $w^\sigma_{\ep,\kappa,m}(\cdot,\cdot+\l^\sigma_{\ep,\kappa,m} t)$ using interpolation of Sobolev norms. We introduce  the auxiliary function $\tilde \omega^\sigma_{\ep,\kappa,m}(\cdot,\cdot)=\omega^\sigma_{\ep,\kappa,m}(\cdot,\cdot+\l^\sigma_{\ep,\kappa,m} t)$ such that
\begin{equation*}
\|\tilde \omega^\sigma_{\ep,\kappa,m}-2A\|_{\dot{H}^{\frac{3-\g}{2}}(\mathsf{P}_{r_1,r_2})}= \|\tilde \omega^\sigma_{\ep,\kappa,m}\|_{\dot{H}^{\frac{3-\g}{2}}(\mathsf{P}_{r_1,r_2})}
\lesssim \| \tilde \omega^\sigma_{\ep,\kappa,m}\|_{\dot{H}^1 (\mathsf{P}_{r_1,r_2})}^{\frac{1+\g}{2}} \|\tilde \omega^\sigma_{\ep,\kappa,m}\|_{\dot{H}^2 (\mathsf{P}_{r_1,r_2})}^{\frac{1-\g}{2}}.
\end{equation*}
Moreover, as $\norm{\tilde \omega^\sigma_{\ep,\kappa,m}}_{\dot{H}^k (\mathsf{P}_{r_1,r_2})}= \norm{\omega^\sigma_{\ep,\kappa,m}}_{\dot{H}^k (\mathsf{P}_{r_1,r_2})}$ for $k=1,2$, our problem reduces to
\begin{equation}\label{boundgoal}
\|\omega^\sigma_{\ep,\kappa,m}(\cdot,\cdot+\l^\sigma_{\ep,\kappa,m} t)-2A\|_{\dot{H}^{\frac{3-\g}{2}}(\mathsf{P}_{r_1,r_2})}\lesssim \| \omega^\sigma_{\ep,\kappa,m}\|_{\dot{H}^1 (\mathsf{P}_{r_1,r_2})}^{\frac{1+\g}{2}} \|\omega^\sigma_{\ep,\kappa,m}\|_{\dot{H}^2 (\mathsf{P}_{r_1,r_2})}^{\frac{1-\g}{2}}.
\end{equation}

\begin{remark}
This way of proceeding will make us lose the independence on $\kappa$. But since we are actually interested on the case $\kappa>0$ it will be good enough.
\end{remark}

To alleviate the notation let us skip the subscripts $(\ep,\kappa,m)$ and the superscript $\sigma$ on $\omega^\sigma_{\ep,\kappa,m}$ and $\ff^\sigma_{\ep,\kappa,m}$ in the rest of the paper. In this section, we will use the notation $\bz=(r,\theta)\in D_\e$ when it is clear from the context. We will keep $\varpi_{\ep,\kappa}$ and $\varphi_{\kappa}$ as we did before.

In order to compute the right-hand side of \eqref{boundgoal}, we have (see \eqref{eq:nablaomega} and \eqref{ansatzrho}) that
\begin{equation}\label{1derivative}
\nabla \omega (r+\ff(\bz),\theta)= \frac{\varpi_{\e,\k}'(r)}{1+\ff_r(\bz)}\left(1,-\frac{\ff_\theta(\bz)}{r+\ff(\bz)}\right),\quad \text{on $\supp(\nabla \omega)$},
\end{equation}
thus, making the appropriate change of variables, we obtain
\begin{align*}
\|\omega\|_{\dot{H}^2(\mathsf{P}_{r_1,r_2})}^2=&\int_{\mathsf{P}_{r_1,r_2}}\left|\n^2 \omega(\bz)\right|^2 \mbox{d}\bz=\int_{\supp(\n^2\omega)}\left|\n^2 \omega(\bz)\right|^2 \mbox{d}\bz\\
=&\int_{D_\e}|\n^2\omega (r+\ff(\bz),\theta)|^2(1+\ff_r(\bz))\mbox{d}\bz.
\end{align*}
Recalling that $\nabla\equiv (\p_r, \tfrac{1}{r}\p_\theta)$, the expression \eqref{1derivative} can be written component-by-component as
\begin{align*}
(\p_r \omega)(r+\ff(\bz),\theta)&=\frac{\varpi_{\e,\k}'(r)}{1+\ff_r(\bz)},\\
(\p_\theta \omega)(r+\ff(\bz),\theta)&=-\ff_\theta(\bz)\frac{\varpi_{\e,\k}'(r)}{1+\ff_r(\bz)}.
\end{align*}
We are interested in $\nabla^2$, that is $\p_r^2, \tfrac{1}{r^2}\p_\theta^2$ and $\tfrac{1}{r}\p_r \p_\theta-\tfrac{1}{r^2}\p_\theta$. Hence, computing second order derivatives we have
\[
(\p_r^2 \omega)(r+\ff(\bz),\theta)(1+\ff_r(\bz))=\p_r\left(\frac{1}{1+\ff_r(\bz)}\right)\varpi_{\e,\k}'(r)+\frac{\varpi_{\e,\k}''(r)}{1+\ff_r(\bz)},
\]
\[
(\p_\theta^2 \omega)(r+\ff(\bz),\theta) +(\p_r \p_\theta w)(r+\ff(\bz),\theta) \ff_\theta(\bz)=\p_\theta\left(\frac{-\ff_\theta(\bz)}{1+\ff_r(\bz)} \right)\varpi_{\e,\k}'(r),
\]
and
\[
(\p_r \p_\theta \omega)(r+\ff(\bz),\theta) + (\p_r^2 w)(r+\ff(\bz),\theta)\ff_\theta(\bz)=\p_\theta\left(\frac{1}{1+\ff_r(\bz)} \right)\varpi_{\e,\k}'(r).
\]
Manipulating the above expression properly we get
\begin{align*}
\pa_{r}^2\omega(r+\ff(\bz),\theta)=&A_1(\bz)\varpi_{\e,\k}'(r)+A_2(\bz)\varpi_{\e,\k}''(r),\\
\frac{1}{|r+\ff(\bz)|^2}\pa_{\theta}^2\omega(r+\ff(\bz),\theta)=&A_3(\bz)\varpi_{\e,\k}'(r)+A_4(\bz)\varpi_{\e,\k}''(r),\\
\frac{1}{r+\ff(\bz)}\pa_r \pa_{\theta}\omega(r+\ff(\bz),\theta)+\frac{1}{|r+\ff(\bz)|^2}\pa_{\theta}\omega(r+\ff(\bz),\theta)=&A_5(\bz)\varpi_{\e,\k}'(r)+A_6(\bz)\varpi_{\e,\k}''(r),\\
\end{align*}
where $A_{i}$ for $i=1,...,6$ are functions which depend on $\pa_r \ff, \pa_\theta \ff, \pa^2_{r}\ff, \pa^2_{\theta}\ff, \pa_{r}\pa_{\theta}\ff$.

The $C^{2,\alpha}-$norm of $\ff$ could depend on $\e$ badly. However we always could choose $\s$ small enough in such a way that the $C^{2,\alpha}$-norm of $\ff$ is small. Noticing that $C^{2,\alpha} \subset C^2$ we have that $A_i$ for $i=1,...,6$ are bounded functions.
Then,  we have that 
\begin{align*}
\|\omega\|_{\dot{H}^1(\mathsf{P}_{r_1,r_2})} &\approx \|\varpi_{\e,\k}'\|_{L^2\left(I_\ep^{R_1,R_2}\right)},\\
\|\omega\|_{\dot{H}^2(\mathsf{P}_{r_1,r_2})} &\approx \|\varpi_{\e,\k}'\|_{L^2\left(I_\ep^{R_1,R_2}\right)}+\|\varpi_{\e,\k}''\|_{L^2\left(I_\ep^{R_1,R_2}\right)}.
\end{align*}
Since $I_\ep^{R_1,R_2}=I_\e^{R_1}\cup I_\e^{R_2}=[R_1-\e,R_1+\ep]\cup[R_2-\ep,R_2+\ep]$, we just have to study the norms:
\[
\|\varpi_{\e,\k}'\|_{L^2([R_i-\e,R_i+\e])}, \|\varpi_{\e,\k}''\|_{L^2([R_i-\e,R_i+\e])}.
\]
Recalling the definitions of $\varpi_{\ep,\kappa}$ and  $\varphi_{\kappa}$ in section \ref{s:profile}. We have for $|r|\leq 1$ that
\begin{align*}
\varpi_{\ep,\kappa}'(R_i+\e r)&=(-1)^i\varphi_\k'\left((-1)^i r\right),\\
\varpi_{\ep,\kappa}''(R_i+\e r)&=\e^{-1}\varphi_\k''\left((-1)^i r\right).
\end{align*}
Then, after doing the changes of variables $r\leftrightarrow R_i+\e r$ and $r\leftrightarrow (-1)^i r$ we get
\begin{equation}\label{auxH1}
\|\varpi_{\e,\k}'\|^2_{L^2([R_i-\e,R_i+\e])}=\e\int_{-1}^1 |\varphi_\k'\left( r\right)|^2 \dr \leq 2\ep \|\varphi_{\kappa}'\|_{L^{\infty}([-1,1])}^2,
\end{equation}
and
\begin{align}\label{varpiH2}
\|\varpi_{\e,\k}''\|^2_{L^2([R_i-\e,R_i+\e])}=\frac{1}{\e}\int_{-1}^{1}|\varphi_\k''\left( r\right)|^2\dr.
\end{align}
Since (modulo a normalization factor)
\begin{align*}
\varphi'_{\kappa}(z)=\int_{-1+\kappa}^{1+\kappa}\Theta\left(\frac{z-\bar{z}}{\k}\right)\mbox{d}\bar{z},
\end{align*}
thus
\begin{align*}
\varphi''_{\kappa}(z)=\frac{1}{\k}\left[\Theta\left(\frac{z-(1-\k)}{\k} \right) - \Theta\left(\frac{z-(-1+\k)}{\k}\right)\right].
\end{align*}
Putting the above expression into \eqref{varpiH2} we get
\begin{align*}
\|\varpi_{\e,\k}''\|^2_{L^2([R_i-\e,R_i+\e])} \leq \frac{1}{\ep\k^2}\left(\int_{-1}^1 \left|\Theta\left(\frac{r-(-1+\kappa)}{\kappa}\right)\right|^2\dr+\int_{-1}^1  \left|\Theta\left(\frac{r-(1-\kappa)}{\kappa}\right)\right|^2\dr\right).
\end{align*}
Since $\supp(\Theta)\subset(-1,1)$ we have that the limits of integration of each integral are $r\in(-1,2\k-1)$ and $r\in(-(2\k-1), 1)$, respectively. Note that since $0\leq \k<1$ we have that $2\k-1<1$ and both integration domains are contained in $(-1,1)$.
Now, we make the changes of variables $\tilde{r}=r-(-1+\kappa)$ in the first integral and $\tilde{r}=r-(1-\kappa)$ in the second one. Thus both integrals run from $-\k$ to $\k$:
\begin{equation}\label{auxH2}
\|\varpi_{\e,\k}''\|^2_{L^2([R_i-\e,R_i+\e])} \leq \frac{2}{\e\k^2}\int_{-\k}^{\k}  \left|\Theta\left(\frac{r}{\kappa}\right)\right|^2 \dr \leq \frac{4}{\ep\k }\|\Theta\|_{L^{\infty}([-1,1])}^2.
\end{equation}

Therefore, combining \eqref{auxH1} and \eqref{auxH2}, there exists $C>0$ such that
\begin{align*}
\|\omega^\sigma_{\ep,\kappa,m}(\cdot,\cdot+\l^\sigma_{\ep,\kappa,m} t)-2A\|_{\dot{H}^{\frac{3-\g}{2}}(\mathsf{P}_{r_1,r_2})}&\lesssim \| \omega^\sigma_{\ep,\kappa,m}\|_{\dot{H}^1 (\mathsf{P}_{r_1,r_2})}^{\frac{1+\g}{2}} \|\omega^\sigma_{\ep,\kappa,m}\|_{\dot{H}^2 (\mathsf{P}_{r_1,r_2})}^{\frac{1-\g}{2}}\\
&\leq C \ep^\frac{1+\gamma}{4}\ep^\frac{-1+\gamma}{4}\kappa^\frac{-1+\gamma}{4}\leq C \ep^\frac{\gamma}{2}\kappa^\frac{-1+\gamma}{4}.
\end{align*}
Moreover, we also have from \eqref{auxDL2_1}-\eqref{auxDL2_2}-\eqref{auxDL2_3} and proceeding as before that
\begin{align*}
\|\omega^\sigma_{\ep,\kappa,m}(\cdot,\cdot+\l^\sigma_{\ep,\kappa,m} t)-2A\|_{L^2(\mathsf{P}_{r_1,r_2})}=\|\omega^\sigma_{\ep,\kappa,m}-2A\|_{L^2(\mathsf{P}_{r_1,r_2})}\leq \|\varpi_{\e,\k}\|_{L^2 (\mathsf{P}_{r_1,r_2})}\lesssim \e.
\end{align*}
Consequently, for any  $\epsilon>0$ and for any $0<s<3/2$, taking $\e$ and $\sigma$ small enough, we find a time periodic solution such that its vorticity satisfies
$\|w^\sigma_{\ep,\kappa, m}(\cdot,\cdot+\l^\sigma_{\ep,\kappa,m} t)-2A\|_{H^s(\mathsf{P}_{r_1,r_2})}<\epsilon.$
\end{proof}

\subsection{Full regularity of the solution}
This section is devote to proving the following result.
\begin{theorem}\label{thm:fullregularity}
The solution $\omega^{\sigma}_{\e,\k,m}$ in Theorem \ref{thmGOAL} is actually $C^{\infty}(\mathsf{P}_{r_1,r_2})$.
\end{theorem}
\begin{proof}
At this point we have proved that $\ff^{\sigma}_{\e,\k,m}\in X(D_\e)$ and consequently the vorticity defined implicitly in Theorem \ref{thmGOAL} satisfies that $\omega^{\sigma}_{\e,\k,m}\in C^{2,\alpha}(\mathsf{P}_{r_1,r_2})$. In order to improve the regularity, we will use functional equation \eqref{functionaleq}, i.e.
\[
\l^{\sigma}_{\e,\k,m}(r+ \ff^{\sigma}_{\e,\k,m}(r,\theta))^2/2 + \bar{\psi}[\ff^{\sigma}_{\e,\k,m}](r,\theta)-A[\l^{\sigma}_{\e,\k,m},\ff^{\sigma}_{\e,\k,m}](r)=0,\qquad (r,\theta)\in D_\e.
\]
with $\bar{\psi}[\ff^{\sigma}_{\e,\k,m}](r,\theta)=\psi(r+\ff^{\sigma}_{\e,\k,m}(r,\theta),\theta)$ and $\psi$ solving \eqref{eq:BVPpsi}.

Let us remove the superscript $\sigma$ and the subscripts $\e,\k$ and $m$ from $\ff^{\sigma}_{\e,\k,m}$ and $\l^{\sigma}_{\e,\k,m}$  to alleviate the notation. That is,
\begin{equation}\label{aux:eqfull}
\l(r + \ff(r,\theta))^2/2 + \bar{\psi}[\ff](r,\theta)-A[\l,\ff](r)=0,\qquad (r,\theta)\in D_\e.
\end{equation}
At this point we know that $\ff \in C^{2,\alpha}(D_\e)$ but we will show that in fact the solution belongs to $C^{3,\a}(D_\e)$.  By iterating this procedure, we can finally conclude that the solution $\ff$ is in $C^\infty(D_\e).$

The main task of this section will be to upgrade the regularity from $C^{2,\a}(D_\e)$ to $C^{3,\a}(D_\e)$. In order to do that, we start recalling that $\ff\in X(D_\e)$ which implies, in particular, that it is mean zero in the angular variable.

So, we can recover $\ff$ from $\p_\theta\ff$ through the expression
\begin{equation}\label{recover}
\ff(r,\theta)=\text{Int}[\pa_\theta \ff](r,\theta):= \int_{0}^\theta \pa_\theta\ff(r,\bar{\theta})\mbox{d}\bar{\theta}-\frac{1}{2\pi}\int_{-\pi}^\pi \left(\int_{0}^{\bar{\theta}}\pa_\theta\ff(r,\tilde{\theta})\mbox{d}\tilde{\theta}\right) \mbox{d}\bar{\theta}, \qquad \theta>0.
\end{equation}
Taking $\p_\theta$ in \eqref{aux:eqfull} we get
\begin{equation}\label{aux:preu1u2}
\l(r+\ff(r,\theta))\p_\theta\ff(r,\theta)+ \p_\theta\bar{\psi}[\ff](r,\theta)=0, \qquad (r,\theta)\in D_\e.
\end{equation}
Since $\bar{\psi}[\ff](r,\theta)=\psi[\ff](r+\ff(r,\theta),\theta)$, we obtain that
\[
\p_\theta\bar{\psi}[\ff](r,\theta)=\p_r\psi[\ff](r+\ff(r,\theta),\theta) \p_\theta\ff(r,\theta)+\p_\theta\psi[\ff](r+\ff(r,\theta),\theta),
\]
and applying it in \eqref{aux:preu1u2} we get
\begin{equation}\label{aux:predivisionfull}
\left[\l(r+\ff(r,\theta))+\p_r\psi[\ff](r+\ff(r,\theta),\theta)\right]\p_\theta\ff(r,\theta)+\p_\theta\psi[\ff](r+\ff(r,\theta),\theta)=0.
\end{equation}

Now we are going to study in detail the terms involved in the above expression, i.e. $\p_\theta\psi[\ff]\circ\Phi[\ff]$ and $\p_r\psi[\ff]\circ\Phi[\ff]$ with $\Phi[\ff](r,\theta)=(r+\ff(r,\theta),\theta)$. Recall that $\psi[\ff]$ solves the system
\begin{equation*}
\left\{
\begin{aligned}
\Delta_B \psi[\ff]&=\omega[\ff],\qquad (r,\theta)\in \mathsf{P}_{r_1,r_2},\\
\psi[\ff]\big|_{r_1}&=0, \\
\psi[\ff]\big|_{r_2}&=\gamma,
\end{aligned}
\right.
\end{equation*}
with
\begin{equation}\label{aux:fullregularityomega}
\omega[\ff]=\begin{cases}
    \varpi\circ\Phi_1^{-1}[\ff], \qquad & (r,\theta)\in \Phi[\ff](D_\e),\\
    0 & (r,\theta)\in \mathsf{P}_{r_1,r_2}\setminus \Phi[\ff](D_\e).
\end{cases}
\end{equation}
Since $\ff\in C^{2+\a}(D_\e)$ with norm small enough, we have that $\Phi[\ff],\Phi^{-1}[\ff]\in C^{2+\a}(D_\e)$  and as an immediate consequence the ``source'' term $\omega[\ff]$ belongs to $C^{2+\a}(\mathsf{P}_{r_1,r_2}).$ Applying the classical theory of elliptic regularity, we get $\psi[\ff]\in C^{4+\a}(\mathsf{P}_{r_1,r_2})$. Then, we have $
\p_r\psi[\ff],\p_\theta\psi[\ff]\in C^{3+\a}(\mathsf{P}_{r_1,r_2}).$

In addition, we can split the first factor of \eqref{aux:predivisionfull} in the following way
\[
\l(r+\ff(r,\theta))+\p_r\psi[\ff](r+\ff(r,\theta),\theta)=\left(\l r + \phi'(r)\right) +\left(\l \ff(r,\theta) + \left[\p_r\psi[\ff](r+\ff(r,\theta),\theta)-\phi'(r)\right]\right)
\]
where $\phi'(r)=\p_r\psi[0](r)$, see \eqref{sys:Phi_0}. Then, remembering that $\phi'\equiv \p_r\psi[0]$ and taking $\sigma$ small enough,  we can prove that 
\begin{equation}\label{aux:sigmasmall}
\l r + \phi'(r)\geq c_{\e,\k,M}>0,\qquad \norm{\l \ff(r,\theta) + \left[\p_r\psi[\ff](r+\ff(r,\theta),\theta)-\phi'(r)\right]}_{L^\infty(D_\e)}\leq \sigma C_{\e,\k,M},
\end{equation}
with
\[
c_{\e,\k,M}-\s C_{\e,\k,M}>\frac{c_{\e,\k,M}}{2}>0,
\]
and we can write \eqref{aux:predivisionfull} as
\begin{equation}\label{aux:divisionfull}
\p_\theta\ff(r,\theta)=-\frac{\p_\theta\psi[\ff](r+\ff(r,\theta),\theta)}{\l(r+\ff(r,\theta))+\p_r\psi[\ff](r+\ff(r,\theta),\theta)}.
\end{equation}
Hence, we have that $\p_\theta\ff\in C^{2+\alpha}(D_\e)$. To get extra regularity on the radial variable let us take one derivatives on $r$ in \eqref{aux:divisionfull} to obtain 
\begin{multline*}
\p_r\p_\theta\ff(r,\theta)=-\frac{\p_r\p_\theta\psi[\ff](r+\ff(r,\theta),\theta)}{\l(r+\ff(r,\theta))+\p_r\psi[\ff](r+\ff(r,\theta),\theta)}(1+\p_r\ff(r,\theta))\\
+\frac{\p_\theta\psi[\ff](r+\ff(r,\theta),\theta)\left(\l+\p_r^2\psi[\ff](r+\ff(r,\theta),\theta)\right)}{\left[\l(r+\ff(r,\theta))+\p_r\psi[\ff](r+\ff(r,\theta),\theta)\right]^2}(1+\p_r\ff(r,\theta)).
\end{multline*}
Writing the above expression in a more compact form, we have that
\begin{equation}\label{compactfull}
\p_r\p_\theta\ff(r,\theta)- \mathsf{K}[\ff] (r,\theta)\p_r\ff(r,\theta)=\mathsf{K}[\ff](r,\theta),
\end{equation}
with
\begin{equation}\label{def:compactpertK}
K[\ff]:=\frac{\left(\l\Phi_1[\ff]+\p_r\psi[\ff]\circ\Phi[\ff]\right)\p_r\p_\theta\psi[\ff]\circ\Phi[\ff]-\p_\theta\psi[\ff]\circ\Phi[\ff]\left(\l+\p_r^2\psi[\ff]\circ\Phi[\ff]\right)}{\left[\l\Phi_1[\ff]+\p_r\psi[\ff]\circ\Phi[\ff]\right]^2}.
\end{equation}

\begin{remark}
Here, it is important to remind that $K[\ff]\in C^{2+\a}(D_\e)$ is an immediate consequence of the fact that 
$\ff\in C^{2+\a}(D_\e)$, which implies that $\Phi[\ff]\in C^{2+\a}(D_\e)$ and $\psi[\ff]\in C^{4+\a}(\mathsf{P}_{r_1,r_2})$.
\end{remark}

For $\sigma$ small enough such that \eqref{aux:sigmasmall} holds, we define for $\gg\in \{\tilde\gg\in C^{1+\a}(D_\e):\int_\T \tilde{\gg}(r,s)\ds=0\}$ the operator
\[
T[\ff]\gg:=\gg - \mathsf{K}[\ff]\text{Int}[\gg].
\]
Here it is important to note that $T[\ff]$ is a compact perturbation of the identity operator since $\mathsf{K}[\ff]$ has norm as small as we want taking $\sigma$ sufficiently small. That is,
\[
\norm{\mathsf{K}[\ff]}_{C^{2+\a}(D_\e)}\leq \sigma C_{\e,\k,m}.
\]
This is a consequence of the terms $\p_r\p_\theta \psi[\ff], \p_\theta \psi[\ff]$ that appear on the numerator  (see \eqref{def:compactpertK}) of $\mathsf{K}[\ff]$, which both satisfy that $\p_r\p_\theta \psi[0]=\p_\theta \psi[0]=0$ since $\psi[0](r,\theta)\equiv \phi(r)$.\\

Thus, for any  $F\in C^{2+\a}(D_\e)$ there exists the inverse operator $T^{-1}[\ff]$, such that, if $\gg\in \{\tilde\gg\in C^{1+\a}(D_\e):\int_\T \tilde{\gg}(r,s)\ds=0\}$ satisfies
\begin{align*}
T[\ff]\gg=F,
\end{align*}
then 
\begin{align*}
\gg=T^{-1}[\ff]F.
\end{align*}
In addition $$T^{-1}[\ff] : C^{2+\a}(D_\e)\to C^{2+\a}(D_\e).$$
Therefore,  $\p_r\p_\theta\ff$ is in $C^{2+\a}(D_\e)$. Thus, we conclude our goal, i.e., that $\ff$ belongs to $C^{3+\a}(D_\ep)$.
Finally, we can iterate this process to show that $\ff$ is in $C^\infty(D_\e).$
\end{proof}

\section{Appendix}\label{s:appendix}

In order to facilitate the presentation of the manuscript, we collect in this section all the technical
results and lemmas used previously.

\subsection{Circulation and boundary conditions}
Here we remind the reader the very well known conservation of circulation.
Since  stream function $\psi$ is defined uniquely up to a constant, we can assume without loss of generality that
\[
\psi(r_1,\theta,t)=0, \qquad  \psi(r_2,\theta,t)=\gamma(t).
\]
Now, we are going to prove that $\gamma(t)\equiv \gamma$ is time-independent and that the following relation holds:
\begin{equation}\label{eq:gammacirculation}
    \gamma=\frac{-1}{2\pi}\int_{[r_1,r_2]\times\T}u_0^{\theta}(r,\theta)\dr\mbox{d}\theta.
\end{equation}

\begin{proof}[Proof of \eqref{eq:gammacirculation}]
Integrating in  angular variable and computing their difference we obtain
\begin{align*}
\gamma(t)=\frac{1}{2\pi}\int_\T \psi(r_2,\theta,t)-\psi(r_1,\theta,t) \mbox{d}\theta &=\frac{1}{2\pi}\int_{[r_1,r_2]\times\T}\p_r \psi(r,\theta,t)\mbox{d}r\mbox{d}\theta\\
&=\frac{-1}{2\pi}\int_{[r_1,r_2]\times\T} u^{\theta}(r,\theta,t) \mbox{d}r\mbox{d}\theta.
\end{align*} 
The equation for the angular component of the velocity is
$(\p_t + u^r\p_r +\frac{1}{r} u^{\theta}\p_\theta)u^\theta  +\frac{1}{r}\p_\theta p=0$, which allow us to write
\begin{align*}
\gamma'(t)&=\frac{1}{2\pi}\int_{[r_1,r_2]\times\T} ((u^r\p_r +\frac{1}{r} u^{\theta}\p_\theta)u^\theta  +\frac{1}{r}\p_\theta p)(r,\theta,t) \mbox{d}r\mbox{d}\theta=\frac{1}{2\pi}\int_{[r_1,r_2]\times\T} u^r\p_r u^\theta   \mbox{d}r\mbox{d}\theta.
\end{align*}
Now, applying integration by parts and using the boundary conditions over $\psi$ we get
\[
\gamma'(t)=-\frac{1}{2\pi}\int_{[r_1,r_2]\times\T} \p_r u^r u^\theta \mbox{d}r\mbox{d}\theta +\frac{1}{2\pi}\underbrace{\int_{[r_1,r_2]\times\T}\p_r\left(u^r u^\theta\right)\mbox{d}r\mbox{d}\theta}_{0}.
\]
In addition, using the diverge-free condition, i.e. $\nabla \cdot u=\p_r u^r +\frac{1}{r}\p_\theta u^\theta=0$, we obtain
\[
\gamma'(t)=-\frac{1}{2\pi}\int_{[r_1,r_2]\times\T} \p_r u^r u^\theta \mbox{d}r\mbox{d}\theta =\frac{1}{2\pi}\int_{[r_1,r_2]\times\T} \p_\theta\left(\frac{1}{r} \frac{(u^\theta)^2}{2}\right) \mbox{d}r\mbox{d}\theta=0.
\]
As an immediate consequence $\gamma(t)=\gamma(0)$ for all $t\geq 0$ and evaluating at initial time  we finally obtain that
\[
\gamma=\frac{1}{2\pi}\int_\T\psi_0(r_2,\theta)-\psi_0(r_1,\theta) \mbox{d}\theta=\frac{1}{2\pi}\int_{[r_1,r_2]\times\T}\p_r\psi_0(r,\theta)\dr \mbox{d}\theta=\frac{-1}{2\pi}\int_{[r_1,r_2]\times\T}u_0^{\theta}(r,\theta)\dr\mbox{d}\theta.
\]
\end{proof}

\subsection{Green's function to solve $\Phi$}\label{app:green}
The main task of this section will be to solve the system
\begin{equation}\label{aux:Poisson}
\left\{
\begin{aligned}
\Delta_B f(r,\theta)&=g(r,\theta),\qquad (\rho,\theta)\in \mathsf{P}_{r_1,r_2}\\
f(r_1,\theta)&=0, \\
f(r_2,\theta)&=0.
\end{aligned}
\right.
\end{equation}

Since $g\in X(D_\e)$ and we look for $f\in X(D_\e)$ by solving the above. We can, without loss of generality, assume that $f$ and $g$ admit the following expansions:
\[
f(r,\theta)=\sum_{n\geq 1} f_n(r)\cos(n\theta), \qquad  g(r,\theta)=\sum_{n\geq 1} g_n(r)\cos(n\theta),
\]
where
\[
f_n(r)=\frac{1}{2\pi}\int_{\T}f(r,\theta)\cos(n\theta)\mbox{d}\theta, \qquad  g_n(r)=\frac{1}{2\pi}\int_{\T}g(r,\theta)\cos(n\theta)\mbox{d}\theta,
\]
denote the $n-$th Fourier coefficient of the functions $f$ and $g$.

Computing the Fourier transform of \eqref{aux:Poisson} in the angular variable, we have to solve for $r\in[r_1,r_2]$ the system
\begin{equation}\label{odefn}
\left\{
\begin{aligned}
f_n''(r)+\frac{1}{r}f_n'(r)-\left(\frac{n}{r}\right)^2 f_n(r)&=g_n(r),\\
f_n(r_i)&=0 \qquad i=1,2.
\end{aligned}
\right.
\end{equation}
Solving the ordinary differential equation of \eqref{odefn} we get
\begin{multline}\label{solutionode}
f_n(r)=K_1 \cosh(n \log r)+ K_2 \sinh(n \log r)\\
+\frac{\sinh(n\log r)}{n}\int_{r_1}^{r} s \cosh(n \log s) g_n(s)\ds-\frac{\cosh(n\log r)}{n}\int_{r_1}^{r} s \sinh(n \log s) g_n(s)\ds,
\end{multline}
with constants values $K_1,K_2$ uniquely determined by the boundary conditions.

Now, in order to simplify the notation, we introduce some auxiliary hyperbolic functions
\begin{equation}\label{defapp:auxhyper}
\sn(r):=\sinh(n\log r), \qquad \cn(r):=\cosh(n\log r).
\end{equation}
Moreover, recalling trigonometric identities for hyperbolic functions we have 
\begin{equation}\label{trigo}
\left\{
\begin{aligned}
\sn(x/y)=\sn(x)\cn(y)-\cn(x)\sn(y),\\
\cn(x/y)=\cn(x) \cn(y)-\sn(x)\sn(y).
\end{aligned}
\right.
\end{equation}
Hence, using \eqref{defapp:auxhyper} and \eqref{trigo}, the expression \eqref{solutionode} can be written as
\begin{equation}\label{fn}
f_n(r)=K_1 \cn(r) + K_2 \sn(r)+\frac{1}{n}\int_{r_1}^{r} s \sn(r/s) g_n(s)\ds.
\end{equation}

Taking into account the boundary conditions of system \eqref{odefn}, we can translate the problem of determine the values $K_1, K_2$ into the following matrix system 
\begin{equation*}
\begin{pmatrix}
\cn(r_1) & \sn(r_1) \\
\cn(r_2) & \sn(r_2) 
\end{pmatrix}
\begin{pmatrix}
K_1 \\
K_2 
\end{pmatrix}
=
\begin{pmatrix}
0 \\
-\frac{1}{n}\int_{r_1}^{r_2} s \sn(r_2/s) g_n(s)\ds
\end{pmatrix}.
\end{equation*}
Since by hypothesis $r_1<r_2$, the matrix of the left is not singular. In fact, using trigonometric identities of hyperbolic functions we get that their determinant is given by
\[
\sn(r_2)\cn(r_1)-\cn(r_2)\sn(r_1)=\sn(r_2/r_1)\neq 0.
\]
Consequently, the unique solution is computed as follows
\begin{align*}
\begin{pmatrix}
K_1 \\
K_2 
\end{pmatrix}
&=\frac{1}{\sn(r_2/r_1)}\begin{pmatrix}
\sn(r_2)  & -\sn(r_1) \\
-\cn(r_2) & \cn(r_1)
\end{pmatrix}
\begin{pmatrix}
0 \\
-\frac{1}{n}\int_{r_1}^{r_2} s \sn(r_2/s) g_n(s)\ds
\end{pmatrix}.
\end{align*}
That is
\begin{equation}\label{def:K1}
K_1:=\frac{-\sn(r_1)}{\sn(r_2/r_1)}\left(-\frac{1}{n}\int_{r_1}^{r_2} s \sn(r_2/s) g_n(s)\ds\right),
\end{equation}
and
\begin{equation}\label{def:K2}
K_2:=\frac{\cn(r_1)}{\sn(r_2/r_1)}\left(-\frac{1}{n}\int_{r_1}^{r_2} s \sn(r_2/s) g_n(s)\ds\right).
\end{equation}
Therefore, returning to \eqref{fn} and using \eqref{def:K1} and \eqref{def:K2} we finally arrive to the following expression
\begin{equation}\label{fnfinal}
f_n(r)=\frac{\sn(r/r_1)}{\sn(r_2/r_1)}\left(-\frac{1}{n}\int_{r_1}^{r_2} s \sn(r_2/s) g_n(s)\ds\right)+\frac{1}{n}\int_{r_1}^{r} s \sn(r/s) g_n(s)\ds.
\end{equation}

\subsubsection{Applying Green's function to solve the Poisson problem} Now, we have all the ingredients to solve \eqref{sys:Phi_1}. That is, 
\begin{equation*}
\left\{
\begin{aligned}
\Delta_B \Phi(r,\theta)&=\varpi'(r)h(r,\theta),\qquad (\rho,\theta)\in \mathsf{P}_{r_1,r_2}\\
\Phi(r_1,\theta)&=0, \\
\Phi(r_2,\theta)&=0.
\end{aligned}
\right.
\end{equation*}
Using the expansion 
\[
\varpi'(r)\hh(r,\theta)=\sum_{n\geq 1} \varpi'(r) h_n(r)\cos(n\theta), \qquad \Phi(r,\theta)=\sum_{n\geq 1} \Phi_n(r)\cos(n\theta),
\]
and applying \eqref{fnfinal} we obtain directly
\begin{equation*}
\Phi_n(r)=\frac{1}{n}\int_{r_1}^{r} s \varpi'(s)\sn(r/s)h_n(s)\ds -\frac{1}{n}\frac{\sn(r/r_1)}{\sn(r_2/r_1)} \int_{r_1}^{r_2}s\varpi'(s)\sn(r_2/s)h_n(s)\ds.  
\end{equation*}

\subsection{Useful lemmas for the asymptotic analysis of $\Upsilon_{\e,\k}^{R_i}(z)$}\label{app:upsilon}
In this section we are going to focus on the asymptotic analysis of $\Upsilon_{\e,\k}^{R_i}(z)$ for $i=1,2$. That is,
\begin{align*}
\Upsilon_{\e,\k}^{R_i}(z)&=(R_i+\e z)\phi'(R_i+\e z)\\
&=C_{\varpi_{\e,\k}} -A((R_i+\e z)^2-r_1^2) -\int_{r_1}^{R_i+\e z} t \varpi_{\e,\k}(t)\mbox{d}t,
\end{align*}
with
\begin{align*}
C_{\varpi_{\e,\k}}&=\log^{-1}\left(\frac{r_2}{r_1}\right)\left[\gamma + A\left(\frac{r_2^2-r_1^2}{2}-r_1^2\log\left(\frac{r_2}{r_1}\right)\right)+\int_{r_1}^{r_2}\frac{1}{s}\left(\int_{r_1}^{s} t \varpi_{\e,\k}(t)\mbox{d}t\right)\mbox{d}s\right].
\end{align*}

The following computations helps us to obtain the asymptotic expansion of each term of the above expression. First of all, it is sufficient to use the definition of the profile $\varpi_{\e,\k}$, recall \eqref{varpiepkappa}, to  obtain directly that
\begin{align}\label{auxint1}
\int_{r_1}^{R_1+\e z} t \varpi_{\e,\k}(t)\mbox{d}t=\e^2\int_{-1}^{z}(R_1+\e t) \varphi_{\k}(-t)\dt,
\end{align}
and
\begin{multline}\label{auxint2}
\int_{r_1}^{R_2+\e z} t \varpi_{\e,\k}(t)\mbox{d}t=\e\frac{R_2^2-R_1^2}{2}\\
+\e^2\left[\int_{-1}^{1}(R_1+\e t)\varphi_\k(-t)\dt +\int_{-1}^{z}(R_2+\e t)\varphi_{\k}(+t)\dt- (R_1+R_2)\right].
\end{multline}

Much more involved is the asymptotic analysis for the constant value $C_{\varpi_{\e,\k}}$ given by
\begin{align*}
C_{\varpi_{\e,\k}}&=\log^{-1}\left(\frac{r_2}{r_1}\right)\left[\gamma + A\left(\frac{r_2^2-r_1^2}{2}-r_1^2\log\left(\frac{r_2}{r_1}\right)\right)+\int_{r_1}^{r_2}\frac{1}{s}\left(\int_{r_1}^{s} t \varpi_{\e,\k}(t)\mbox{d}t\right)\mbox{d}s\right].
\end{align*}
Most of the work will be to analyze the last integral. That is, 
\begin{multline}\label{auxint3}
\int_{r_1}^{r_2}\frac{1}{s}\left(\int_{r_1}^{s} t \varpi_{\e,\k}(t)\mbox{d}t\right)\ds =\int_{r_1}^{r_2}(\log s)'\left(\int_{r_1}^{s} t \varpi_{\e,\k}(t)\mbox{d}t\right)\ds\\
=-\int_{r_1}^{r_2}s\varpi_{\e,\k}(s)\log s\ds +\log(r_2)\left(\int_{r_1}^{r_2} t \varpi_{\e,\k}(t)\mbox{d}t\right).
\end{multline}
Since the second term was already studied previously in \eqref{auxint2}, we only need to focus on the remaining integral term.  That is,
\begin{multline}\label{auxint4}
\int_{r_1}^{r_2}s\varpi_{\e,\k}(s)\log s\ds =    \e \int_{R_1+\e}^{R_2-\e}s\log s\ds\\
+\e^2\left[\int_{-1}^{1}(R_1+\e s)\varphi_{\k}(-s) \log(R_1+\e s)\ds + \int_{-1}^{1}(R_2+\e s)\varphi_{\k}(+s)\log (R_2 +\e s)\ds\right].
\end{multline}
In addition, we have
\begin{align*}
\int_{R_1+\e}^{R_2-\e}s\log s\ds &= \left(\frac{s^2}{2}\log s -\frac{s^2}{4}\right)\bigg|_{s=R_1+\e}^{s=R_2-\e}\\
&=\frac{(R_1+\e)^2-(R_2-\e)^2}{4}+\frac{(R_2-\e)^2}{2}\log(R_2-\e)-\frac{(R_1+\e)^2}{2}\log(R_1+\e),
\end{align*}
which can be written, using Taylor's expansion, as
\[
\int_{R_1+\e}^{R_2-\e}s\log s\ds= f(\e)=f(0)+\int_0^\e f'(\xi)\mbox{d}\xi,
\]
with
\[
f(\xi):=\frac{(R_1+\xi)^2-(R_2-\xi)^2}{4}+\frac{(R_2-\xi)^2}{2}\log(R_2-\xi)-\frac{(R_1+\xi)^2}{2}\log(R_1+\xi),
\]
and derivative
\[
f'(\xi)=-(R_2-\xi)\log(R_2-\xi)-(R_1+\xi)\log(R_1+\xi).
\]

Straightforward calculations imply that \eqref{auxint4} can be written simply as
\begin{multline}\label{auxint5}
\int_{r_1}^{r_2}s\varpi_{\e,\k}(s)\log s\ds=\e\left(\frac{R_1^2-R_2^2}{4}+\frac{R_2^2}{2}\log(R_2)-\frac{R_1^2}{2}\log(R_1)\right)\\
+\e^2 \left[ \int_{-1}^{1}(R_1+\e s)\varphi_{\k}(-s) \log(R_1+\e s)\ds +  \int_{-1}^{1}(R_2+\e s)\varphi_{\k}(+s)\log (R_2 +\e s)\ds\right]\\
-\e\int_0^\e (R_2-\xi)\log(R_2-\xi)+(R_1+\xi)\log(R_1+\xi)\mbox{d}\xi.
\end{multline}

At this point we have already analyzed all the terms involved in \eqref{auxint3}.
Combining \eqref{auxint2} and \eqref{auxint5} appropriately, we obtain that
\begin{multline}\label{auxint6}
\int_{r_1}^{r_2}\frac{1}{s}\left(\int_{r_1}^{s} t \varpi_{\e,\k}(t)\mbox{d}t\right)\ds=\e\left[\frac{R_2^2-R_1^2}{2}\left(\log(r_2) + \frac{1}{2}\right)+ \frac{R_1^2}{2}\log(R_1)-\frac{R_2^2}{2}\log(R_2) \right] \\
+\e^2\, \mathsf{Remainder}_{\e,\k}.
\end{multline}
with
\begin{multline}\label{auxremainder}
\mathsf{Remainder}_{\e,\k}=\log(r_2)\left[\int_{-1}^{1}(R_1+\e t)\varphi_\k(-t)\dt +\int_{-1}^{z}(R_2+\e t)\varphi_{\k}(+t)\dt- (R_1+R_2)\right]\\
-\left[\int_{-1}^{1}(R_1+\e s)\varphi_{\k}(-s) \log(R_1+\e s)\ds + \int_{-1}^{1}(R_2+\e s)\varphi_{\k}(+s)\log (R_2 +\e s)\ds\right]\\
+\frac{1}{\e}\int_0^\e (R_2-\xi)\log(R_2-\xi)+(R_1+\xi)\log(R_1+\xi)\mbox{d}\xi.
\end{multline}

As an immediate consequence of \eqref{auxint6} and \eqref{auxremainder} we obtain the expansion
\begin{multline*}
C_{\varpi,\k} \log\left(\frac{r_2}{r_1}\right)= \left[\gamma + A\left(\frac{r_2^2-r_1^2}{2}-r_1^2\log\left(\frac{r_2}{r_1}\right)\right)\right]\\
+\e\left[\frac{R_2^2-R_1^2}{2}\left(\log(r_2) + \frac{1}{2}\right)+ \frac{R_1^2}{2}\log(R_1)-\frac{R_2^2}{2}\log(R_2) \right] 
+\e^2\, \mathsf{Remainder}_{\e,\k}.
\end{multline*}

\begin{remark}
    Note that $\mathsf{Remainder}_{\e,\k}=O(1)$ in terms of the smallness parameters $\e$ and $\k$.
\end{remark}

Now, we have all the ingredients to decompose $\Upsilon_{\e,\k}^{R_i}(z)$ as
\begin{equation*}
\Upsilon_{\e,\k}^{R_i}(z)=\tilde{\Upsilon}_0^{R_i}+\e \tilde{\Upsilon}_1^{R_i}(z)+\e^2 \tilde{\Upsilon}_{\e,\k}^{R_i}(z),
\end{equation*}
with
\begin{align*}
\tilde{\Upsilon}_0^{R_1}&:=\log^{-1}\left(\frac{r_2}{r_1}\right)\left[\gamma + A\left(\frac{r_2^2-r_1^2}{2}-r_1^2\log\left(\frac{r_2}{r_1}\right)\right)\right]-A(R_1^2 - r_1^2),\\
\tilde{\Upsilon}_0^{R_2}&:=\log^{-1}\left(\frac{r_2}{r_1}\right)\left[\gamma + A\left(\frac{r_2^2-r_1^2}{2}-r_1^2\log\left(\frac{r_2}{r_1}\right)\right)\right]-A(R_2^2 - r_1^2),
\end{align*}
\begin{align*}
\tilde{\Upsilon}_1^{R_1}(z)&:=\log^{-1}\left(\frac{r_2}{r_1}\right)\left[\frac{R_2^2-R_1^2}{2}\left(\log(r_2) + \frac{1}{2}\right)+ \frac{R_1^2}{2}\log(R_1)-\frac{R_2^2}{2}\log(R_2) \right] - 2AR_1 z,\\
\tilde{\Upsilon}_1^{R_2}(z)&:=\log^{-1}\left(\frac{r_2}{r_1}\right)\left[\frac{R_2^2-R_1^2}{2}\left(\log(r_2) + \frac{1}{2}\right)+ \frac{R_1^2}{2}\log(R_1)-\frac{R_2^2}{2}\log(R_2) \right] - 2AR_2 z -\frac{R_2^2-R_1^2}{2},
\end{align*}
and
\begin{align*}
\tilde{\Upsilon}_{\e,\k}^{R_1}(z)&:=\frac{\mathsf{Remainder}_{\e,\k}}{\log(r_2/r_1)}-A z^2   -\int_{-1}^{z}(R_1+\e t) \varphi_{\k}(-t)\dt, \\
\tilde{\Upsilon}_{\e,\k}^{R_2}(z)&:=\frac{\mathsf{Remainder}_{\e,\k}}{\log(r_2/r_1)} - A z^2   - \left[\int_{-1}^{1}(R_1+\e t)\varphi_\k(-t)\dt +\int_{-1}^{z}(R_2+\e t)\varphi_{\k}(+t)\dt- (R_1+R_2)\right].
\end{align*}

\begin{remark}
    Note that $\tilde{\Upsilon}_{\e,\k}^{R_i}(z)=O(1)$ for $i=1,2$ in terms of the smallness parameters $\e$ and $\k$.
\end{remark}

Finally, we can make the above expression more manageable by simplifying the above expressions in the following way 
\begin{equation}\label{auxUpsilon0}
\tilde{\Upsilon}_0^{R_i}=\log^{-1}\left(\frac{r_2}{r_1}\right)\left[\gamma + A\left(\frac{r_2^2-r_1^2}{2}-r_1^2\log\left(\frac{r_2}{r_1}\right)\right)\right]-A(R_i^2 - r_1^2), 
\end{equation}
\begin{align}
\tilde{\Upsilon}_1^{R_1}(z)&=\log^{-1}\left(\frac{r_2}{r_1}\right)\left[\frac{R_2^2-R_1^2}{2}\left(\log(r_2) + \frac{1}{2}\right)+ \frac{R_1^2}{2}\log(R_1)-\frac{R_2^2}{2}\log(R_2) \right] - 2AR_1 z, \label{auxUpsilon11}\\
\tilde{\Upsilon}_1^{R_2}(z)&=\tilde{\Upsilon}_1^{R_1}(0)- 2AR_2 z -\frac{R_2^2-R_1^2}{2}, \label{auxUpsilon12}
\end{align}
and
\begin{align}
\tilde{\Upsilon}_{\e,\k}^{R_1}(z)&=\frac{\mathsf{Remainder}_{\e,\k}}{\log(r_2/r_1)} - A z^2   -\int_{-1}^{z}(R_1+\e t) \varphi_{\k}(-t)\dt, \label{auxUpsilon21}\\
\tilde{\Upsilon}_{\e,\k}^{R_2}(z)&=\tilde{\Upsilon}_{\e,\k}^{R_1}(0)+A(1-z^2)- \left[\int_{-1}^{z}(R_2+\e t)\varphi_{\k}(+t)\dt- (R_1+R_2)\right], \label{auxUpsilon22}
\end{align}
with $\mathsf{Remainder}_{\e,\k}$ given explicitly by \eqref{auxremainder}.

\subsection{Functional derivatives of auxiliary terms}
In this section we will discuss the technical details to conclude that the following bounds hold: 
\begin{equation}\label{DFT1Q1}
\left|\frac{1}{\e}\int_0^\e \left(\p_\d\mathcal{T}_{m,\d}^1\right)[\bb_0](z)\mbox{d}\delta\right|, \left|\frac{1}{\e}\int_0^\e (\p_\d\mathcal{Q}_{m,\d}^1)[\bb_0](z)\mbox{d}\delta\right| \leq C \|\bb_0\|_{L^2},
\end{equation}
\begin{equation}\label{DFT2}
\left|\frac{1}{\e}\int_0^\e  \left(\p_\d \cT_{m,\d}^2 \right) [\bb_1^\e;\aa_1](z) \mbox{d}\delta\right|\leq C\left(\|\aa_1\|_{L^2} +\|\bb_1^\e\|_{L^2} \right),
\end{equation}
and
\begin{equation}\label{DFQ2}
\left|\frac{1}{\e}\int_0^\e \left(\p_\d \mathcal{Q}_{m,\d}^2 \right)[\bb_1^\e,\l_2^\e;\aa_1,\bb_0](z) \mbox{d}\d\right|\leq C\left(|\l_2^\e| \|\bb_0\|_{L^2}+\|\aa_1\|_{L^2}+\|\bb_1^\e\|_{L^2}\right),
\end{equation}
where the constant $C$ does not depend on neither $\e, \k$ or $m$.

Proceeding similarly, we can also prove (see \eqref{eq:auxR} and \eqref{eq:auxW}) that 
\begin{equation}\label{DFRW}
\left|\frac{1}{\e}\int_0^\e (\p_\d R_{\d}[\vv])(z)\mbox{d}\d \right|, \left|\frac{1}{\e}\int_0^\e (\p_\d W_{\d}[\vv])(z)\mbox{d}\d \right| \leq C \|\vv\|_{L^2},
\end{equation}
and (see \eqref{eq:auxS} and \eqref{eq:auxV})
\begin{equation}\label{DFSV}
\left|\frac{1}{\e}\int_0^\e (\p_\d S_{\d}[\vv])(z)\mbox{d}\d \right|, \left|\frac{1}{\e}\int_0^\e (\p_\d V_{\d}[\vv])(z)\mbox{d}\d \right| \leq C \|\vv\|_{L^2},
\end{equation}

\begin{proof}[Proof of \eqref{DFT1Q1}]
We start writing \eqref{T1e} and  \eqref{Q1e} in a more convenient way as
\begin{equation*}
\cT_{m,\d}^1[\aa_1,\bb_0,\l_0](z)=\a_0^{R_1}[\l_0] \aa_1(z) -\frac{1}{\sm(r_2/r_1)}\frac{1}{m}\int_{-1}^{+1} \varphi_\k' (s)\Phi_{\cT_1}(\d, z, s)\bb_0(s) \ds,
\end{equation*}
and
\begin{multline*}
\mathcal{Q}_{m,\d}^1[\bb_0,\bb_1^\e,\l_0,\l_1](z)=\a_0^{R_2}[\l_0] \bb_1^\e(z)+\a_1^{R_2}[\l_0,\l_1]\bb_0(z)\\
-\frac{1}{\sm(r_2/r_1)}\frac{1}{m}\int_{-1}^{+1} \varphi_\k' (s)\Phi_{\cQ_1}^1(\d, z ,  s)\bb_0(s)\ds
+\frac{1}{m}\int_{-1}^{ z} \varphi_\k' (s) \Phi_{\cQ_1}^2(\d, z ,  s) \bb_0(s)\ds,
\end{multline*}
with
\[
\Phi_{\cT_1}(\d, z ,  s):=(R_1+\d z)(R_2+\d s)\sm\left(\frac{R_1+\d z}{r_1}\right)\sm\left(\frac{r_2}{R_2+\d s}\right).
\]
and
\begin{align*}
\Phi_{\cQ_1}^1(\d, z ,  s)&:=(R_2+\d z)(R_2+\d s)\sm\left(\frac{R_2+\d z}{r_1}\right)\sm\left(\frac{r_2}{R_2+\d s}\right),\\
\Phi_{\cQ_1}^2(\d, z ,  s)&:=(R_2+\d z)(R_2+\d s)\sm\left(\frac{R_2+\d z}{R_2+\d s}\right).
\end{align*}
Note that computing their derivative we have
\begin{multline*}
(\p_\d\Phi_{\cT_1})(\d,z, s)=z(R_2+\d s)\sm\left(\frac{r_2}{R_2+\d s}\right)\left[\sm\left(\frac{R_1+\d z}{r_1}\right)+  m\, \cm\left(\frac{R_1+\d z}{r_1}\right) \right]\\
+s(R_1+\d z)\sm\left(\frac{R_1+\d z}{r_1}\right)\left[\sm\left(\frac{r_2}{R_2+\d s}\right)-m\, \cm\left(\frac{r_2}{R_2+\d s}\right)\right]
\end{multline*}
and
\begin{multline*}
(\p_\d \Phi_{\cQ_1}^1)(\d, z ,  s)=z(R_2+\d s)\sm\left(\frac{r_2}{R_2+\d s}\right)\left[\sm\left(\frac{R_2+\d z}{r_1}\right)+m\,\cm\left(\frac{R_2+\d z}{r_1}\right)\right]\\
+s(R_2+\d z)\sm\left(\frac{R_2+\d z}{r_1}\right)\left[\sm\left(\frac{r_2}{R_2+\d s}\right)-m\,\cm\left(\frac{r_2}{R_2+\d s}\right)\right],
\end{multline*}
\begin{equation*}
(\p_\d\Phi_{\cQ_1}^2)(\d, z ,  s)= \left[ z(R_2+\d s)+s (R_2+\d z)\right]\sm\left(\frac{R_2+\d z}{R_2+\d s}\right)+\cm\left(\frac{R_2+\d z}{R_2+\d s}\right)m R_2(z-s).
\end{equation*}
Therefore, 
\[
(\p_\d \cT_{m,\d}^1)[\bb_0](z)=-\frac{1}{\sm(r_2/r_1)}\frac{1}{m}\int_{-1}^{+1} \varphi_\k' (s)(\p_\d\Phi_{\cT_1})(\d,z, s)\bb_0(s) \ds,
\]
and
\begin{multline*}
(\p_\d\mathcal{Q}_{m,\d}^1)[\bb_0](z)=-\frac{1}{\sm(r_2/r_1)}\frac{1}{m}\int_{-1}^{+1} \varphi_\k' (s)(\p_\d\Phi_{\cQ_1}^1)(\d, z ,  s)\bb_0(s)\ds\\
+\frac{1}{m}\int_{-1}^{ z} \varphi_\k' (s) (\p_\d\Phi_{\cQ_1}^2)(\d, z ,  s) \bb_0(s)\ds.
\end{multline*}
Consequently, we have
\[
\|(\p_\d \cT_{m,\d}^1)[\bb_0]\|_{L^\infty}, \|(\p_\d\mathcal{Q}_{m,\d}^1)[\bb_0]\|_{L^\infty}\leq C \|\bb_0\|_{L^2},
\]
where the constant $C$ does not depend on neither $\e, \k$ or $m$, and $\|\cdot\|_{L^p}=\|\cdot\|_{L^p([-1,1])}$.
\end{proof}

\begin{proof}[Proof of \eqref{DFT2}]
We start writing \eqref{T2e} in a more convenient way as
\begin{multline*}
\cT_{m,\d}^2[\aa_1,\aa_2^\e,\bb_1^\e,\l_0,\l_1](z)=\a_1^{R_1}[\l_0,\l_1] \aa_1(z)+\a_0^{R_1}[\l_0]\aa_2^{\e}(z)\\
+\frac{1}{\sm(r_2/r_1)}\frac{1}{m}\int_{-1}^{+1} \varphi_\k' (-s)\Phi_{\cT_2}^1(\d,z,s) \aa_1(s)\ds\\
-\frac{1}{\sm(r_2/r_1)}\frac{1}{m}\int_{-1}^{+1} \varphi_\k' (s)\Phi_{\cT_2}^2(\d,z,s) \bb^\ep_1(s) \ds\\
-\frac{1}{m}\int_{-1}^{z} \varphi_\k' (-s)\Phi_{\cT_2}^3(\d,z,s)\aa_1(s)\ds,
\end{multline*}
with
\[
\Phi_{\cT_2}^1(\d,z,s):=(R_1+\d z)(R_1+\d s)\sm\left(\frac{R_1+\d z}{r_1}\right)\sm\left(\frac{r_2}{R_1+\d s}\right),
\]
\[
\Phi_{\cT_2}^2(\d,z,s):=(R_1+\d z)(R_2+\d s)\sm\left(\frac{R_1+\d z}{r_1}\right)\sm\left(\frac{r_2}{R_2+\d s}\right),
\]
\[
\Phi_{\cT_2}^3(\d,z,s):=(R_1+\d z)(R_1+\d s)\sm\left(\frac{R_1+\d z}{R_1+\d s}\right).
\]
Note that
\begin{multline*}
(\p_\d\Phi_{\cT_2}^1)(\d,z,s)= z (R_1+\d s)\sm\left(\frac{r_2}{R_1+\d s}\right)\left[\sm\left(\frac{R_1+\d z}{r_1}\right)+m\,\cm\left(\frac{R_1+\d z}{r_1}\right)\right]\\
+s (R_1+\d z)\sm\left(\frac{R_1+\d z}{r_1}\right)\left[\sm\left(\frac{r_2}{R_1+\d s}\right)-m\, \cm\left(\frac{r_2}{R_1+\d s}\right)\right],
\end{multline*}
\begin{multline*}
(\p_\d\Phi_{\cT_2}^2)(\d,z,s)= z (R_2+\d s)\sm\left(\frac{r_2}{R_2+\d s}\right)\left[\sm\left(\frac{R_1+\d z}{r_1}\right)+m\,\cm\left(\frac{R_1+\d z}{r_1}\right)\right]\\
+s(R_1+\d z)\sm\left(\frac{R_1+\d z}{r_1}\right)\left[\sm\left(\frac{r_2}{R_2+\d s}\right)-m\,\cm\left(\frac{r_2}{R_2+\d s}\right)\right],
\end{multline*}
and
\begin{equation*}
(\p_\d\Phi_{\cT_2}^3)(\d,z,s)= \left[z (R_1+\d s)+s (R_1+\d z)\right]\sm\left(\frac{R_1+\d z}{R_1+\d s}\right)+\cm\left(\frac{R_1+\d z}{R_1+\d s}\right) m R_1(z-s).
\end{equation*}

Hence
\begin{multline*}
(\p_\d\cT_{m,\d}^2)[\aa_1,\bb_1^\e](z)=\frac{1}{\sm(r_2/r_1)}\frac{1}{m}\int_{-1}^{+1} \varphi_\k' (-s)(\p_\d\Phi_{\cT_2}^1)(\d,z,s) \aa_1(s)\ds\\
-\frac{1}{\sm(r_2/r_1)}\frac{1}{m}\int_{-1}^{+1} \varphi_\k' (s)(\p_\d \Phi_{\cT_2}^2)(\d,z,s) \bb^\ep_1(s) \ds\\
-\frac{1}{m}\int_{-1}^{z} \varphi_\k' (-s) (\p_\d\Phi_{\cT_2}^3)(\d,z,s)\aa_1(s)\ds,
\end{multline*}
and consequently
\[
\|(\p_\d\cT_{m,\d}^2)[\aa_1,\bb_1^\e]\|_{L^\infty}\leq C\left(\|\aa_1\|_{L^2} +\|\bb_1^\e\|_{L^2} \right).
\]
\end{proof}

\begin{proof}[Proof of \eqref{DFQ2}]
We start writing \eqref{Q2e} in a more convenient way as
\begin{multline*}
\mathcal{Q}_{m,\d}^2[\aa_1,\aa_2^\e,\bb_0,\bb_1^\e,\l_0,\l_1,\l_2^\e](z)=\a_1^{R_2}[\l_0,\l_1]\bb_1^\e(z)+ \a_{2,\d}^{R_2}[\l_0,\l_1,\l_2^\e]\mathsf{b}_0(z)\\
+\frac{1}{\sm(r_2/r_1)}\frac{1}{m}\int_{-1}^{+1} \varphi_\k' (-s)  \Phi_{\cQ_2}^1(\d, z ,  s) \aa_1(s)\ds -\frac{1}{\sm(r_2/r_1)}\frac{1}{m}\int_{-1}^{+1} \varphi_\k' (s)\Phi_{\cQ_2}^2(\d,z ,s) \bb^\ep_1(s)\ds\\
-\frac{1}{m}\int_{-1}^{+1} \varphi_\k' (-s) \Phi_{\cQ_2}^3(\d,z,s) \aa_1(s)\ds+\frac{1}{m}\int_{-1}^{ z} \varphi_\k' (s)\Phi_{\cQ_2}^4(\d,z,s)  \bb^\ep_1(s)\ds,
\end{multline*}
with
\begin{align*}
\Phi_{\cQ_2}^1(\d,z,s)&:=(R_2+\d z)(R_1+\d s)\sm\left(\frac{R_2+\d z}{r_1}\right)\sm\left(\frac{r_2}{R_1+\d s}\right),\\
\Phi_{\cQ_2}^2(\d,z,s)&:=(R_2+\d z)(R_2+\d s)\sm\left(\frac{R_2+\d z}{r_1}\right)\sm\left(\frac{r_2}{R_2+\d s}\right),\\
\Phi_{\cQ_2}^3(\d,z,s)&:=(R_2+\d z)(R_1+\d s)\sm\left(\frac{R_2+\d z}{R_1+\d s}\right),\\
\Phi_{\cQ_2}^4(\d,z,s)&:=(R_2+\d z)(R_2+\d s)\sm\left(\frac{R_2+\d z}{R_2+\d s}\right).
\end{align*}
Recall also that $\a_{2,\d}^{R_2}[\l_0,\l_1,\l_2^\e]$ is defined by \eqref{def:alpha2new} as
\[
\a_{2,\d}^{R_2}[\l_0,\l_1,\l_2^\e](z)=\l_0 z^2 +\l_1 2 z R_2 + \l_2^\e R_2^2+\l_2^\e 2z R_2\d+\l_2^\e z^2\d^2.
\]
Note that computing their derivative we have
\begin{multline*}
(\p_\d \Phi_{\cQ_2}^1 )(\d,z,s)= z(R_1+\d s)\sm\left(\frac{r_2}{R_1+\d s}\right)\left[\sm\left(\frac{R_2+\d z}{r_1}\right)+m\,\cm\left(\frac{R_2+\d z}{r_1}\right)\right]\\
+s (R_2+\d z)\sm\left(\frac{R_2+\d z}{r_1}\right)\left[\sm\left(\frac{r_2}{R_1+\d s}\right)-m\,\cm\left(\frac{r_2}{R_1+\d s}\right)\right],
\end{multline*}
\begin{multline*}
(\p_\d \Phi_{\cQ_2}^2 )(\d,z,s)= z(R_2+\d s)\sm\left(\frac{r_2}{R_2+\d s}\right)\left[\sm\left(\frac{R_2+\d z}{r_1}\right)+m\,\cm\left(\frac{R_2+\d z}{r_1}\right)\right]\\
+s (R_2+\d z)\sm\left(\frac{R_2+\d z}{r_1}\right)\left[\sm\left(\frac{r_2}{R_2+\d s}\right)-m\,\sm\left(\frac{r_2}{R_2+\d s}\right)\right],
\end{multline*}

\begin{equation*}
(\p_\d \Phi_{\cQ_2}^3 )(\d,z,s)= \left[z (R_1+\d s)+s (R_2+\d z)\right]\sm\left(\frac{R_2+\d z}{R_1+\d s}\right)
+\cm\left(\frac{R_2+\d z}{R_1+\d s}\right) m (z R_1- s R_2),
\end{equation*}

\begin{equation*}
(\p_\d \Phi_{\cQ_2}^4 )(\d,z,s)= \left[z(R_2+\d s)+s (R_2+\d z)\right]\sm\left(\frac{R_2+\d z}{R_2+\d s}\right)
+\cm\left(\frac{R_2+\d z}{R_2+\d s}\right) m R_2(z-s),
\end{equation*}
and also
\[
(\p_\d\a_{2,\d}^{R_2})[\l_2^\e](z)= 2z \l_2^\e (R_2+\d z).
\]
Hence 
\begin{multline*}
(\p_\d \mathcal{Q}_{m,\d}^2)[\aa_1,\bb_0,\bb_1^\e,\l_2^\e](z)= (\p_\d\a_{2,\d}^{R_2})[\l_2^\e]\mathsf{b}_0(z)\\
+\frac{1}{\sm(r_2/r_1)}\frac{1}{m}\int_{-1}^{+1} \varphi_\k' (-s)  (\p_\d \Phi_{\cQ_2}^1)(\d, z ,  s) \aa_1(s)\ds -\frac{1}{\sm(r_2/r_1)}\frac{1}{m}\int_{-1}^{+1} \varphi_\k' (s) (\p_\d \Phi_{\cQ_2}^2)(\d,z ,s) \bb^\ep_1(s)\ds\\
-\frac{1}{m}\int_{-1}^{+1} \varphi_\k' (-s) (\p_\d \Phi_{\cQ_2}^3)(\d,z,s) \aa_1(s)\ds+\frac{1}{m}\int_{-1}^{ z} \varphi_\k' (s) (\p_\d \Phi_{\cQ_2}^4)(\d,z,s)  \bb^\ep_1(s)\ds,
\end{multline*}
and consequently
\[
\|(\p_\d \mathcal{Q}_{m,\d}^2)[\aa_1,\bb_0,\bb_1^\e,\l_2^\e]\|_{L^\infty}\leq C\left(|\l_2^\e| \|\bb_0\|_{L^2}+\|\aa_1\|_{L^2}+\|\bb_1^\e\|_{L^2}\right).
\]
\end{proof}

\textbf{Acknowledgments:}  A.C. acknowledge financial support from the Severo Ochoa Programme for Centres of Excellence Grant CEX2019-000904-S funded by MCIN/AEI/10.13039/501100011033, Grant PID2020-114703GB-I00 funded by MCIN/AEI/10.13039/501100011033, Grants RED2022-134784-T and RED2018-102650-T funded by MCIN/AEI/10.13039/501100011033
and by a 2023 Leonardo Grant for Researchers and Cultural Creators, BBVA Foundation. The BBVA Foundation accepts no responsibility for the opinions, statements, and contents included in the project and/or the results thereof, which are entirely the responsibility of the authors.

D.L. is  supported by RYC2021-030970-I funded by MCIN/AEI/10.13039/501100011033 and the NextGenerationEU. D.L also acknowledge financial support from Grant PID2020-114703GB-I00 and Grant PID2022-141187NB-I00 funded by MCIN/AEI/10.13039/501100011033.

\end{document}